\documentclass[letter]{amsart}

\usepackage{amssymb}
\usepackage{amsfonts}
\usepackage{amsmath}
\usepackage{amsthm}
\usepackage{mathtools}
\usepackage{graphicx}
\usepackage{mathrsfs}
\usepackage{dsfont}
\usepackage{amscd}
\usepackage{multirow}
\usepackage[all]{xy}
\normalfont
\usepackage[T1]{fontenc}
\usepackage{calligra}
\usepackage{verbatim}
\usepackage{fourier}
\usepackage[usenames,dvipsnames]{color}
\usepackage[colorlinks=true,linkcolor=Blue,citecolor=Violet]{hyperref}
\usepackage{enumerate}
\usepackage{slashed}
\usepackage{nicematrix}
\usepackage{microtype}
\usepackage{faktor}
\usepackage{tikz-cd}
\usepackage{tensor}

\DeclareMathOperator{\id}{Id}

\newcommand{\N}{{\mathds{N}}}
\newcommand{\Z}{{\mathds{Z}}}

\newcommand{\R}{{\mathds{R}}}
\newcommand{\C}{{\mathds{C}}}
\newcommand{\T}{{\mathds{T}}}

\newcommand{\D}{{\mathfrak{D}}}
\newcommand{\A}{{\mathfrak{A}}}
\newcommand{\B}{{\mathfrak{B}}}

\newcommand{\Lip}[1][L]{{\mathsf{#1}}}
\newcommand{\Hilbert}[1][H]{{\mathscr{#1}}}

\newcommand{\dpropinquity}[1]{{\mathsf{\Lambda}^\ast_{#1}}}

\newcommand{\dmetpropinquity}[1]{{\mathsf{\Lambda}^{\ast\mathsf{met}}_{#1}}}

\newcommand{\spectralpropinquity}[1]{{\mathsf{\Lambda}^{\mathsf{spec}}_{#1}}}
\newcommand{\oppropinquity}[1]{{\mathsf{\Lambda}^{\mathsf{op}}_{#1}}}

\newcommand{\Kantorovich}[1]{{\mathsf{mk}_{#1}}}

\newcommand{\KantorovichJ}[1]{{\mathsf{MK}_{#1}}}

\newcommand{\Haus}[1]{{\mathsf{Haus}\!\left[{#1}\right]\,}}

\newcommand{\StateSpace}{{\mathscr{S}}}

\newcommand{\MongeKant}{{Mon\-ge-Kan\-to\-ro\-vich metric}}

\newcommand{\gMVB}{metrical C*-correspondence}

\newcommand{\mcc}[3]{{\mathrm{metCor}\left({#1},{#2},{#3}\right)}}

\newcommand{\qcms}{quantum compact metric space}

\newcommand{\sa}[1]{{\mathfrak{sa}\left({#1}\right)}}

\newcommand{\inner}[3]{{\left<{#1},{#2}\right>_{#3}}}

\newcommand{\dom}[1]{{\operatorname*{dom}\left({#1}\right)}}
\newcommand{\codom}[1]{{\operatorname*{codom}\left({#1}\right)}}
\newcommand{\diam}[2]{{\mathrm{diam}\left({#1},{#2}\right)}}
\newcommand{\qdiam}[2]{{\mathrm{qdiam}\left({#1},{#2}\right)}}

\newcommand{\norm}[2]{\left\|{#1}\right\|_{#2}}

\newcommand{\Jordan}[2]{{{#1}\circ{#2}}}
\newcommand{\Lie}[2]{{\left\{{#1},{#2}\right\}}}

\newcommand{\CDN}[1][DN]{{\mathsf{#1}}}

\newcommand{\worknote}[1]{}
\newcommand{\opnorm}[3]{{\left|\mkern-1.5mu\left|\mkern-1.5mu\left| {#1} \right|\mkern-1.5mu\right|\mkern-1.5mu\right|_{#3}^{#2}}}

\newcommand{\tunnelextent}[1]{{\chi\left({#1}\right)}}

\newcommand{\alg}[1]{{\mathfrak{#1}}}

\newcommand{\module}[1]{{\mathscr{#1}}}

\newcommand{\spectrum}[1]{\mathrm{Sp}\left({#1}\right)}

\newcommand{\ModStateSpace}{{\widetilde{\mathscr{S}}}}

\newcommand{\tunnelsep}[2]{{\mathsf{sep}\left({#2}\middle\vert{#1}\right)}}
\newcommand{\tunneldispersion}[2]{{\mathsf{dis}\left({#2}\middle\vert {#1} \right)}}

\renewcommand{\geq}{\geqslant}
\renewcommand{\leq}{\leqslant}

\newcommand{\Dirac}[1][D]{{\slashed{#1}}}

\theoremstyle{plain}
\newtheorem{theorem}{Theorem}[section]

\newtheorem{corollary}[theorem]{Corollary}

\newtheorem{lemma}[theorem]{Lemma}
\newtheorem{proposition}[theorem]{Proposition}

\newtheorem{theorem-definition}[theorem]{Theorem-Definition}

\theoremstyle{definition}
\newtheorem{definition}[theorem]{Definition}

\newtheorem{notation}[theorem]{Notation}

\newtheorem{convention}[theorem]{Convention}

\theoremstyle{remark}

\newtheorem{remark}[theorem]{Remark}

\numberwithin{equation}{section}

\usepackage[T1]{fontenc}
\newtheorem*{theorem*}{Theorem}

\begin{document}
	
	\title[]{Collapse in Noncommutative Geometry and Spectral Continuity}
	
	\author{Carla Farsi}
	\email{carla.farsi@colorado.edu}
	\address{Department of Mathematics \\ University of Colorado \\ Boulder CO 80309-0395}
	
	\author{Fr\'{e}d\'{e}ric Latr\'{e}moli\`{e}re}
	\email{frederic@math.du.edu}
	\urladdr{http://www.math.du.edu/\symbol{126}frederic}
	\address{Department of Mathematics \\ University of Denver \\ Denver CO 80208}

	\date{\today}
	
	\begin{abstract} In the classical realm of Riemannian geometry, informally, if two manifolds are close in the Gromov-Hausdorff distance, and belong to a class of closed manifolds with bounded curvature and diameter, then the spectra of their Laplacian or Dirac operators are also close under many scenarii. Of particular interest is the case where a sequence of manifolds converges for the Gromov-Hausdorff distance to a manifold of lower dimension,  i.e., {\it collapses} to the limit. The   question then arises of the continuity, in some sense, of the geometrically relevant operators and their spectra. This has been investigated at length in many classical papers. 	In the more general context of noncommutative geometry  the notion of convergence with respect to the variation of certain metric structures   plays a very important  role, as it allows the approximation of spaces and algebras  over  Hilbert  spaces by simpler ones (e.g. finite--dimensional and/or matricial).
	So if  two quantum compact metric spaces are close in the metric sense,  then how similar are they, as noncommutative spaces?  In this paper, we initiate the study of the continuity of spectra and other properties of metric spectral triples on noncommutative $G$-bundles, for $G$ compact Lie,  under collapse in the \lq \lq vertical" direction. As a first step in this study, we work with the {\it spectral propinquity}, an analogue of the Gromov-Hausdorff distance for metric spectral triples introduced by the second author. The spectral propinquity  is a form of metric for differential structures. Inspired by results from collapse in Riemannian geometry, we  study spectral triples which decompose, in some sense, in a vertical and a horizontal direction; they can be shown to be special Kasparov products.  We perturb the (metric in the) vertical component by a parameter $\varepsilon $, and then we take the limit for  $\varepsilon $ approaching zero, thus  obtaining  a metric spectral triples convergence results.  As a consequence, by  the work of the second author, we also derive  continuity results for the spectra of the Dirac operators of these spectral triples. 	Examples of applications of our work include collapse of products of spectral triples with one Abelian factor, spectral triples associated to principal $U(1)$-bundles over closed Riemannian spin  manifolds, and spectral triples associated to noncommutative principal $G$-bundles as in the work of Schwieger and Wagner;  include $G$-crossed products.
	\end{abstract}
	\maketitle
	\tableofcontents
	
	
	\section{Introduction}
	
	The spectral properties of certain classical operators of geometric origin, such as the Laplacian and the Dirac operators, are tightly related to the geometry of the underlying manifold. This observation, at the foundation of spectral geometry, is also the starting point of Connes' approach to noncommutative geometry, where spectral triples can be viewed as abstractions of Dirac-type operators  \cite{Connes89,Connes}. 
	
	In his classical pioneering work \cite{Fukaya87}, Fukaya proved continuity of the eigenvalues of the Laplacian for a class of closed manifolds with uniformly bounded diameter and bounded below sectional curvature.
	Similar results about the continuity of the spectra of Dirac operators have also been established, see e.g.  \cite{Ammann99,Ammann-Bar-98,Lott2002Berk,Lott2002Duke,Lott2002Euro, Roos2018Thesis, Roos20}. These results focus on ``dimensional collapse'', i.e., the situation where closed Riemannian manifolds converge to a lower dimensional manifold for the Gromov-Hausdorff distance.  (See also Section \eqref{sec:smooth projectable case} for more classical references.)
	Recently, there has been an abundance of research on  possible generalizations of such geometric (manifold) results to the non-commutative realm.
	In this general area, 	 a substantial body of work has been produced on noncommutative principal $G$-bundles (with $G$ compact Lie) in the deformed and undeformed contexts, 	and their properties, see \cite{Baum07,Echterhoff09,Hannabuss10,Landi05, Li09, Brz20, Aschieri20}, together with the 
	factorization of Dirac operator in the setting of $G$-$KK$-classes and spectral triples; see  for example   \cite{Cacic21, Cacic24, ForsythRennie19, KaadvS18a, Dab-Sit, Dab-Sit-Zucca, ZuccaThesis, Mesland24a, Mesland24b, ZuccaThesis} and references therein. Indeed many of these papers detail  Kasparov products  factorizations of Dirac operators into a vertical and horizontal component,up to a correction term, like we also  are considering.
	In particular we would like to point out the recent paper \cite{Cacic24}
	which, together with \cite{Cacic21} proposes a gauge theory for noncommutative principal $G$-bundles that even extends  to settings not covered by spectral triples. 
	Their work  has allowed for the direct introduction into the realm of unbounded $KK$-theory of  geometric tools such as geodesic completeness, localization,locally bounded perturbations, and homotopies.

	In this generalized noncommutative context, metric spectral triples are special in the sense that are characterized by  the property that their Dirac operator  induces the weak* topology on the state space of the associated C*-algebra, thus giving it the structure of a  \emph{quantum compact metric space.} 
	This work focuses on spectral properties of the Dirac operator of a $G$--spectral triple, and their geometric limits when the vertical component is collapsed.
	
	\medskip
	
	The quantum compact metric space framework, via the pioneering work of Rieffel, see e.g.  \cite{Rieffel98a, Rieffel99, Rieffel00},  has proven crucial   in the approximation of objects by their  discrete and/or fuzzy analogs, even in the absence of a spectral triple inducing it.  This also extends to more general settings than C*-algebras, and also in part to  tensor products \cite{Kaad23}, as well as the equivariant context, see e.g. \cite{Kaad21}.
	In the formulation  we will use here (introduced by the second author) the relevant distance  is called the \emph{propinquity} (or, sometimes the \emph{dual propinquity}).

	\medskip
	
	The propinquity, originally introduced for quantum compact metric spaces arising from state spaces of unital C*-algebras, was also extended by the second author to many other classes of C*-algebraic objects, and many examples rooted in matrix models in mathematical physics together with  the problem of their convergence  could thus be  studied;  see e.g.  \cite{Latremoliere13b,Latremoliere13,Latremoliere13c, Latremoliere15, Latremoliere14, Latremoliere18d,  Latremoliere21a}.  
	These extensions have been used by the second author to define the \emph{spectral propinquity} over \emph{metric spectral triples} \cite{Latremoliere18g, Latremoliere22}. In this way large categories of noncommutative quantized objects have been endowed with metrics that permit direct comparisons, as well as the definition of  continuous functions between different classes.  In another direction, 
	Rieffel's original definitions (as well as \cite{Kerr02}) have been applied  to broader contexts such as operator systems and truncations of geometric operator spectra; see \cite{Connes21} and the many papers it has inspired as well as \cite{Li06}. 	Many underlying questions  in this general area remain to be fully explored and of particular interest to us is the continuous dependence in the (spectral) propinquity of families of  quantum compact metric spaces and spectral triples.  
	
	\medskip

	This paper adds  an additional viewpoint  to the growing area  of noncommutative principal  $G$-bundles and their geometric invariants:    we prove that spectral triples on  noncommutative $G$-bundles converge  to the base spectral triple with respect to the spectral propinquity, when the metric on the fibres collapses. The \lq \lq vertical direction" is represented  by an unbounded Kasparov module which connects the total algebra to the base algebra, and is made to collapse in our constructions. Intuitively, we \lq \lq shrink"  the vertical factor of the Kasparov product.
	As we show in the last section, this  is  a  generalization to the noncommutative context of 
	classical manifold results on principal $U(1)$-bundles at the level of limiting spectral triples. In a nutshell, the family of Dirac operator spectral triples on  principal $G$-bundles converges to the spectral triple on the base algebra, when the metric on the fibers tends to zero. Of note is that even though several results concerning limits of parameter-depending quantum compact metric spaces are present in the literature, see e.g.  
	\cite{Li09a, Latremoliere15d, Kaad18, Got21, Aguilar22, Kaad22}, very few examples of limits of metric spectral triples are known. 
	
	Our results are based in a crucial way on the analysis performed by Schwieger and Wagner for free dynamical systems \cite{SchwiegerWagner17a, SchwiegerWagner17b, SchwiegerWagner17c, SchwiegerWagnerA, SchwiegerWagner22, SchwiegerWagner24} (see also \cite{Cacic21}), and in particular on the paper \cite{SchwiegerWagner22}, in which they construct a spectral triple on a noncommutative $G$-bundle via a spectral triple on the fixed point algebra, as the  concrete realization of a Kasparov product in $KK$-theory.  See also \cite{Kaad24} for $U(1)$-construction only at the level of quantum compact metric spaces.
	The Schwieger and Wagner construction parallels and complements the noncommutative principal $G$-bundles $KK$-theory factorization results of many authors; see \cite{Kaad13,ForsythRennie19, KaadvS18a, Brain16, Dab-Sit, Dab-Sit-Zucca, ZuccaThesis, Mesland16,Mesland24a, Mesland24b, ZuccaThesis} and the references within.

	This paper	has connections to  mathematical physics, and  in particular the co-action formalism  plays a pivotal role. 
	Because of this, we also feel that our work bridges a gap between  different mathematical specialties.
	
	\medskip
	We now describe the structure of the paper as well as our main results.  Section \eqref{sec:introd-mat} contains introductory material aimed at familiarizing the readers with aspects of the (spectral) propinquity and can be skipped by readers already familiar with that material.
	Motivated by the structure of limits for closed manifolds in the sense of the Gromov-Hausdorff distance, when the limit is also smooth, as described by Fukaya \cite{Fukaya87}, we will work in this paper with spectral triples $(\A,\Hilbert,\Dirac)$ with a particular structure, akin to a principal $G$--bundle over a base space. 
	Informally, we will assume give a C*-subalgebra $\B$ of $\A$, with $1\in \B$, with $\B$ our ``base space'' and $\A$ the analogue of the algebra of continuous sections of some $G$--bundle over the noncommutative space $\B$. The noncommutative analogue of the projection in a bundle is given here by a conditional expectation from $\A$ onto $\B$. We assume given two self-adjoint operators $\Dirac_h$ and $\Dirac_v$ --- respectively seen as the ``horizontal'' and the ``vertical component'' of $\Dirac$, and defined on the domain of $\Dirac$ in $\Hilbert$, such that $\Dirac = \Dirac_v + \Dirac_h$. We will require that $\Dirac_v$ commutes with $\B$, though not with $\Dirac_h$ in general. If $p$ is the orthonormal projection on the kernel $\ker\Dirac_v$ of $\Dirac_v$, we also ask that $(\B,\ker\Dirac_v,p\Dirac_h p)$ is a spectral triple as well. Under some technical conditions listed in Theorem (\ref{main-thm}), we show how the  the spectral triple $(\A,\Hilbert,\Dirac)$ ``collapses''  to the spectral triple $(\B,\ker\Dirac_v,p \Dirac_h p)$ of the base space, when we ``shrink'' the fibers, i.e., equivalently,  rescale $\Dirac_v$. To this end, we will employ the \emph{spectral propinquity}, a metric defined by the second author on the space of metric spectral triples \cite{Latremoliere18g,Latremoliere22}. More in detail our main result is:
	
	\begin{theorem*}(Theorem \eqref{main-thm})
		Let $(\A,\Hilbert,\Dirac)$ be a metric spectral triple, and let $\B \subseteq\A$ be a unital C*-subalgebra of $\A$, and such that $\Dirac = \Dirac_h + \Dirac_v$, where $\Dirac_v$ is self-adjoint and such that $0$ is isolated in $\spectrum{\Dirac}$, together with the following assumptions. Setting $\Dirac_\varepsilon \coloneqq \Dirac_h + \frac{1}{\varepsilon}\Dirac_v$ for all $\varepsilon \in (0,1)$, the triple $(\A,\Hilbert,\Dirac_\varepsilon)$ is a spectral triple, such that:
		\begin{enumerate}
			\item for all $\varepsilon \in (0,1)$, 
			\begin{equation}\label{main-thm-eq-1}
			\opnorm{[\Dirac_h,a]}{}{\Hilbert}\leq \opnorm{[\Dirac_\varepsilon,a]}{}{\Hilbert} \text,
			\end{equation}
			\item there exists $M > 0$ such that for all $\varepsilon \in (0,1)$,
			\begin{equation}\label{main-thm-eq-3}
			\opnorm{[\Dirac_v,a]}{}{\Hilbert} \leq M\, \varepsilon\, \opnorm{[\Dirac_\varepsilon,a]}{}{\Hilbert} \text,
			\end{equation}
			\item $[\Dirac_v,b] = 0$ for all $b \in \B$,
			\item writing $p$ for the projection onto $\ker \Dirac_v$, we assume that $[p,b] = 0$ for all $b \in \B$ and $[\Dirac_h,p] = 0$,
			\item $(\B,\ker\Dirac_v, \Dirac_\B)$, where $ \Dirac_\B \coloneqq p\Dirac_h p$,  is a metric spectral triple,
			\item there exists a positive linear map $\mathds{E} : \A \rightarrow \B$, whose restriction of $\mathds{E}$ to $\B$ is the identity, and a constant $k>0$ such that for all $a\in\A$:
			\begin{equation}\label{main-thm-eq-5}
			\norm{ a - \mathds{E}(a) }{\A} \leq k \opnorm{ [\Dirac_v,a] }{}{\Hilbert} \text,
			\end{equation}
			and
			\begin{equation}\label{main-thm-eq-6}
			\opnorm{p[\Dirac_h,\mathds{E}(a)]p}{}{\Hilbert} = \opnorm{[\Dirac_h,\mathds{E}(a)]}{}{\Hilbert} \leq \opnorm{[\Dirac_h,a]}{}{\Hilbert} \text.
			\end{equation}
		\end{enumerate}
		Then $(\A,\Hilbert,\Dirac_\varepsilon)$ is a metric spectral triple, and:
		\begin{equation*}
		\lim_{\varepsilon\rightarrow 0^+} \spectralpropinquity{}((\A,\Hilbert,\Dirac_\varepsilon),(\B,\ker\Dirac_v,\Dirac_\B)) = 0 \text.
		\end{equation*}
	\end{theorem*}
	
	 Our starting data is a  metric spectral triple whose Dirac operator is decomposed into a \lq \lq horizontal" and \lq \lq vertical" component.
	On the vertical component  we perform a \lq \lq perturbation" consisting of  rescaling it by the factor $\frac{1}{\varepsilon}$.
	We then take the limit for $\varepsilon \to 0$ of the perturbed spectral triples, i.e.,  we collapse the vertical component. The limit  converges  in the spectral propinquity to the spectral triple associated to the horizontal component. In Section 	\eqref{sec:easy-examples}
	we  then  apply our main theorem to a plethora of examples, the first of which is that of a  product of metric spectral triples, with one of them over an Abelian C*-algebra. We see  that such tensor products are always metric, and indeed, collapse occurs as expected, see  Theorem \eqref{product-case-thm}. A special case of the product example is the collapse of any spectral triple to a point --- interestingly, we obtain in Corollary \eqref{cor:hrm-spin} a nontrivial limit where the spectral triple acts on the kernel of the Dirac operator. This shows that the dimension of the space of harmonic spinors is a sort of  ``trace'' of the original spectral triple after collapse. In Section \eqref{sec:SW-appl}, we apply our work to the spectral triples constructed by Schwieger and Wagner in \cite{SchwiegerWagner22} over noncommutative principal $G$-fiber bundles, see \eqref{thm: main conv result G-bundles}. This very interesting class of examples, which are certainly no longer products in general, include C*-crossed-products, and also classical and nontrivial examples like homogeneous spaces of compact Lie groups and principal $U(1)$-bundles, covered in Section 	\eqref{sec:smooth projectable case}.

\section*{Acknowledgements}

The second author is very thankful for the very valuable help provided by K. Schwieger and S. Wagner about their work in \cite{SchwiegerWagner17a,SchwiegerWagner17b,SchwiegerWagner17c,SchwiegerWagner22, SchwiegerWagner24}. 
This work was partially supported by the Simons Foundation (Simons Foundation collaboration grant \#523991.

\section*{Statements}

On behalf of all authors, the corresponding author states that there is no conflict of interest.

This manuscript has no associated data.
	
\section{Introductory Material}
\label{sec:introd-mat}

In this section we  review the basic definitions 
and 
results fundamental results on the (spectral) propinquity.

Spectral triples, introduced by Connes in 1985, have emerged has the preferred encoding tool for geometric information over noncommutative algebras. They are unbounded K-cycles for K-homology; in other words, they are abstractions of first order pseudo-elliptic operators. 
\begin{definition}[{\cite{Connes}}]
	A \emph{spectral triple} $(\A,\Hilbert,\Dirac)$ is a triple consisting of a unital C*-algebra $\A$, a Hilbert space $\Hilbert$ which is also a left $\A$-module, and a self-adjoint operator defined on a dense subspace $\dom{\Dirac}$ of $\Hilbert$, such that:
	\begin{equation}\label{AD-eq}
		\A_{\Dirac} \coloneqq \left\{ a \in \A : a\dom{\Dirac}\subseteq\dom{\Dirac}\text{, }[\Dirac,a]\text{ bounded } \right\}
	\end{equation}
	is a dense *-algebra of $\A$, and $(\Dirac+i)^{-1}$ is a compact operator. 
	
	The operator $\Dirac$ is called the \emph{Dirac operator} of the spectral triple $(\A,\Hilbert,\Dirac)$.
\end{definition}

Motivated by the structure of limits for manifolds in the sense of the Gromov-Hausdorff distance, when the limit is also smooth, as described by Fukaya \cite{Fukaya87}, we will work in this paper with spectral triples $(\A,\Hilbert,\Dirac)$ with a particular structure, akin to a bundle over a base space. Informally, we will assume given a C*-subalgebra $\B$ of $\A$, with $1\in \B$, with $\B$ our ``base space'' and $\A$ the analogue of the algebra of continuous sections of some bundle over the noncommutative space $\B$. The noncommutative analogue of the projection in a bundle is given here by a conditional expectation from $\A$ onto $\B$. We assume given two self-adjoint operators $\Dirac_h$ and $\Dirac_v$ --- respectively seen as the ``horizontal'' and the ``vertical component'' of $\Dirac$, and defined on the domain of $\Dirac$ in $\Hilbert$, such that $\Dirac = \Dirac_v + \Dirac_h$. We will require that $\Dirac_v$ commutes with $\B$, though not with $\Dirac_h$ in general. If $p$ is the orthonormal projection on the kernel $\ker\Dirac_v$ of $\Dirac_v$, we also ask that $(\B,\ker\Dirac_v,p\Dirac_h p)$ is a spectral triple as well. Under some technical conditions listed in Theorem (\ref{main-thm}), we will study the ``collapse'' of the spectral triple $(\A,\Hilbert,\Dirac)$ to the spectral triple $(\B,\ker\Dirac_v,p \Dirac_h p)$, when we ``shrink'' the fibers, i.e. rescale $\Dirac_v$ by a factor of $\frac{1}{\varepsilon}$. This collapse is to be intended as a limit in the \emph{spectral propinquity}, a metric defined by the second author on the space of metric spectral triples.

\medskip

The spectral propinquity is a recent development in noncommutative metric geometry. Noncommutative metric geometry is a framework developed over two decades, with its roots in the observation by Connes \cite{Connes89} that a spectral triple $(\A,\Hilbert,\Dirac)$ defines an extended pseudo-metric on the state space $\StateSpace(\A)$ of the C*-algebra $\A$, by setting, for any two $\varphi,\psi \in \StateSpace(\A)$:
\begin{equation}\label{Connes-dist-eq}
	\Kantorovich{\Dirac}(\varphi,\psi) \coloneqq \sup\left\{ |\varphi(a) - \psi(a)| : a\in\A_{\Dirac}, \Lip_{\Dirac}(a) \leq 1 \right\}
\end{equation}
where
\begin{equation}\label{Dirac-L-eq}
	\Lip_{\Dirac} : a \in \A_{\Dirac} \longmapsto \opnorm{[\Dirac,a]}{}{\Hilbert}
\end{equation}
where we use the following notation, here and throughout this entire paper.
\begin{notation}
	If $E$ is a normed vector space, then its norm is denoted by $\norm{\cdot}{E}$  by default. Moreover, if $E$ and $F$ are both normed vector spaces, and $T : E\rightarrow F$ is a bounded linear operator, then the operator norm of $T$ is denoted by $\opnorm{T}{E}{F}$; if $E=F$ then we simply write $\opnorm{T}{}{E}$.
\end{notation}

The seminorm $\Lip_{\Dirac}$ is akin to a Lipschitz seminorm, and thus, Connes' distance $\Kantorovich{\Dirac}$ can be seen as a generalization of the {\MongeKant}, introduced by Kantorovich \cite{Kantorovich40,Kantorovich58} over any metric space. With this in mind, the natural question becomes: under what condition is $\Kantorovich{\Dirac}$ a metric for the weak* topology on the state space of $\A$, just as the classical {\MongeKant} is? This leads us to Rieffel's pioneering work in \cite{Rieffel98a,Rieffel99}. The following definition, used by the second author \cite{Latremoliere13} in his work on convergence of spectral triples, captures the core properties that a Lipschitz seminorm possesses and enables us to derive a noncommutative theory of Gromov-Hausdorff convergence. We will only need {\qcms s} which satisfy the usual form of the Leibniz inequality, and refer to \cite{Latremoliere15} for a more general definition.

\begin{definition}[\cite{Latremoliere13,Latremoliere14,Latremoliere15}]\label{qcms-def}
	A \emph{\qcms} $(\A,\Lip)$ is an ordered pair of a unital C*-algebra $\A$, and a seminorm $\Lip$ defined on a dense subspace $\dom{\Lip}$ of the space $\sa{\A} \coloneqq \{ a \in \A : a=a^\ast\}$ of self-adjoint elements of $\A$, such that:
	\begin{enumerate}
		\item $\{ a \in \dom{\Lip} : \Lip(a) = 0 \} = \R 1$,
		\item the {\MongeKant} $\Kantorovich{\Lip}$ defined on the state space $\StateSpace(\A)$ of $\A$ by
		\begin{equation}\label{MongeKantorovich-eq}
			\forall\varphi,\psi\in\StateSpace(\A) \quad \Kantorovich{\Lip}(\varphi,\psi) \coloneqq \sup\left\{ |\varphi(a) - \psi(a)| : a\in\dom{\Lip}, \Lip(a)\leq 1 \right\} 
		\end{equation}
		metrizes the weak* topology,
		\item for all $a,b \in \dom{\Lip}$, the Jordan product $\Jordan{a}{b}\coloneqq\frac{ab+ba}{2}$ and the Lie product $\Lie{a}{b}\coloneqq\frac{ab-ba}{2i}$ both lie in $\dom{\Lip}$, and
		\begin{equation*}
			\max\left\{ \Lip(\Jordan{a}{b}), \Lip(\Lie{a}{b}) \right\} \leq \Lip(a)\norm{b}{\A} + \norm{a}{\A}\Lip(b)\text,
		\end{equation*}
		\item $\{ a \in \dom{\Lip} : \Lip(a) \leq 1 \}$ is closed in $\sa{\A}$.
	\end{enumerate}
	The seminorm $\Lip$ is then called a \emph{L-seminorm} (where $L$ stands for Lipschitz). 
\end{definition}
\begin{convention}
	Let $(\A,\Lip)$ be a {\qcms}. We assign $L(a)\coloneqq\infty$ whenever $a\notin\dom{\Lip}$, with the algebraic conventions typically in use in measure theory. Thus $\dom{\Lip} = \{ a \in \sa{\A} : \Lip(a) \leq 1 \}$. With this extension, $\Lip$ is a lower semicontinuous function over $\sa{\A}$.
\end{convention}
In particular, we will focus in this paper on {\qcms s} constructed from spectral triples.
\begin{definition}
	A spectral triple $(\A,\Hilbert,\Dirac)$ is \emph{metric} when $(\A,\Lip_{\Dirac})$ is a {\qcms}, where $\Lip_{\Dirac}$ is defined in Equation \eqref{Dirac-L-eq}.
\end{definition}
We remark that a spectral triple is metric if, and only if, Connes' metric given in Equation \eqref{Connes-dist-eq}, induces the weak* topology on $\StateSpace(\A)$, as all other properties of a {\qcms} are satisfied automatically.

\medskip

Rieffel's motivation for the introduction of {\qcms s} was the construction in \cite{Rieffel00} of an analogue of the Gromov-Hausdorff distance for noncommutative geometry, with an eye to applications  in mathematical physics, where various approximations of physical models are constructed as informal limits of finite dimensional models. As this nascent area of research progressed,  the continuity, with respect to Rieffel's metric of various structures associated with {\qcms s}, such as modules or group actions, gained momentum.  Hence it became important to discover a noncommutative version of the Gromov-Hausdorff distance adapted to the category of C*-algebra, and even further, to spectral triples. The second author thus developed the \emph{propinquity} on the class of {\qcms s}, and a stronger metric, the \emph{spectral propinquity}, on the class of metric spectral triples. 

The propinquity, upon which the spectral propinquity is based upon,  is indeed a complete metric, up to the appropriate notion of isomorphism for {\qcms s}, given by \emph{full quantum isometries}.

For the convenience of the reader,  will now recall Latr\'emoli\`ere's constructions, starting with the definition of \emph{ (full) quantum isometry}; see  \cite{Latremoliere13,Latremoliere13b,Latremoliere14,Latremoliere15,Rieffel2021} for more details. 
\begin{definition}
	A \emph{quantum isometry} $\pi : (\A,\L_\A)\rightarrow(\B,\Lip_\B)$ between two {\qcms s} $(\A,\Lip_\A)$ and $(\B,\Lip_\B)$ is a *-epimorphism such that, for all $b\in\dom{\Lip_\B}$:
	\begin{equation*}
		\Lip_\B(b) = \inf\left\{ \Lip_\A(a) | a \in \pi^{-1}(b) \cap \dom{\Lip_\A}\right\} \text.
	\end{equation*}
	A quantum isometry which is a *-isomorphism, and whose inverse is also a quantum isometry, is called an \emph{full quantum isometry}. Specifically, $\pi : (\A,\Lip_\A)\rightarrow(\B,\Lip_\B)$ is a full quantum isometry if, and only if, it is a *-isomorphism from $\A$ onto $\B$ such that $\Lip_\B\circ\pi = \Lip_\A$ over $\sa{\A}$.
\end{definition}
The notion of quantum isometry is motivated by McShane's extension theorem for real-valued Lipschitz functions \cite{McShane34}. If $\pi : (\A,\Lip_\A) \rightarrow (\B,\Lip_\B)$ is a quantum isometry, then its dual map 
\begin{equation*}
	\pi^\ast : \varphi \in \StateSpace(\B) \mapsto \varphi\circ\pi \in \StateSpace(\A)
\end{equation*}
is indeed, an isometry from $(\StateSpace(\B),\Kantorovich{\Lip_\B})$ into $(\StateSpace(\A),\Kantorovich{\Lip_\A})$.

Following the ideas of Edwards \cite{Edwards75}, Gromov \cite{Gromov81} and Rieffel \cite{Rieffel00}, we are led to introducing the following notion of a ``isometric embedding'' of two {\qcms s} into a third one in noncommutative geometry.
\begin{definition}\label{tunnel-def}
	Let $(\A,\Lip_\A)$ and $(\B,\Lip_\B)$ be two {\qcms s}. A \emph{tunnel} $\tau\coloneqq(\D,\Lip_\B,\pi_\A,\pi_\B)$ from $\dom{\tau}\coloneqq (\A,\Lip_\A)$ to $\codom{\tau}\coloneqq(\B,\Lip_\B)$ is given by a {\qcms} $(\D,\Lip_\D)$, and two quantum isometries $\pi_\A : (\D,\Lip_\D) \rightarrow (\A,\Lip_\A)$ and $\pi_\B : (\D,\Lip_\D)\rightarrow (\B,\Lip_\B)$.
\end{definition}
A tunnel enables us to quantify how far two {\qcms s} are from each others, as follows.

\begin{notation}
		The Hausdorff distance induced on the space of closed subsets of a compact metric space $(X,d)$ is denoted by $\Haus{d}$.
	\end{notation}
\begin{definition}\label{extent-def}
	The \emph{extent}  $\tunnelextent{\tau}$ of a tunnel $\tau\coloneqq(\D,\Lip_\D,\pi_\A,\pi_\B)$ from $(\A,\Lip_\A)$ to $(\B,\Lip_\B)$ is the number:
	\begin{equation*}
		\tunnelextent{\tau} \coloneqq \max\left\{ \Haus{\Kantorovich{\Lip_\D}}\left(\StateSpace(\D),\pi_\A^\ast(\StateSpace(\A))\right),
		\Haus{\Kantorovich{\Lip_\B}}\left(\StateSpace(\D),\pi_\B^\ast(\StateSpace(\B))\right) \right) \text.
		\Big\}
	\end{equation*}
\end{definition}

The \emph{Gromov-Hausdorff propinquity} is a complete distance, up to full quantum isometry, on the class of {\qcms s}, defined as follows, and our starting point in defining a distance between metric spectral triples.
\begin{definition}\label{prop-def}
	The \emph{propinquity} between two {\qcms s} $(\A,\Lip_\A)$ and $(\B,\Lip_\B)$ is the real number:
	\begin{equation*}
		\dpropinquity{}((\A,\Lip_\A),(\B,\Lip_\B)) \coloneqq \inf\left\{ \tunnelextent{\tau} : \tau \text{ is a tunnel from }(\A,\Lip_\A)\text{ to }(\B,\Lip_\B) \right\} \text.
	\end{equation*}
\end{definition}
We refer to \cite{Latremoliere13,Latremoliere13b,Latremoliere14,Latremoliere15} for some basic references on this metric. We record that it induces the same topology as the Gromov-Hausdorff distance on the class of classical metric spaces.

\medskip

Spectral triples include more information than their associated {\qcms s} and Connes' metrics,  so we now wish to strengthen the propinquity in such a way  that distance $0$ means unitary equivalence in the following sense:
\begin{definition}
	Two spectral triples $(\A,\Hilbert,\Dirac)$ and $(\B,\Hilbert[J],\Dirac[S])$ are \emph{unitarily equivalent} when there exists a unitary operator $U : \Hilbert \rightarrow \Hilbert[J]$ such that
	\begin{equation*}
		U\dom{\Dirac}=\dom{\Dirac[S]} \text{ and } U^\ast \Dirac[S] U = \Dirac
	\end{equation*}
	while $\mathrm{Ad} U$ restricts to a *-isomorphism from $\A$ onto $\B$ (seen as C*-algebras of operators on $\Hilbert$ and $\Hilbert[J]$.
\end{definition}

Now,  metric spectral triples give rise to  special \emph{\gMVB s}, as explained in  the paragraph after Remark \eqref{rm:names}). Therefore 
once we have extended 
the propinquity to the class of \emph{metrical C*-correspondences}, it will be also automatically extended to metric spectral triples. We now go through the details of these constructions.

\begin{definition}
	A \emph{$C^\ast$-correspondence} $(\module{M},\A,\B)$, where $\A$ and $\B$ are two unital $C^\ast$-algebras, is a right Hilbert $\B$-module, together with a unital *-morphism from $\A$ to the C*-algebra of adjointable, $\B$-linear operators over $\module{M}$.
\end{definition}

\begin{definition}\label{mvb-def}
	A {\gMVB} $(\module{M},\CDN,\A,\Lip_\A,\B,\Lip_\B)$ is a C*-correspondence $(\module{M},\A,\B)$, two {\qcms s} $(\A,\Lip_\A)$ and $(\B,\Lip_\B)$, and a norm $\CDN$ on a dense $\C$-subspace $\dom{\CDN}$ of $\module{M}$, such that:
	\begin{enumerate}
		\item $\{ \omega \in \dom{\CDN} : \CDN(\omega)\leq 1 \}$ is compact in $\module{M}$,
		\item $\CDN(\omega)\geq\norm{\omega}{\module{M}}$ for all $\omega\in\dom{\CDN}$,
		\item for all $a \in \dom{\Lip_\A}$, and for all $\omega\in\dom{\CDN}$, we have $a\omega\in\dom{\CDN}$, and
		\begin{equation*}
			\CDN(a\omega) \leq (\norm{a}{\A} + \Lip_\A(a))\CDN(\omega)\text,
		\end{equation*}
		\item for all $\omega,\eta\in\dom{\CDN}$, we have $\inner{\omega}{\eta}{\B}\in\dom{\Lip_\B}$, and
		\begin{equation*}
			\Lip_\A(\inner{\omega}{\eta}{\B}) \leq 2\CDN(\omega)\CDN(\eta) \text.
		\end{equation*}
	\end{enumerate}
\end{definition}

\begin{remark}\label{rm:names}
	We refer to Property (3) in Definition (\ref{mvb-def}) as the \emph{modular Leibniz inequality}, and Property (4) as the \emph{inner Leibniz inequality}. We call a norm the norm $\CDN$ of Definition (\ref{mvb-def}) a \emph{D-norm}.
\end{remark}

Four our purpose, given a metric spectral triple $(\A,\Hilbert,\Dirac)$, we obtain a  {\gMVB} as follows \cite[Theorem 2.7]{Latremoliere18g}: we define
\begin{equation*}
	\CDN : \omega\in\dom{\Dirac}\mapsto \norm{\omega}{\Hilbert} + \norm{\Dirac\omega}{\Hilbert} 
\end{equation*}
as the graph norm of the Dirac operator $\Dirac$. Then
\begin{equation}\label{eq:metric-sp-corr}
	\left( \Hilbert, \CDN, \A, \Lip_{\Dirac}, \C, 0 \right)
\end{equation}
is a {\gMVB}, denoted by $\mcc{\A}{\Hilbert}{\Dirac}$, where $\Lip_{\Dirac}$ is defined by Equation \eqref{Dirac-L-eq}.

\medskip

We now extend the propinquity to the class of {\gMVB s}. We note in passing that  metric spectral triples give rise to rather specific {\gMVB s} defined using Hilbert spaces, rather than more general Hilbert modules.
However   to establish the triangle inequality of the extended propinquity we require the  more general concepts outlined  below and cannot restrict to 
{\gMVB s} arising from metric spectral triples only.

Now, we introduce quantum isometries between {\gMVB s}.

\begin{definition}
	Let $\mathds{M} \coloneqq (\module{M},\CDN,\A,\Lip_\A,\B,\Lip_\B)$ and $\mathds{P} \coloneqq (\module{P},\CDN[TN],\D,\Lip_\D,\alg{E},\Lip_{\alg{E}})$ be two {\gMVB s}. A \emph{quantum isometry} $(\Pi,\pi,\theta)$ from $\mathds{M}$ to $\mathds{P}$ is given by two quantum isometries $\pi : (\A,\Lip_\A)\rightarrow(\D,\Lip_\D)$ and $\theta:(\B,\Lip_\B)\rightarrow(\alg{E},\Lip_{\alg{E}})$, as well as a $\C$-linear map $\Pi:\module{M}\rightarrow\module{P}$, such that:
	\begin{enumerate}
		\item $\Pi(a\omega) = \pi(a)\Pi(\omega)$ for all $a\in\A$ and $\omega\in\module{M}$,
		\item $\Pi(\omega b) = \Pi(\omega)\theta(b)$ for all $\omega\in\module{M}$ and $b\in\B$,
		\item $\theta(\inner{\omega}{\eta}{\module{M}})=\Pi(\inner{\omega}{\eta}{\module{M}})$ for all $\omega,\eta\in\module{M}$,
		\item $\CDN[TN](\omega) = \inf \CDN(\Pi^{-1}(\{\omega\}))$ for all $\omega\in\dom{\CDN[TN]}$.
	\end{enumerate}
\end{definition}

The definition of a distance between {\gMVB s}, called the \emph{metrical propinquity}, relies on a notion of isometric embedding called a tunnel, which is defined as follows.

\begin{definition}[{\cite[Definition 2.19]{Latremoliere18g}}]\label{mcc-tunnel-def}
	Let $\mathds{M}_1$ and $\mathds{M}_2$ be two metrical C*-correspondences. A \emph{metrical tunnel} $\tau = (\mathds{J},\Pi_1,\Pi_2)$ from $\mathds{M}_1$ to $\mathds{M}_2$ is a triple given by a metrical C*-correspondence $\mathds{J}$, and for each $j\in\{1,2\}$, a metrical quantum isometry $\Pi_j : \mathds{J}\mapsto \mathds{M}_j$.
\end{definition}

We now proceed by defining the extent of a metrical tunnel; this only involves our previous notion of extent of a tunnel between {\qcms s}.

\begin{definition}[{\cite[Definition 2.21]{Latremoliere18g}}]\label{mcc-extent-def}
	Let $\mathds{M}_j = (\module{M}_j,\CDN_j,\A_j,\Lip_j,\B_j,\Lip[S]_j)$ be a metrical C*-correspondence, for each $j \in \{1,2\}$. Let $\tau = (\mathds{P},(\Pi_1,\pi_1,\theta_1),(\Pi_2,\pi_2,\theta_2))$ be a metrical tunnel from $\mathds{M}_1$ to $\mathds{M}_2$, with $\mathds{P} = (\module{P},\CDN[TN],\D,\Lip_\D,\alg{E},\Lip_{\alg{E}})$.
	
	The \emph{extent} $\tunnelextent{\tau}$ of the metrical tunnel $\tau$ is
	\begin{equation*}
	\tunnelextent{\tau} \coloneqq \max\left\{ \tunnelextent{\D,\Lip_\D,\pi_1,\pi_2}, \tunnelextent{\alg{E},\Lip[T]_{\alg{E}},\theta_1,\theta_2} \right\} \text.
	\end{equation*}
\end{definition}

Given two metric spectral triples, we can thus either take the propinquity $\dpropinquity{}$ between their underlying {\qcms s}, or take the metrical propinquity \cite{Latremoliere16c,Latremoliere18d} denoted by  $\dmetpropinquity{}$ between the metrical C*-correspondence they define, which is defined as the infimum of the extent of every possible metrical tunnel between them. (See Figure \eqref{fig:metr-tunnel})

\begin{figure}[t]
	\centering
	\scalebox{.60}{
		\xymatrix{	& & ( \D  ,\Lip_\D )  \ar@{.>}[d] \ar@/_/@{>>}[llddd]_{ \pi_{\A }}  \ar@/^/@{>>}[rrddd]^{  \pi_{\B}} & & \\
			& & (  \module{P},\CDN[TN] ) 
			\ar@/_/@{>>}[ldd]_{\Pi_{\A}} \ar@/^/@{>>}[rdd]^{\Pi_{\B }} & & \\
			& &  (\alg{E} , \Lip_{\alg{E}} )\ar@{>>}[ldd]^{\theta_{\A}} \ar@{>>}[rdd]_{\theta_{\B }} \ar@{-->}[u] &  \\
			(\A ,\Lip_{\Dirac_{\A}}) \ar@{.>}[r] & (\A, \Hilbert_{\A}, \Dirac_\A ) & & (\B, \Hilbert_{\B}, \Dirac_\B ) & (\B,\Lip_{\Dirac_{\B}}) \ar@{.>}[l] \\
			& (\C, 0) \ar@{-->}[u] & & (\C, 0)  \ar@{-->}[u]}	}
	\caption{A metrical tunnel between C*-correspondences associated to metric spectral triples. The extent of the above metrical tunnel is the max of the  extents of the top and bottom tunnels. All the maps in the picture are quantum isometries. }\label{fig:metr-tunnel}
\end{figure}
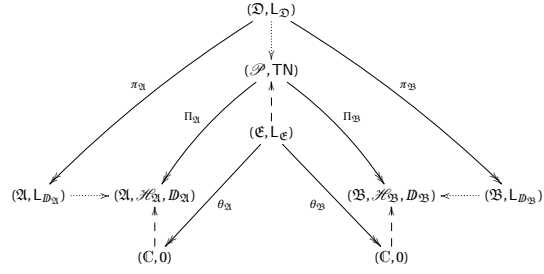

However, the metrical propinquity does not lead to the desired property that distance zero between metric spectral triples implies unitary equivalence of the spectral triples. 
To obtain a metric with the desired property, which we call the spectral propinquity, we involve the second author's  work on the geometry of quantum dynamics \cite{Latremoliere18b,Latremoliere18c,Latremoliere18g}. We now recall the additional properties  the spectral propinquity needs to satisfy. In doing this we follow the construction in \cite{Latremoliere22},  which provides a more conceptual approach, rather than the (equivalent) original construction in \cite{Latremoliere18g}.


The idea of the spectral propinquity is to add, to the metrical propinquity, a measure of how far apart are partial orbits for the natural action of $[0,\infty)$ by unitaries given by exponentiating the Dirac operators of the spectral triples. We thus will involve taking, for any choice of tunnel, the Hausdorff distance between certain sets related to these orbits, for an appropriate metric. Our construction thus begins with an extension of the idea of the {\MongeKant} to metrical C*-correspondences.

\begin{definition}[{\cite[Definition 2.1]{Latremoliere22}}]\label{Kantorovich-mod-def}
Let $(\module{M},\CDN,\A,\Lip_\A,\B,\Lip_\B)$ be a metrical\newline C*-correspondence. For any two continuous $\C$-valued $\C$-linear functionals $\varphi,\psi$ over $\module{M}$, we define:
	\begin{equation*}
	\Kantorovich{\CDN}(\varphi,\psi) \coloneqq \sup\left\{ |\varphi(\omega) - \psi(\omega)| : \omega\in\dom{\CDN}, \CDN(\omega)\leq 1 \right\} \text. 
	\end{equation*} 
\end{definition}

With the notation of Definition (\ref{Kantorovich-mod-def}), $\Kantorovich{\CDN}$ is a metric on the topological dual $\module{M}^\ast$ of $\module{M}$ (seen as a Banach space over $\C$). Since the closed unit ball of the D-norm $\CDN$ is compact in the module norm, a standard argument shows that the metric $\Kantorovich{\CDN}$ induces the weak* topology on bounded subsets of $\module{M}^\ast$.

We then naturally extend the metric in Definition (\ref{Kantorovich-mod-def}) to arbitrary families of linear functionals.
\begin{definition}[{\cite[Definition 2.3]{Latremoliere22}}]\label{Kantorovich-family-def}
	Let $(\module{M},\CDN[TN],\A,\Lip_\A,\B,\Lip_\B)$ be a metrical\newline C*-correspondence. Let $J$ be a nonempty set. For any two families $(\varphi_j)_{j \in J}, (\psi_j)_{j\in J} \in \left(\module{M}^\ast\right)^J$ of continuous $\C$-linear functionals of $\mathds{M}$, we set:
	\begin{equation*}
	\KantorovichJ{\CDN[TN]}((\varphi_j)_{j\in J},(\psi_j)_{j\in J}) \coloneqq \sup \left\{  \Kantorovich{\CDN[TN]}(\varphi_j,\psi_j) : j \in J \right\} \text.
	\end{equation*}
\end{definition}

Our construction calls for a sort of analogue of the state space, but for metrical C*-correspondences. The following choice is what was used to build the spectral propinquity.
\begin{definition}[{\cite[Notation 3.9]{Latremoliere18g},\cite[Definition 2.4]{Latremoliere22}}]\label{pseudo-state-def}
	If $\mathds{M} \coloneqq (\module{M},\CDN[TN],\A,\Lip_\A,\B,\Lip_\B)$ is a {\gMVB}, then a continuous linear functional $\varphi \in \module{M}^\ast$ is a \emph{pseudo-state} of $\mathds{M}$ when there exist $\mu \in \StateSpace(\B)$ and $\omega \in \module{M}$ with $\CDN[TN](\omega)\leq 1$ such that $\varphi$ is given by:
	\begin{equation*}
	\varphi :\xi \in \module{M} \longmapsto \mu\left(\inner{\omega}{\xi}{\module{M}}\right) \text.
	\end{equation*}
	
	The set of all pseudo-states of $\mathds{M}$ is denoted by $\ModStateSpace(\mathds{M})$.
\end{definition}

We now have the tools to define how far apart two families of operators  on two different metrical C*-correspondences are, according to a given tunnel. We call this quantity the \emph{separation between these two families}, according to the chosen tunnel; we also introduce the \emph{dispersion}, which accounts for both the separation and the extent of the tunnel. If $\mathds{M}\coloneqq(\module{M},\CDN[TN],\A,\Lip_\A,\B,\Lip_\B)$ is a metrical C*-correspondence, we will denote by $\alg{L}(\mathds{M})$ the C*-algebra of all $\B$-linear, adjointable operators on the right Hilbert $\B$-module $\module{M}$. 

\begin{definition}[{\cite[Definition 2.7]{Latremoliere22}}]\label{separation-def}
	Let $\mathds{A}$ and $\mathds{B}$ be two {\gMVB s}. Let   $\tau\coloneqq (\mathds{P},(\Pi_\mathds{A},\pi_{\mathds{A}},\theta_{\mathds{A}}),(\Pi_\mathds{B},\pi_{\mathds{B}},\theta_{\mathds{B}}))$ be a metrical tunnel from $\mathds{A}$ to $\mathds{B}$. Let $\CDN[TN]$ be the D-norm of the metrical C*-correspondence $\mathds{P}$. 
	
	Let $J$ be a nonempty set. If $A\coloneqq (a_j)_{j \in J}$ is a family of operators in $\alg{L}(\mathds{A})$, and $B\coloneqq (b_j)_{j \in J}$ is a family of operators in $\alg{L}(\mathds{B})$, then we define the \emph{separation} of $A$ and $B$ according to $\tau$ by: 
	\begin{multline*}
	\tunnelsep{\tau}{A, B} \coloneqq \Haus{\KantorovichJ{\CDN[TN]}}\Big(\left\{ (\varphi\circ a_j\circ\Pi_{\mathds{A}})_{j \in J} : \varphi \in \ModStateSpace(\mathds{A}) \right\}, \\ \left\{ (\psi\circ b_j\circ\Pi_{\mathds{B}})_{j \in J} : \psi \in \ModStateSpace(\mathds{B}) \right\} \Big) \text.
	\end{multline*}
	The \emph{dispersion} of $A$ and $B$ according to $\tau$ is
	\begin{equation*}
	\tunneldispersion{\tau}{A,B} \coloneqq \max\{\tunnelextent{\tau},\tunnelsep{\tau}{A,B} \}\text.
	\end{equation*}
\end{definition}

In particular, if $a$ and $b$ are two bounded adjointable operators on two {\gMVB s} $\mathds{A}$ and $\mathds{B}$, then we can define a distance between them, called the \emph{operational propinquity} $\oppropinquity{}(a,b)$, as
\begin{equation*}
	\oppropinquity{}(a,b) \coloneqq \inf\left\{ \tunneldispersion{\tau}{(a),(b)} : \text{$\tau$ metrical tunnel from $\mathds{A}$ to $\mathds{B}$} \right\} \text.
\end{equation*}
It is  proved in \cite{Latremoliere22} that $\oppropinquity{}(a,b) = 0$ if, and only if, there exists a full quantum isometry from $\mathds{A}$ onto $\mathds{B}$ which intertwines $a$ and $b$. This metric is really defined between families of operators, but we will focus on the spectral propinquity here.

\medskip

We now use \cite[Theorem 3.5]{Latremoliere22} to provide an equivalent formulation of the spectral propinquity, using the dispersion between certain families of exponential of the Dirac operators.
\begin{definition}[{\cite[Definition 4.2]{Latremoliere18g},\cite[Theorem 3.5]{Latremoliere22}}]\label{spectral-propinquity-def}
	The \emph{spectral propinquity} between two metric spectral triples $(\A_1,\Hilbert_1,\Dirac_1)$ and $(\A_2,\Hilbert_2,\Dirac_2)$ is
	\begin{equation*}
	\begin{split}
	\spectralpropinquity{}((\A_1,\Hilbert_1,\Dirac_1),
	&(\A_2,\Hilbert_2,\Dirac_2)) \coloneqq \\
	&\inf\Bigg\{ \frac{\sqrt{2}}{2}, \varepsilon > 0 : \\
	&\quad \exists \tau \text{ tunnel from $\mcc{\A_1}{\Hilbert_1}{\Dirac_1}$ to $\mcc{\A_2}{\Hilbert_2}{\Dirac_2}$} \hbox{ such that }\\ 
	&\; \tunneldispersion{\tau}{(\exp(it\Dirac_1))_{0\leq t \leq \frac{1}{\varepsilon}}, (\exp(it\Dirac_2))_{0\leq t \leq \frac{1}{\varepsilon}} } < \varepsilon  \Bigg\} \text.
	\end{split}
	\end{equation*}
\end{definition}

The spectral propinquity enjoys some very important properties:
\begin{enumerate}
	\item $\spectralpropinquity{}((\A,\Hilbert,\Dirac),(\B,\Hilbert[J],\Dirac[S])) = 0$ if the two metric spectral triples $(\A,\Hilbert,\Dirac)$ and $(\B,\Hilbert[J],\Dirac[S])$ are unitarily equivalent \cite{Latremoliere18g}.
	\item if 
	\begin{equation*}
		\lim_{n\rightarrow\infty} \spectralpropinquity{}((\A_n,\Hilbert_n,\Dirac_n),(\A_\infty,\Hilbert_\infty,\Dirac_\infty)) = 0
	\end{equation*}
	where $(\A_n,\Hilbert_n,\Dirac_n)$ is a metric spectral triple for all $n\in\N\cup\{\infty\}$, then for all bounded continuous function $f : \R\rightarrow\C$, we also have
	\begin{equation*}
		\lim_{n\rightarrow\infty} \oppropinquity{}(f(\Dirac_n),f(\Dirac)) = 0 \text,
	\end{equation*}
	and
	\begin{equation*}
		\spectrum{\Dirac_\infty} = \left\{ \lim_{n\rightarrow\infty} \lambda_n : (\lambda_n)_{n\in\N} \text{ convergent with }\forall n \in\N \; \lambda_n \in \spectrum{\Dirac_n} \right\} \text,
	\end{equation*}
	where $\spectrum{A}$ is the spectrum of the operator $A$.
\end{enumerate}

	\section{A collapse result}
	
	In this section we establish our main result, Theorem \eqref{main-thm}, which is a 
	a general result about collapse of spectral triples in the context of the spectral propinquity. We will use two important lemmas in the proof our main result, Lemma \eqref{restriction-unitary-lemma} and Lemma \eqref{Fourier-estimate-lemma} which we present below. In  Lemma \eqref{restriction-unitary-lemma} we look at the restriction of the 1-parameter group induced by a self-adjoint operator which is equal to  the sum of two self-adjoint operators, to the kernel of one of the terms. This lemma is helpful to us since it  relates the 1-parameter group generated by a spectral triple on the larger algebra to its collapsed limit.  Our second lemma,  Lemma \eqref{restriction-unitary-lemma}, provides  a needed Fourier analysis technical  result applicable when  $0$ is an isolated value in the spectrum of a self-adjoint operator.
	
	\begin{lemma}\label{restriction-unitary-lemma}
		Let $\Dirac$, $\Dirac[S]$ be two self-adjoint operators such that $\Dirac+\Dirac[S]$ is also a well-defined self-adjoint operator. Assume moreover that $\Dirac$ and $\Dirac+\Dirac[S]$ both have discrete spectra. Let $p$ be the orthogonal projection onto the kernel $\ker\Dirac[S]$ of $\Dirac[S]$. If $p$ commutes with $\Dirac$, then for all $t \in \R$:
		\begin{equation*}
		\exp(i t (\Dirac+\Dirac[S])) p = \exp(i t p \Dirac p) \text. 
		\end{equation*}
	\end{lemma}
	
	\begin{proof}
		First, note that $(\Dirac+\Dirac[S])p = \Dirac p$, so $0 \in \spectrum{\Dirac p}$ if and only if $0 \in \spectrum{(\Dirac+\Dirac[S])p}$.
		
		Fix $z\in \C\setminus\R$. Since $\Dirac+\Dirac[S]$ is self-adjoint, the operator $\Dirac+\Dirac[S]+ z$ is invertible, with bounded inverse, and therefore: $(\Dirac+\Dirac[S]+z)(\Dirac+\Dirac[S]+z)^{-1} = 1 \text{ so } (\Dirac+\Dirac[S]+z)(\Dirac+\Dirac[S] + z)^{-1} p = p\text.$
		Since $p$ commutes with both $\Dirac$ and $\Dirac[S]$, it commutes with $\Dirac+\Dirac[S]$ and thus, with $(\Dirac+\Dirac[S] + z)^{-1}$. We thus have
		\begin{equation*}
		p = (\Dirac+\Dirac[S]+z)p(\Dirac+\Dirac[S]+z)^{-1} = (\Dirac+z)p(\Dirac+\Dirac[S]+z)^{-1}p
		\end{equation*}
		and since $\Dirac$ is self-adjoint, $\Dirac+z$ is again invertible, so
		\begin{equation}\label{restriction-eq}
		(\Dirac+z)^{-1} p = p(\Dirac+\Dirac[S]+z)^{-1} p = (\Dirac+\Dirac[S]+z)^{-1} p \text.
		\end{equation}
		Therefore, the resolvent of $\Dirac$ an $\Dirac+\Dirac[S]$ agree on $\ker\Dirac[S]$ over $\C\setminus\R$. By continuity, the restrictions of the resolvent of $\Dirac$ and $\Dirac+\Dirac[S]$ to the kernel of $\Dirac[S]$ agree on the intersection of the resolvent sets of $\Dirac$ and $\Dirac+\Dirac[S]$.
		
		Let $A\subseteq\R$ be any bounded Borel subset of $\R$. Since the spectra of both $\Dirac$ and $\Dirac+\Dirac[S]$ are discrete subsets of $\R$, we can find a closed simple curve $C_A$ such that $A$ lies inside $C_A$ while $C_A$ is entirely within the intersection of the resolvent set of $\Dirac$ and $\Dirac[S]$. Denote the spectral measure of $\Dirac$ by $\mathds{P}_{\Dirac}$ and the spectral measure of $\Dirac+\Dirac[S]$ by $\mathds{P}_{\Dirac+\Dirac[S]}$. We then get
		\begin{align*}
		\mathds{P}_{\Dirac}(A) p
		&= \int_{C_A} (\Dirac+z)^{-1}\, dz \, p \\
		&= \int_{C_A} (\Dirac+z)^{-1} p \, dz \\
		&= \int_{C_A} \underbracket[1pt]{(\Dirac+\Dirac[S]+z)^{-1} p}_{\text{by Eqn \eqref{restriction-eq} }} \, dz \\
		&=\int_{C_A} (\Dirac+\Dirac[S]+z)^{-1}) \, dz \, p \\
		&= \mathds{P}_{\Dirac + \Dirac[S]}(A) p \text.
		\end{align*}
		By $\sigma$-addivity, the spectral measures of $\Dirac$ and $\Dirac+\Dirac[S]$ thus satisfy $\mathds{P}_{\Dirac}(\cdot) p = \mathds{P}_{\Dirac+\Dirac[S]}(\cdot)p$.
		
		Now we denote by $\mathrm{R}([a,b])$ to be  the net of subdivisions of an interval $[a,b]$ with the usual ordering, 
		Since $\exp(i t \cdot)$ is a continuous function, 
		its integral over any interval  is a  limit of Riemann sums 
		and so  we obtain, by using the continuous functional calculus
		\begin{align*}
		\exp(i t (\Dirac+\Dirac[S])) p 
		&= \left(\int_\R \exp(i t s) \, d \mathds{P}_{\Dirac+\Dirac[S]}(s)\right) p \\
		&=\left(\lim_{x,y\rightarrow\infty} \int_{-y}^x \exp(it s)\, d\mathds{P}_{\Dirac+\Dirac[S]}(s)\right) p \\
		&= \lim_{x,y\rightarrow\infty} \lim_{s \in \mathcal{R}([-y,x])} \sum_{j=1}^{\# s} \exp(i t s_j)\mathds{P}_{\Dirac+\Dirac[S]}([s_j,s_{j+1}]) p \\
		&= \lim_{x,y\rightarrow\infty} \lim_{s \in \mathcal{R}([-y,x])} \sum_{j=1}^{\# s} \exp(i t s_j)\mathds{P}_{\Dirac}([s_j,s_{j+1}]) p \\
		&=\left(\lim_{x,y\rightarrow\infty} \int_{-y}^x \exp(it s)\, d\mathds{P}_{\Dirac}(s)\right) p \\
		&= \left(\int_\R \exp(i t s) \, d\mathds{P}_{\Dirac}(s)\right) p \\
		&= \exp(i t \Dirac) p \text,
		\end{align*}
		as needed.
	\end{proof}
	
	Our second lemma, Lemma \eqref{Fourier-estimate-lemma},  shows that, if $0$ is an isolated value in the spectrum of a self-adjoint operator, we can use Fourier analysis to get an estimate on the distance between any vector in the domain of the operator and its projection on the kernel of the operator, in terms of the graph norm of the operator and some well-chosen function. But first we need to recall some classical definitions.
	
	\begin{notation}
		A function on $\R$ is smooth when it is infinitely differentiable. We denote by $\mathcal{S}(\R)$ the space of \emph{Schwartz functions}, i.e. $f \in \mathcal{S}(\R)$ exactly when $f : \R\rightarrow\R$ is a smooth function such that for all $k,n$:
		\begin{equation*}
		\lim_{x\rightarrow \pm \infty} |1+x^n| \frac{d^k f}{dx^k}  = 0 \text.
		\end{equation*}
		It is clear from the above definition that  $\mathcal{S}(\R) \subseteq \bigcap_{p \geq 1} L^p(\R)$.
		
		If $f \in \mathcal{S}(\R)$, then we denote the Fourier transform of $f$ by
		\begin{equation*}
		\widehat{f} : t \in \R \mapsto \int_{-\infty}^\infty f(s) \exp(-2 i \pi t s) \, ds \text.
		\end{equation*}
		Note that $\widehat{f} \in \mathcal{S}(\R)$. With this particular convention, we get
		\begin{equation*}
		f : t \in \R \mapsto \int_{-\infty}^{\infty} \widehat{f}(s) \exp( 2i \pi t s) \, ds \text.
		\end{equation*}
	\end{notation}
	
	\begin{lemma}\label{Fourier-estimate-lemma}
		Let $\Dirac$ be a self-adjoint operator over $\Hilbert$, such that for some $\delta>0$, we have $\{ 0 \} = \spectrum{\Dirac}\cap(-\delta,\delta)$. Let $p$ be the orthonormal projection on the kernel $\ker\Dirac$ of $\Dirac$. Set $\CDN : \xi\in\dom{\Dirac}\mapsto \norm{\xi}{\Hilbert} + \norm{\Dirac\xi}{\Hilbert}$. If $f : \R \rightarrow \R$ is a smooth function with $f(0) = 1$, supported on $(-\delta,\delta)$, then for all $\xi \in \dom{\Dirac}$:
		\begin{equation*}
		\norm{\xi - p\xi}{\Hilbert} \leq 2 \CDN(\xi) \norm{f'}{L^2(\R)} \text.
		\end{equation*}
	\end{lemma}

	\begin{proof}
		Let $\delta > 0$ such that $\spectrum{\Dirac}\cap(-\delta,\delta) = \{ 0 \}$. Let $f : \R\rightarrow\R$ be a smooth function supported on $\left(-\delta,\delta\right)$ with $f(0) =1$ --- in particular, $f \in \mathcal{S}(\R)$. Thus, by the continuous functional calculus: 
		\begin{equation*}
		p = \int_\R \widehat{f}(t) \exp\left(2 i \pi t \Dirac\right) \, dt   \text.
		\end{equation*}
		Let now $\xi \in \dom{\Dirac}$.
		Since $f(0) = 1$, we have $\int_\R \widehat{f}(t)\,dt = f(0) = 1\text.$

		To be exceedingly formal, let us set
		\begin{equation*}
		h : t \in \R \mapsto \begin{cases}
		\frac{\norm{\xi-\exp(2i \pi t \Dirac)\xi}{\Hilbert}}{2\pi|t|} \text{ if $t\neq 0$,} \\
		\norm{\Dirac\xi}{\Hilbert} \text{ if $t=0$.}
		\end{cases}
		\end{equation*}
		We then compute:
		\begin{align}\label{Fourier-eq-1}
		\norm{\xi - p\xi}{\Hilbert} 
		&= \norm{\int_\R \widehat{f}(t)(\xi - \exp(2i\pi t \Dirac)\xi) \, dt}{\Hilbert} \\
		&\leq \int_\R |\widehat{f}(t) 2\pi t| h(t) \, dt \nonumber \\
		&= \int_\R |\widehat{f'}(t)| h(t) \, dt \nonumber \text.
		\end{align}
		
		Now, $\norm{\xi-\exp(2i\pi t\Dirac)\xi}{\Hilbert}\leq 2\norm{\xi}{\Hilbert} \leq 2\CDN(\xi)$ for all $t\in\R$. So:
		\begin{equation}\label{Fourier-eq-2}
		\int_{|t|>1}\frac{\norm{\xi-\exp(2 i \pi t\Dirac)\xi}{\Hilbert}^2}{4\pi^2 t^2}\, dt \leq \frac{1}{\pi^2} \CDN(\xi)^2\cdot 2 \leq \CDN(\xi)^2 \text.
		\end{equation}
		On the other hand, for all $t > 0$,
		\begin{align}\label{Fourier-eq-3}
		\norm{\exp(2i\pi t \Dirac)\xi - \xi}{\Hilbert}
		&=\norm{\exp(2i\pi t \Dirac)\xi - i\exp(2i\pi \cdot 0\Dirac)\xi}{\Hilbert} \\
		&=\norm{\int_0^t \frac{d}{ds}\exp(2 i \pi s \Dirac) \xi \, ds}{\Hilbert} \nonumber \\
		&=2\pi \norm{\int_0^t i\exp(2 i \pi s \Dirac)\Dirac \xi \, ds}{\Hilbert} \nonumber \\
		&\leq 2\pi \int_0^t \underbracket[1pt]{\norm{\Dirac \xi}{\Hilbert}}_{\text{since }\exp(i2\pi s\Dirac)\text{ is unitary}}\, ds \nonumber \\
		&\leq 2\pi t \norm{\Dirac \xi}{\Hilbert} \nonumber \text.
		\end{align}
		The reasoning above also applies for $t<0$ by replacing $\int_0^t$ with $\int_t^0$ in the above  equations.
		Therefore, for all $t\in \R\setminus\{0\}$, we have $\frac{\norm{\xi-\exp(it\Dirac)\xi}{\Hilbert}}{|2\pi t|}\leq \norm{\Dirac \xi}{\Hilbert}\leq \CDN(\xi)$. So for all $t \in\R$,
		\begin{equation*}
		0\leq h(t) \leq \norm{\Dirac\xi}{\Hilbert} \leq \CDN(\xi)\text,
		\end{equation*}
		and therefore:
		\begin{equation}\label{Fourier-eq-4}
		\int_{|t|\leq 1} h^2(t) \, dt \leq 2\norm{\Dirac \xi}{\Hilbert}^2 \leq 2 \CDN(\xi)^2 \text.
		\end{equation}
		It follows from Equations \eqref{Fourier-eq-2} and \eqref{Fourier-eq-4} that:
		\begin{equation*}
		\int_\R h^2(t)\, dt \leq 3 \CDN(\xi)^2 \leq 4\CDN(\xi)^2 \text.
		\end{equation*}
		Thus, we can use the Cauchy-Schwartz inequality to conclude from Equation \eqref{Fourier-eq-1} that:
		\begin{equation*}
		\norm{\xi-p \xi}{\Hilbert} \leq \norm{f'}{L^2(\R)} \sqrt{\int_\R h^2(t)\, dt} \text,
		\end{equation*}
		and therefore,
		\begin{equation*}
		\norm{\xi - p\xi}{\Hilbert}\leq 2 \norm{f'}{L^2(\R)} \CDN(\xi) \text,
		\end{equation*}
		as claimed.
	\end{proof}
	
\medskip
	
	We now prove our main theorem. We start from a metric spectral triple 
	$(\A,\Hilbert,\Dirac)$ 
	whose Dirac operator $\Dirac$ can be decomposed as a sum of two operators, seen as a ``horizontal'' and a ``vertical'' component, $\Dirac_h$  and $\Dirac_v$  respectively. What makes them horizontal and vertical is the list of our assumptions, and is related to the existence of some map from $\A$ onto a C*-subalgebra $\B$ with $1\in \B$, as well as  its interplay with the decomposition of   $\Dirac$ especially in principal $G$-bundles examples.  The ``horizontal'' part can be used to define a Dirac operator and an associated spectral triple on  $\B$. When the $\varepsilon$-perturbed vertical part is made to collapse i.e.  $\varepsilon \to 0$, 
	the perturbed spectral triple tends to the \lq \lq horizontal" spectral triple in the spectral propinquity.
	
	\begin{theorem}\label{main-thm}
		Let $(\A,\Hilbert,\Dirac)$ be a metric spectral triple, and let $\B \subseteq\A$ be a unital C*-subalgebra of $\A$, and such that $\Dirac = \Dirac_h + \Dirac_v$, where $\Dirac_v$ is self-adjoint and such that $0$ is isolated in $\spectrum{\Dirac}$, together with the following assumptions. Setting $\Dirac_\varepsilon \coloneqq \Dirac_h + \frac{1}{\varepsilon}\Dirac_v$ for all $\varepsilon \in (0,1)$, the triple $(\A,\Hilbert,\Dirac_\varepsilon)$ is a spectral triple, such that:
		\begin{enumerate}
			\item for all $\varepsilon \in (0,1)$, 
			\begin{equation}\label{main-thm-eq-1}
			\opnorm{[\Dirac_h,a]}{}{\Hilbert}\leq \opnorm{[\Dirac_\varepsilon,a]}{}{\Hilbert} \text,
			\end{equation}
			\item there exists $M > 0$ such that for all $\varepsilon \in (0,1)$,
			\begin{equation}\label{main-thm-eq-3}
			\opnorm{[\Dirac_v,a]}{}{\Hilbert} \leq M\, \varepsilon\, \opnorm{[\Dirac_\varepsilon,a]}{}{\Hilbert} \text,
			\end{equation}
			\item $[\Dirac_v,b] = 0$ for all $b \in \B$,
			\item writing $p$ for the projection onto $\ker \Dirac_v$, we assume that $[p,b] = 0$ for all $b \in \B$ and $[\Dirac_h,p] = 0$,
			\item $(\B,\ker\Dirac_v, \Dirac_\B)$, where $ \Dirac_\B \coloneqq p\Dirac_h p$,  is a metric spectral triple,
			\item there exists a positive linear map $\mathds{E} : \A \rightarrow \B$, whose restriction of $\mathds{E}$ to $\B$ is the identity, and a constant $k>0$ such that for all $a\in\A$:
			\begin{equation}\label{main-thm-eq-5}
			\norm{ a - \mathds{E}(a) }{\A} \leq k \opnorm{ [\Dirac_v,a] }{}{\Hilbert} \text,
			\end{equation}
			and
			\begin{equation}\label{main-thm-eq-6}
			\opnorm{p[\Dirac_h,\mathds{E}(a)]p}{}{\Hilbert} = \opnorm{[\Dirac_h,\mathds{E}(a)]}{}{\Hilbert} \leq \opnorm{[\Dirac_h,a]}{}{\Hilbert} \text.
			\end{equation}
		\end{enumerate}
		Then $(\A,\Hilbert,\Dirac_\varepsilon)$ is a metric spectral triple, and:
		\begin{equation*}
		\lim_{\varepsilon\rightarrow 0^+} \spectralpropinquity{}((\A,\Hilbert,\Dirac_\varepsilon),(\B,\ker\Dirac_v,\Dirac_\B)) = 0 \text.
		\end{equation*}
	\end{theorem}
	
	\begin{proof}
		For convenience, we write $\Dirac_\B \coloneqq p \Dirac_h p$. We also denote the dense *-subalgebra $\{ a\in \A : a\dom{\Dirac}\subseteq\dom{\Dirac}\}$ by $\dom{\delta}$, and we write $\dom{\Lip_\A} \coloneqq \dom{\delta}\cap\sa{\A}$. We set $\Lip_\varepsilon(a) \coloneqq \opnorm{[\Dirac_\varepsilon,a]}{}{\Hilbert}$  for all $a\in\dom{\Lip_\A}$, with $\Lip_\A\coloneqq\Lip_1$, and $\Lip_\B(b) \coloneqq \opnorm{[\Dirac_\B,b]}{}{\ker\Dirac_v}$.
		
		We begin by checking that $(\A,\Hilbert,\Dirac_\varepsilon)$ is a metric spectral triple for all $\varepsilon \in (0,1)$. Let $a\in\dom{\delta}$. We observe that $\Dirac_1=\Dirac$ and that, by assumption, $(\A,\Hilbert,\Dirac)$ is a metric spectral triple; moreover, by applying Expressions \eqref{main-thm-eq-1} and \eqref{main-thm-eq-3}, we get:
		\begin{align}\label{comp-eq}
		L_1(a)=\opnorm{[\Dirac,a]}{}{\Hilbert}
		&\leq \opnorm{[\Dirac_h,a]}{}{\Hilbert}+\opnorm{[\Dirac_v,a]}{}{\Hilbert} \\
		&\leq \opnorm{[\Dirac_\varepsilon,a]}{}{\Hilbert} + M \varepsilon \opnorm{\Dirac_\varepsilon,a]}{}{\Hilbert} \nonumber \\
		&\leq (1+M \varepsilon ) \opnorm{[\Dirac_\varepsilon,a]}{}{\Hilbert}  = (1+M \varepsilon ) L_\varepsilon(a) \nonumber \text.
		\end{align}
		By \cite[Lemma 1.10]{Rieffel98a}, we therefore conclude that since $(\A,\Hilbert,\Dirac)$ is a metric spectral triple, so is $(\A,\Hilbert,\Dirac_\varepsilon)$.
		
		\medskip
		
		We now fix $\varepsilon \in (0,1)$. We will construct a tunnel between $(\A,\Lip_\varepsilon)$ and $(\B,\Lip_\B)$. We set 
		\begin{equation*}
			\D\coloneqq \A\oplus \B\text, 
		\end{equation*}	
		and for  all $(a,b) \in \dom{\Lip_\A}\times\dom{\Lip_\B}$ define
		\begin{equation*}
		\Lip[T]_\varepsilon(a,b) \coloneqq \max\left\{ \opnorm{[\Dirac_\varepsilon,a]}{}{\Hilbert}, \opnorm{[\Dirac_\B,b]}{}{\ker\Dirac_v}, \frac{1}{k\, M\, \varepsilon}\norm{a-b}{\A}  \right\} \text.
		\end{equation*}
		
		It is immediate to see that the domain of $\Lip[T]_\varepsilon$ is dense in $\sa{\D}$, that $\Lip[T]_\varepsilon$ satisfies the Leibniz inequality of  Definition (\ref{qcms-def}), and that $\Lip[T]_\varepsilon(a,b) = 0$ implies $a=b=t1$ for some $t\in \R$.
		To show that $\Lip[T]_\varepsilon$ is an L-seminorm we will use  \cite{Rieffel98a} (see also \cite[Theorem 2.1]{Rieffel05}), 
		which guarantees that the induced topology by $\Lip[T]_\varepsilon$ on the state space $\StateSpace(\D)$ is the weak* topology. 
		For, fix  $\mu \in \StateSpace(\A)$. For all $a\in\dom{\Lip_\A}$ and for all $\varphi\in\StateSpace(\A)$, we have by for example by \cite[Theorem 2.8]{Latremoliere23a}; see also \cite{Rieffel98a}
		\begin{equation*}
			|\varphi(a-\mu(a))| = |\varphi(a) - \mu(a)| \leq \Kantorovich{\Dirac}(\varphi,\mu) \opnorm{[\Dirac,a]}{}{\Hilbert} \leq \qdiam{\A}{\Dirac} \opnorm{[\Dirac,a]}{}{\Hilbert} \text,
		\end{equation*}
		where $\qdiam{\A}{\Dirac}$ is the diameter of the metric space $(\StateSpace(\A),\Kantorovich{\Dirac})$.
		
		 Therefore, since $a\in\sa{\A}$, we concude that $\norm{a-\mu(a)}{\A} \leq \qdiam{\A}{\Dirac}\opnorm{\Dirac,a]}{}{\Hilbert}$,  Thus, if $\mu(a) = 0$, then 
		\begin{equation*}
		\norm{a}{\A} \leq \qdiam{\A}{\Dirac} \Lip_1(a)\leq \qdiam{\A}{\Dirac}\underbracket[1pt]{(1+M\varepsilon)\Lip_\varepsilon(a)}_{\text{ by Eqn. \eqref{comp-eq}}}\text.
		\end{equation*}
		Therefore, if $\Lip[T]_\varepsilon(a,b)\leq 1$, then $\norm{b}{\A}\leq\norm{a-b}{\A}+\norm{a}{\A} \leq k M \varepsilon + (1+M) \qdiam{\A}{\Dirac}$. In summary, we have proven the inclusion: 
		\begin{multline*}
		\left\{ (a,b) \in \D : \Lip[T]_\varepsilon(a,b) \leq 1, \mu(a) = 0 \right\} \subseteq
		\left\{ a \in \dom{\Lip_\A} : \opnorm{[\Dirac_\varepsilon,a]}{}{\Hilbert}\leq 1, \mu(a) = 0 \right\} \\ \times \left\{ b \in \dom{\Lip_\B} : \Lip_\B(b) \leq 1, \norm{b}{\A} \leq  k M \varepsilon + (1+M)\qdiam{\A}{\Dirac} \right\} \text.
		\end{multline*}
		Since the set on the right hand side is compact as the product of two compact sets by \cite{Rieffel98a}, and since the set on the left hand side is closed since $\Lip[T]_\varepsilon$ is lower-semicontinuous on $\sa{\D}$ (as the maximum of three lower semi-continuous functions over $\sa{\D}$), we conclude that the set on the left hand side is compact, and hence that $\Lip[T]_\varepsilon$ is an L-seminorm, by \cite{Rieffel98a}.

		We now prove that the *-epimorphism $j_\A : (a,b)\in\D\mapsto a$ is a quantum isometry. Let $a \in \dom{\Lip_\A}$ with $\Lip_\varepsilon(a)=1$. Then by Expressions (\ref{main-thm-eq-5}) and (\ref{main-thm-eq-3}),
		\begin{align*}
		\norm{a-\mathds{E}(a)}{\A} 
		&\leq k \opnorm{[\Dirac_v,a]}{}{\Hilbert} \\
		&\leq k \varepsilon M \opnorm{[\Dirac_\varepsilon,a]}{}{\Hilbert} = k M \varepsilon \text.
		\end{align*}
		Moreover, again by Expressions  (\ref{main-thm-eq-6}) and \eqref{main-thm-eq-1}, we have 
		\begin{equation*}
			\opnorm{[\Dirac_h,\mathds{E}(a)]}{}{\Hilbert} \leq \opnorm{[\Dirac_h,a]}{}{\Hilbert}\leq \opnorm{[\Dirac_\varepsilon,a]}{}{\Hilbert}\leq 1 \text.
		\end{equation*}
	 Since $p$ commutes with $\mathds{E}(a)$ and with $\Dirac_h$, we estimate that
		\begin{equation*}
		\opnorm{[p\Dirac_h p,\mathds{E}(a)]}{}{\Hilbert} = \opnorm{p[\Dirac_h,\mathds{E}(a)]p}{}{\Hilbert} = \opnorm{[\Dirac_h,\mathds{E}(a)]}{}{\Hilbert} \leq 1\text.
		\end{equation*}
		Therefore, $\Lip[T]_\varepsilon(a,\mathds{E}(a)) \leq 1$.
		
		Since for all $b \in \dom{\Lip_\B}$,  $\Lip[T]_\varepsilon(a,b)\geq \Lip_\varepsilon(a)=1$ by construction , we conclude that $\Lip[T]_\varepsilon(a,b)=1$ and that $j_\A$ is indeed a quantum isometry.
		
	We now prove that $j_\B : (a,b) \in \D \mapsto b$ is also a quantum isometry. Again, for all $(a,b) \in \dom{\Lip[T]_\varepsilon}$, by definition, $\Lip[T]_\varepsilon(a,b) \geq \Lip_\B(b)$\text, $\norm{b-b}{\A} = 0$ and by  Assumption (4) we have $L_\B(b)=\opnorm{[\Dirac_\B,b]}{}{\Hilbert}= \opnorm{[\Dirac_\B,b]}{}{\ker(\Dirac_v)}= \opnorm{p[\Dirac_h,b]p}{}{\Hilbert}$, so $\Lip[T]_\varepsilon(b,b) = \Lip_\B(b)$, which shows that  $j_\B$ is a quantum isometry.
		
		Hence $\tau_\varepsilon \coloneqq (\D=\A\oplus\B,\Lip[T]_\varepsilon,j_\A,j_\B)$ is a tunnel from $(\A,\Lip_\varepsilon)$ to $(\B,\Lip_\B)$. We now bound from above the  extent of $\tau_\varepsilon$. Let $\varphi \in \StateSpace(\D)$ and  define $\psi:a\in\A \mapsto \varphi(a,\mathds{E}(a))$. By construction, $\psi \in \StateSpace(\A)$. If $(a,b) \in \A\oplus\B=\D$ and $\Lip[T]_\varepsilon(a,b)\leq 1$, then:
		\begin{align*}
		|\varphi(a,b) - \psi(a)|
		&= |\varphi(a,b) - \varphi(a,\mathds{E}(a))| \\
		&= |\varphi(0,b-\mathds{E}(a))|\\
		&= |\varphi(0,\mathds{E}(a-b))| \\
		&\leq \norm{\mathds{E}(a-b)}{\A} \\
		&\leq \underbracket[1pt]{\norm{a-b}{\A}}_{\opnorm{E}{}{\A}\leq 1} \leq k M \varepsilon \text.
		\end{align*}
		Thus $\Kantorovich{\Lip[T]_\varepsilon}(\varphi,\psi)\leq k M \varepsilon$, and so $\Haus{\Kantorovich{\Lip[T]_\varepsilon}}(\StateSpace(\D=\A\oplus\B),j_\A^\ast(\StateSpace(\A)))\leq k M \varepsilon\text.$
		
		On the other hand, fixed $\varphi \in \StateSpace(\D)$, let $\theta : b \in \B \mapsto \varphi(b,b)$. Again, $\theta\in\StateSpace(\B)$. If $(a,b)\in\A\oplus\B=\D$ and $\Lip[T]_\varepsilon(a,b)\leq 1$, then
		\begin{align*}
		|\varphi(a,b) - \theta(b)|
		&=|\varphi(a,b) - \varphi(b,b)| = |\varphi(a-b,0)| \leq \norm{a-b}{\A} \leq k M \varepsilon \text.
		\end{align*}
		Thus $\Kantorovich{\Lip[T]_\varepsilon}(\varphi,\theta)\leq k\varepsilon$ and so $\Haus{\Kantorovich{\Lip[T]_\varepsilon}}(\StateSpace(\D=\A\oplus\B),j_\B^\ast(\StateSpace(\B)))\leq k M \varepsilon\text.$
		
		We have thus established that $\tunnelextent{\tau_\varepsilon} \leq k M \varepsilon$.
		
		\medskip
		
		We now build a  tunnel  which will be the precursor to a metrical tunnel between the  C*-correspondences induced by the spectral triples $(\A,\Hilbert,\Dirac_\varepsilon)$ and $(\B,\ker\Dirac_v,\Dirac_\B)$. Let $\delta>0$ be chosen so that $\spectrum{\Dirac_v}\cap(-\delta,\delta)=\{0\}$. Of course, we also have that  $\left(-\frac{\delta}{\varepsilon},\frac{\delta}{\varepsilon}\right)\cap \spectrum{\frac{1}{\varepsilon}\Dirac_v} = \{ 0 \}$. Choose a smooth function $f_\varepsilon : \R \rightarrow\R$, supported on $\left(-\frac{\delta}{\varepsilon},\frac{\delta}{\varepsilon}\right)$, such that $f_\varepsilon(0) = 1$ and $|f_\varepsilon'(t)|\leq \frac{2\varepsilon}{\delta}$. As a consequence of our choice, we get the following estimate:
		\begin{equation*}
		\norm{f'}{L^2(\R)} = \sqrt{\int_\R |f'(t)|^2\,dt} \leq \sqrt{\int_\R \frac{4\varepsilon^2}{\delta^2}\, dt} = \sqrt{\frac{2\delta}{\varepsilon}\frac{4\varepsilon^2}{\delta^2}} = 2\sqrt{\frac{2\varepsilon}{\delta}} \text.
		\end{equation*}
		
		By Lemma (\ref{Fourier-estimate-lemma}), it follows that for all $\xi \in \dom{\Dirac_v}$,
		\begin{equation}\label{main-eq-1}
		\norm{\xi - p \xi}{\Hilbert} \leq 4\sqrt{\frac{2\varepsilon}{\delta}}\left(\norm{\xi}{\Hilbert} + \norm{\frac{1}{\varepsilon}\Dirac_v\xi}{\Hilbert} \right) \leq 4\sqrt{\frac{2\varepsilon}{\delta}}\CDN_\varepsilon(\xi) \text.
		\end{equation} 
		
		We now set 
		\begin{equation*}
		K_\varepsilon \coloneqq \max\left\{k\, M\, \varepsilon,4\sqrt{\frac{2\varepsilon}{\delta}} \right\}\text,
		\end{equation*}
		and for all $(\xi,\eta)\in\dom{\Dirac_h}\oplus\ker\Dirac_v$:
		\begin{equation*}
		\CDN[TN]_\varepsilon(\xi,\eta) \coloneqq \max\left\{ \CDN_\varepsilon(\xi), \CDN_\B(\eta), \frac{1}{K_\varepsilon} \norm{p \xi - \eta}{\Hilbert} \right\} \text.
		\end{equation*}
		
		We now check the modular Leibniz inequality for $\CDN[TN_\varepsilon]$. If $(a,b) \in \dom{\delta}\times\dom{\Dirac_\B}\subseteq\A\otimes\B$, $\xi \in \dom{\Dirac}$,  and $\zeta\in \ker{\Dirac_v}$, we have:
		\begin{align*}
		\norm{p a \xi - b  \zeta}{{\Hilbert}}
		&\leq \norm{p a \xi - b p \xi}{\Hilbert} + \norm{b p\xi- b \zeta}{\Hilbert} \\
		&\leq \norm{p a \xi - p b \xi}{\Hilbert} + \norm{b}{\B}\norm{p\xi- \zeta}{\Hilbert} \\
		&\leq \norm{a-b}{\A}\norm{\xi}{\Hilbert} + \norm{b}{\B} 4\sqrt{\frac{2\varepsilon}{\delta}}\CDN[TN]_\varepsilon(\xi, \zeta) \\
		&\leq K_\varepsilon \Lip[T]_\varepsilon(a,b)\CDN_\varepsilon(\xi) + \norm{b}{\B} 4\sqrt{\frac{2\varepsilon}{\delta}}\CDN[TN]_\varepsilon(\xi, \zeta) \\
		&\leq  \max\left\{ K_\varepsilon, 4 \sqrt{\frac{2\varepsilon}{\delta}} \right\} \left(\Lip[T]_{\varepsilon}(a,b) + \norm{(a,b)}{\D}\right) \CDN[TN]_\varepsilon(\xi, \zeta) \text.
		\end{align*}
		Therefore, using the modular Leibniz inequalities for $\CDN_\varepsilon$ and $\CDN_\B$, we conclude:
		\begin{equation*}
		\CDN[TN]_\varepsilon( (a,b)\cdot(\xi, \zeta) ) \leq  \left(\Lip[T]_{\varepsilon}(a,b) + \norm{(a,b)}{\D}\right) \CDN[TN]_\varepsilon(\xi, \zeta) \text.
		\end{equation*}

		By Lemma (\ref{Fourier-estimate-lemma}) applied to our $f_\varepsilon$ and to $\frac{1}{\varepsilon}\Dirac_v$, if $\xi \in \Hilbert$, then $\norm{\xi-p \xi}{\Hilbert} \leq 4\sqrt{\frac{2\varepsilon}{\delta}} \CDN_\varepsilon(\xi)$ by Equation \eqref{main-eq-1}. Moreover, since $\Dirac_h$ commutes with $p$ and of course, $p$ commutes with $\Dirac_v$ and $\Dirac_v p = 0$, we conclude that $p$ commutes with $\Dirac_\varepsilon$ and therefore:
		\begin{equation*}
			\norm{\Dirac_\B p\xi}{\Hilbert_\B} = \norm{p \Dirac_h p \xi}{\Hilbert} = \norm{(\Dirac_h + \frac{1}{\varepsilon}\Dirac_v) p\xi}{\Hilbert} = \norm{\Dirac_\varepsilon p\xi}{\Hilbert} \leq \norm{\Dirac_\varepsilon \xi}{\Hilbert}\text.
		\end{equation*} 
		So, altogether, we get:
		\begin{equation*}
		\CDN[TN]_\varepsilon(\xi,p\xi) = \CDN_\varepsilon(\xi) \text.
		\end{equation*}
		
		For all $\eta\in\ker\Dirac_v$,   $p\eta=\eta$, and since 
		$[\Dirac_h,p]=0$ and $p^2=p$, we have

		\begin{equation*}
			\norm{\Dirac_\varepsilon\eta}{\Hilbert} = \norm{\Dirac_h \eta+ \frac{1}{\varepsilon}\Dirac_v \eta}{\Hilbert} = \norm{\Dirac_h \eta}{\Hilbert} = \norm{p\Dirac_h p \eta}{\Hilbert},
		\end{equation*}
		which implies $\CDN[TN]_\varepsilon(\eta,\eta) = \CDN_\B(\eta)$.

		 Hence, the maps $J_\A : (\xi,\eta)\in\Hilbert\oplus\ker\Dirac_v\mapsto \xi$ and $J_\B:(\xi,\eta)\in\Hilbert\oplus\Dirac_v\mapsto \eta$ are both quantum isometries. Therefore, 
		\begin{equation}\label{eq:mod tunn}
		\left( \left( \Hilbert\oplus\ker\Dirac_v, \CDN[TN]_\varepsilon, \A\oplus\B, \Lip[T]_\varepsilon \right), (j_\A,J_\A), (j_\B,J_\B) \right)
		\end{equation}
		is a tunnel. By definition, its extent is the extent of $(\D,\Lip[T]_\varepsilon,j_\A,j_\B)$, which is no more than $k M \varepsilon$.
		
		We now turn the  tunnel in Equation \eqref{eq:mod tunn} into a metrical one. The space $\Hilbert\oplus\ker\Dirac_v$ is naturally a $\C\oplus\C$ Hilbert module, by setting, for all $(\xi,\eta),(\xi',\eta') \in \Hilbert\oplus\ker\Dirac_v$
		\begin{equation*}
		 \inner{(\xi,\eta)}{(\xi',\eta')}{\C\oplus\C} = \left( \inner{\xi}{\xi'}{\Hilbert}, \inner{\eta}{\eta'}{\Hilbert} \right) \text,
		\end{equation*}
		and, 	forall $(\xi,\eta)\in\Hilbert\oplus\ker\Dirac_v $ and forall $(z,w) \in \C\oplus\C $
		\begin{equation*}
	(\xi,\eta)(z,w) = (z\xi,w\eta) \text.
		\end{equation*}
		Now, we define $\alpha_\varepsilon \coloneqq 4\sqrt{\frac{2\varepsilon}{\delta}}$ and:
		\begin{equation*}
		\Lip[Q]_\varepsilon : (z,w) \in \C\oplus\C \mapsto (\alpha_\varepsilon + K_\varepsilon)^{-1} |z-w| \text.
		\end{equation*}
		We then get, for all $\xi,\eta\in\dom{\Dirac_\varepsilon}$, $\zeta,\omega\in\ker\Dirac_v\cap\dom{\Dirac_h}$:
		\begin{align*}
		\Big|\inner{\xi}{\eta}{\Hilbert} &- \inner{\zeta}{\omega}{\Hilbert}\Big| \\
		&\leq\left|\inner{\xi-p\xi}{\eta}{\Hilbert}\right| + \left|\inner{p\xi}{\eta}{\Hilbert} - \inner{\zeta}{\omega}{\Hilbert}\right| \\
		&\leq\norm{\xi-p\xi}{\Hilbert}\norm{\eta}{\Hilbert} + \left|\inner{p\xi-\zeta}{\eta}{\Hilbert}\right| + \left|\inner{\zeta}{\eta}{\Hilbert}-\inner{\zeta}{\omega}{\Hilbert}\right|\\
		&\leq \alpha_\varepsilon\CDN_\varepsilon(\xi)\CDN_\varepsilon(\eta) + K_\varepsilon\, \CDN[TN]_\varepsilon(\xi,\zeta)  \norm{\eta}{\Hilbert} + \left|\inner{\zeta}{\eta- p \eta}{\Hilbert}\right| + \left|\inner{\zeta}{p\eta-\omega}{\Hilbert}\right|  \\
		&\leq \alpha_\varepsilon\CDN[TN]_\varepsilon(\xi, \zeta)\CDN[TN]_\varepsilon(\eta, \omega) + K_\varepsilon\,\CDN[TN]_\varepsilon(\xi,\zeta)\CDN[TN]_\varepsilon(\eta, \omega)  + \left|\inner{\zeta}{\eta- p \eta}{\Hilbert}\right| + \left|\inner{\zeta}{p \eta-\omega}{\Hilbert}\right|  \\
		&\leq (\alpha_\varepsilon+K_\varepsilon)\CDN[TN]_\varepsilon(\eta,\omega)\CDN[TN]_\varepsilon(\xi,\zeta) + \norm{\zeta}{\Hilbert}\norm{\eta- p \eta}{\Hilbert} + \norm{\zeta}{\Hilbert}\norm{p \eta-\omega}{\Hilbert} \\
		&\leq (\alpha_\varepsilon+K_\varepsilon)\CDN[TN]_\varepsilon(\eta,\omega)\CDN[TN]_\varepsilon(\xi,\zeta) + \alpha_\varepsilon \CDN[TN]_\varepsilon(\xi,\zeta) \CDN[DN]_\varepsilon(\eta) + K_\varepsilon\CDN[TN]_\varepsilon(\xi,\zeta)   \CDN[TN]_\varepsilon(\eta,\omega) \\
		&\leq (\alpha_\varepsilon+K_\varepsilon)\CDN[TN]_\varepsilon(\eta,\omega)\CDN[TN]_\varepsilon(\xi,\zeta) + (\alpha_\varepsilon+K_\varepsilon)\CDN[TN]_\varepsilon(\xi,\zeta)\CDN[TN]_\varepsilon(\eta,\omega) \\
		&= (\alpha_\varepsilon+K_\varepsilon) \cdot 2 \CDN[TN]_\varepsilon(\xi,\zeta)\CDN[TN]_\varepsilon(\eta,\omega) \text.
		\end{align*}
		So
		\begin{equation*}
		\Lip[Q]_\varepsilon\left(\inner{(\xi,\zeta)}{(\eta,\omega)}{\C\oplus\C}\right)\leq 2 \CDN[TN]_\varepsilon(\xi,\zeta)\CDN[TN]_\varepsilon(\eta,\omega)\text.
		\end{equation*}
		Now define $t_\A : (z,w)\in\C\oplus\C \mapsto z$ and $t_\B:(z,w)\in\C\oplus\C\mapsto w$. It is straightforward to check that $t_\A$ and $t_\B$ are quantum isometries onto $(\C,0)$. Therefore 
		\begin{equation*}
		\tau(\varepsilon):=\left( \left(\Hilbert\oplus\ker\Dirac_v,\CDN[TN]_\varepsilon,\D,\Lip[T]_\varepsilon,\C\oplus\C,\Lip[Q]_\varepsilon\right), (J_\A,j_\A,t_\A), (J_\B,j_\B,t_\B) \right)
		\end{equation*}	
		is a metrical tunnel. The extent of $(\C\oplus\C,\Lip[Q]_\varepsilon,t_\A,t_\B)$ is at most $(\alpha_\varepsilon+K_\varepsilon)$. Hence, the extent of $\tau(\varepsilon)$ is $\max\left\{ k M \varepsilon, (\alpha_\varepsilon+K_\varepsilon)\right\}$.

		We conclude by considering the action of $[0,\infty)$ by the unitaries $t\in[0,\infty)\mapsto\exp(it\Dirac_\varepsilon)$ and $t\in[0,\infty)\mapsto\exp(it p \Dirac_h p)$ to estimate the spectral propinquity.

		Let first $\xi \in \dom{\Dirac}$. By Lemma (\ref{restriction-unitary-lemma}), we observe that $\exp(it\Dirac_\varepsilon) p = \exp(i t p \Dirac_h p) = p \exp(i t \Dirac_h ) p$, and:  
		\begin{align*}
		\norm{\exp(i t \Dirac_{\varepsilon})\xi - \exp(i t p \Dirac_h p)\xi}{\Hilbert}
		&= \norm{\exp(i t \Dirac_{\varepsilon})\xi-\exp(i t \Dirac_\varepsilon) p \xi}{\Hilbert} \\
		&= \norm{\exp(i t \Dirac_{\varepsilon})(\xi - p\xi)}{\Hilbert} \\
		&\leq \underbracket[1pt]{\norm{\xi-p\xi}{\Hilbert}}_{\exp(it\Dirac_{\varepsilon})\text{ unitary}} \\
		&\leq 4\sqrt{\frac{2\varepsilon}{\delta}}\CDN(\xi) \text.
		\end{align*}
		
		Therefore, we get that, for all $\omega\in\dom{\CDN_\varepsilon}$ with $\CDN_\varepsilon(\omega)\leq 1$, for all $\xi \in\dom{\CDN}$ with $\CDN(\xi)\leq 1$, and  for all $t\geq 0$:
		\begin{align*}
		&\quad \left|\inner{\exp(it\Dirac_{\varepsilon})\xi}{J_\A(\omega)}{}- \inner{\exp(i t p \Dirac_h p) p\xi}{J_\B(\omega)}{}\right| \\
		&=\left|\inner{\exp(it\Dirac_{\varepsilon})\xi- \exp(it p \Dirac_h p)\xi)}{J_\A(\omega)}{ \Hilbert} + \inner{\exp(i t p\Dirac_h p)\xi}{J_\A(\omega)-J_\B(\omega)}{\Hilbert}\right| \\
		&=\left|\inner{\exp(it\Dirac_{\varepsilon})(\xi- p\xi)}{J_\A(\omega)}{ \Hilbert} + \inner{p\exp(i t \Dirac_h )p\xi}{J_\A(\omega)-J_\B(\omega)}{\Hilbert}\right| \\
		&=\left|\inner{\exp(it\Dirac_{\varepsilon})(\xi-p\xi)}{J_\A(\omega)}{ \Hilbert} + \underbracket[1pt]{\inner{\exp(i t \Dirac_h) p\xi}{p J_\A(\omega)-p J_\B(\omega)}{\Hilbert}}_{p=p^2=p^\ast}\right| \\
		&\leq \left|\inner{\exp(it\Dirac_{\varepsilon} )(\xi-p \xi)}{J_\A(\omega)}{ \Hilbert}\right|
		+  \left|\inner{\exp(i t \Dirac_h)p \xi}{p J_\A(\omega)-\underbracket[1pt]{J_\B(\omega)}_{p J_\B = J_\B}}{\Hilbert}\right| \\
		&\leq \norm{\exp(i t \Dirac_\varepsilon)(\xi- p \xi)}{\Hilbert}\norm{J_\A(\omega)}{\Hilbert} + \norm{\xi}{\Hilbert}\norm{p J_\A(\omega)-J_\B(\omega)}{ \Hilbert} \\
		&\leq  \alpha_\varepsilon \CDN_\varepsilon(\xi)  \CDN[TN]_\varepsilon(\omega) + \CDN_\varepsilon(\xi) \cdot K_\varepsilon \CDN[TN]_\varepsilon(\omega) \\
		&\leq \alpha_\varepsilon + K_\varepsilon \text{.}
		\end{align*}
		Therefore, for each $\xi \in \dom{\Dirac_\varepsilon}$ with $\CDN_\varepsilon(\xi)\leq 1$, there exists $\eta\coloneqq p \xi$ such that for all $t\in \R$,
		\begin{equation*}
		\sup_{\substack{\omega\in\dom{\CDN[TN]_\varepsilon} \\ \CDN[TN]_\varepsilon(\omega)\leq 1}}\left|\inner{\exp(it\Dirac_{\varepsilon})\xi}{J_\A(\omega)}{}- \inner{\exp(i t p \Dirac_h p) \eta}{J_\B(\omega)}{}\right| \leq \alpha_\varepsilon + K_\varepsilon \text.
		\end{equation*}
		Similarly, if $\eta \in \dom{p \Dirac_h p}$ with $\CDN_\B(\eta)\leq 1$, then by setting $\xi = \eta$, we also obtain for all $t\in \R$:
		\begin{equation*}
		\sup_{\substack{\omega\in\dom{\CDN[TN]_\varepsilon} \\ \CDN[TN]_\varepsilon(\omega)\leq 1}}\left|\inner{\exp(it\Dirac_{\varepsilon})\xi}{J_\A(\omega)}{}- \inner{\exp(i t p \Dirac_h p) \eta}{J_\B(\omega)}{}\right| \leq K_\varepsilon + \alpha_\varepsilon \text.
		\end{equation*}
		
		So, for any $t \geq 0$, the $\frac{1}{t}$-covariant reach of $\tau(\varepsilon)$ is at most $K_\varepsilon + \alpha_\varepsilon$. We also  note that $\lim_{\varepsilon\rightarrow 0} (\alpha_\varepsilon + K_\varepsilon) = 0$. If we set  $M_\varepsilon \coloneqq \alpha_\varepsilon + K_\varepsilon$
		we  have thus shown that 
		\begin{equation*}
		\tunnelsep{\tau(\varepsilon)}{(\exp(i t \Dirac_\varepsilon))_{0\leq t \leq M_\varepsilon}, (\exp(i t \Dirac_\B))_{0\leq t \leq M_\varepsilon}}\leq M_\varepsilon \text.
		\end{equation*}
		Hence $\tunneldispersion{\tau(\varepsilon)}{(\exp(i t \Dirac_\varepsilon))_{0\leq t \leq M_\varepsilon}, (\exp(i t \Dirac_\B))_{0\leq t \leq M_\varepsilon}}\leq M_\varepsilon$ since $\tunnelextent{\tau(\varepsilon)}\leq K_\varepsilon \leq M_\varepsilon$. Therefore:
		\begin{equation*}
		0\leq \spectralpropinquity{}((\A,\Hilbert,\Dirac_\varepsilon),(\B,\ker\Dirac_v,\Dirac_\B)) \leq M_\varepsilon \xrightarrow{\varepsilon\rightarrow 0} 0 \text.
		\end{equation*}
		Our proof is complete. 
	\end{proof}
	
	A consequence of Theorem (\ref{main-thm}) is the following result on the convergence of the bounded continuous functional calculus, and as a corollary, of the spectra of the Dirac operators.
	
	\begin{corollary}\label{main-cor}
		Under the same hypothesis as in Theorem (\ref{main-thm}), for any $f \in C_b(\R)$, and for any sequence $(\varepsilon_n)_{n\in\N}$ in $(0,\infty)$ converging to $0$, we have
		\begin{equation*}
		\lim_{n\rightarrow\infty} \oppropinquity{}(f(\Dirac_{\varepsilon_n}),f(\Dirac_\B)) = 0 \text,
		\end{equation*}
		and in particular,
		\begin{equation*}
		\spectrum{\Dirac_\B} = \left\{ \lim_{n\rightarrow \infty} \lambda_n : (\lambda_n)_{n\in\N} \text{ convergent with }\forall n \in \N \; \lambda_n \in \spectrum{\Dirac_{\varepsilon_n}} \right\}\text.
		\end{equation*}
	\end{corollary}
	
		The following result, Proposition   \eqref{comparison-lemma}, will be used to compute  examples. We place it here to explain one of the assumption of Theorem (\ref{main-thm}), namely that we can compare the norms of the ``derivatives'' in the horizontal and vertical directions to the ``total'' derivative. This indeed happens in  common constructions of Dirac operators.
	
	\begin{proposition}\label{comparison-lemma}
		Let $\Hilbert$ be a Hilbert space, and let $\gamma_1,\ldots,\gamma_d$ be $d$ anticommuting self-adjoint unitaries on $\Hilbert$. Let $\A$ be a C*-algebra of operators acting on $\Hilbert$, such that $\gamma_j$ commutes with $\A$ for all $j\in\{1,\ldots,d\}$.
		
		If $\Dirac\coloneqq \sum_{j=1}^d \Dirac_j \gamma_j$, where $\Dirac_1$,\ldots,$\Dirac_d$ are possibly unbounded operators defined on a dense subspace $\dom{\Dirac}$ of a Hilbert space $\Hilbert$, if $\gamma_k\dom{\Dirac_j} \subseteq \dom{\Dirac_j}$ and $[\Dirac_j,\gamma_k] = 0$ for all $k,j \in \{1,\ldots,d\}$, and if $a \in \A$ satisfies $a\dom{\Dirac}\subseteq\dom{\Dirac}$ and $[\Dirac,a]$  bounded, then, for any nonempty finite subset $F \subseteq \{1,\ldots,d\}$, we conclude:
		\begin{equation*}
			\opnorm{[\sum_{j \in F} \Dirac_j \gamma_j,a]}{}{\Hilbert} \leq \sqrt{|F|}\ \opnorm{[\Dirac,a]}{}{\Hilbert} \text.
		\end{equation*}
	\end{proposition}
	
	\begin{proof}
		Let $F\subseteq \{1,\ldots,d\}$ be a nonempty finite set. We simply note that for all $l \in \{1,\ldots,d\}$:
		\begin{equation*}
			\sum_{j \in F} (\gamma_j \gamma_l + \gamma_l \gamma_j) =
			\begin{cases}
				0 \text{ if $l \notin F$,} \\
				2 \text{ if $l \in F$.}
			\end{cases} 
		\end{equation*}
		Therefore, since $\gamma_k$ commutes with $\Dirac_j$ and $a \in \A$ for all $j,k \in \{1,\ldots,d\}$:
		\begin{equation*}
			[\sum_{j\in F} \Dirac_j\gamma_j,a] = \frac{1}{2}\left(\sum_{j\in F} \gamma_j [\Dirac,a] + [\Dirac,a]\sum_{j\in F} \gamma_j  \right),
		\end{equation*}
		which implies
		\begin{align*}
			\opnorm{[\sum_{j \in F}\Dirac_j,a]}{}{\Hilbert} &\leq \frac{1}{2} \left(\opnorm{\sum_{j\in F}\gamma_j}{}{\Hilbert}\opnorm{[\Dirac,a]}{}{\Hilbert} + \opnorm{\sum_{j\in F}\gamma_j}{}{\Hilbert}\opnorm{[\Dirac,a]}{}{\Hilbert}\right)\\
			&\leq \opnorm{\sum_{j\in F}\gamma_j}{}{\Hilbert} \opnorm{[\Dirac,a]}{}{\Hilbert} \text.
		\end{align*}
		
		Now, 
		\begin{equation*}
			\left(\sum_{j \in F} \gamma_j\right)^\ast\left(\sum_{j \in F} \gamma_j\right) = \sum_{j < k \in F} (\gamma_j \gamma_k + \gamma_k \gamma_j) + \sum_{j\in F} \gamma_j^2 = |F|\text,
		\end{equation*} 
		so $\opnorm{\sum_{j \in F} \gamma_j}{}{\Hilbert} = \sqrt{F}$. We thus obtain our result.
	\end{proof}

	\section{Examples}
	\label{sec:easy-examples}
	
	\subsection{Collapsing to a point}
	
	Our first application of Theorem (\ref{main-thm}) is to the case when we collapse a metric spectral triple to a point. There are, in fact, infinitely many (metric) spectral triples over $\C$, of the form $(\C,\Hilbert,D)$, for any Hilbert space $\Hilbert$, and a self-adjoint operator $D$ with compact resolvent, where the action on $\C$ on $\Hilbert$ is  $z\in \C \mapsto z\cdot \mathrm{id}_{\Hilbert}$. They all give us, of course, the Lipschitz seminorm $0$, but they are obviously not unitarily equivalent to each other in general. Of central interest to us are the spectral triples $(\C,\C^n,0)$, for $n \in \N$, because they arise naturally  as limits in the spectral propinquity of arbitrary metric spectral triples collapsed to a point. Moreover, in that case, $n$ is the dimension of the space of harmonic spinors. Specifically we obtain, as a corollary to 
Theorem (\ref{main-thm}) the following result:	\begin{corollary}\label{cor:hrm-spin}
		If $(\A,\Hilbert,\Dirac)$ is a metric spectral triple with $0\in \spectrum{\Dirac}$, then:
		\begin{equation*}
		\lim_{\varepsilon\rightarrow 0} \spectralpropinquity{}((\A,\Hilbert,\frac{1}{\varepsilon}\Dirac),(\C,\ker\Dirac,0)) = 0 \text.
		\end{equation*}
	\end{corollary}
	
	\begin{proof}
		Let $\mu \in \StateSpace(\A)$. The map $a \in \A \mapsto \mu(a) \in \C$ is a conditional expectation. Moreover, it is immediate to verify that $[\Dirac,\mu(a)] = 0$. 
		
		Since $(\A,\Hilbert,\Dirac)$ is a metric spectral triple, we have by \cite{Rieffel98a}, for all $a\in\A_{\Dirac}$ (see Equation \eqref{AD-eq}):
		\begin{equation*}
		\norm{a-\mathds{E}(a)}{\A} \leq \qdiam{\A}{\Dirac} \opnorm{[\Dirac,a]}{}{\Hilbert} \text.
		\end{equation*}
		
		Now, $\Dirac = 0 + \Dirac$ (note: $0$ is the horizontal component here, and $\Dirac$ the vertical direction), and the other hypothesis of Theorem (\ref{main-thm}) are now trivially met. We obtain the desired conclusion.
	\end{proof}
	
	If $0\notin\spectrum{\Dirac}$, then there is no bounded sequence $(\lambda_n)_{n\in\N}$ with $\lambda_n \in n \spectrum{\Dirac} = \spectrum{\frac{1}{\frac{1}{n}}\Dirac}$ for all $n\in\N$; therefore  no convergent sequence $\lambda_n \in n \spectrum{\Dirac} = \spectrum{\frac{1}{\frac{1}{n}}\Dirac}$ can exist. So by Corollary (\ref{main-cor}), it is not possible for $(\A,\Hilbert,n\Dirac)_{n\in\N}$, and even less $(\A,\Hilbert,\frac{1}{\varepsilon}\Dirac)$, to converge to anything for the spectral propinquity.
	
	\medskip
	
	We now extend this first collapse result to products of metric spectral triples, with one of the spectral triples Abelian.

	\subsection{Collapsing $C(X,\A)$}
	
	Our next application of Theorem (\ref{main-thm}) is to study the collapse a product of two metric spectral triples, with one of the two constructed over an Abelian C*-algebra. There are several possible constructions of products of spectral triples, depending on whether they are even and/or odd. For our purpose, it turns out to be natural to begin with the case when we take a product between an even and an odd spectral triple.

	\begin{definition}
		Let $(\A,\Hilbert,\Dirac)$ and $(\B,\Hilbert[J],\Dirac[S])$ be two spectral triples, such that there exists a self-adjoint unitary $\gamma$ on $\Hilbert$ which commutes with $\A$ and anticommutes with $\Dirac$ (i.e. $\gamma$ is a $\faktor{\Z}{2}$ grading for the even spectral triple $(\A,\Hilbert,\Dirac)$). We set:
		\begin{equation*}
		\Dirac_x \coloneqq \Dirac\times_\gamma \Dirac[S] = \Dirac\otimes 1_{\Hilbert[J]} + \gamma\otimes\Dirac[S] \text,
		\end{equation*}
		defined on $\dom{\Dirac}\otimes\dom{\Dirac[S]}$ inside $\Hilbert\otimes\Hilbert[J]$.
	\end{definition}
	
	A simple computation shows that $(\A\otimes\B, \Hilbert\otimes\Hilbert[J],\Dirac_x)$ is a spectral triple \cite{DabDoss}. We also note that, like with Proposition (\ref{comparison-lemma}):
	\begin{lemma}
		Let $(\A,\Hilbert,\Dirac)$ and $(\B,\Hilbert[J],\Dirac[S])$ be two spectral triples, with $(\A,\Hilbert,\Dirac)$ even with grading $\gamma$. Let $\Dirac_\times \coloneqq \Dirac\otimes 1_{\Hilbert[J]} + \gamma\otimes\Dirac[S]$. For all $c \in \A\otimes\B$ such that $c\dom{\Dirac_\times}\subseteq\dom{\Dirac_\times}$, we have:
		\begin{equation*}
		\max\left\{ \opnorm{[\Dirac\otimes 1_{\Hilbert[J]},c]}{}{\Hilbert\otimes\Hilbert[J]}, \opnorm{[1_{\Hilbert}\otimes\Dirac[S],c]}{}{\Hilbert\otimes\Hilbert[J]} \right\} \leq \opnorm{[\Dirac_\times,c]}{}{\Hilbert\otimes\Hilbert[J]} \text.
		\end{equation*}
	\end{lemma}
	
	\begin{proof}
		We observe that, since $\gamma^2 = 1$ and $\gamma\Dirac=-\Dirac \gamma$,
		\begin{equation*}
		[\Dirac\otimes 1_{\Hilbert[J]},c] = \frac{1}{2}\left([\Dirac_\times,c](\gamma\otimes 1_{\Hilbert[J]}) - (\gamma\otimes 1_{\Hilbert[J]})[\Dirac_\times,c]\right)
		\end{equation*}
		and
		\begin{equation*}
		[1_{\Hilbert}\otimes \Dirac[S],c] = \frac{1}{2}\left([\Dirac_\times,c](\gamma\otimes 1_{\Hilbert[J]}) + (\gamma\otimes 1_{\Hilbert[J]})[\Dirac_\times,c]\right)\text.
		\end{equation*}
		Our result now follows immediately.
	\end{proof}  
	
	Another, related, comparison between vertical, horizontal, and global components can be gleaned from the above results, which will prove very helpful for our purpose.
	
	\begin{lemma}\label{integral-comp-lemma}
		Let $(\A,\Hilbert,\Dirac)$ be an even spectral triple with grading $\gamma$, and $(\B,\Hilbert[J],\Dirac[S])$ another spectral triple. Let $\Dirac_\times \coloneqq \Dirac\otimes 1_{\Hilbert[J]} + \gamma\otimes\Dirac[S]$. For all $c \in \A\otimes\B$ such that $c\dom{\Dirac_\times}\subseteq\dom{\Dirac_\times}$, and for all $\varphi \in \StateSpace(\A)$ (resp. $\psi \in \StateSpace(\B)$), we have:
		\begin{equation*}
		\opnorm{[\Dirac,(\id_\A\otimes\psi)(c)]}{}{\Hilbert} \leq \opnorm{[\Dirac_{\times},c]}{}{\Hilbert\otimes \Hilbert[J]} \text{(resp. } \opnorm{[\Dirac[S],(\varphi\otimes\id_\B)(c)]}{}{\Hilbert[J]} \leq \opnorm{[\Dirac_{\times},c]}{}{\Hilbert\otimes \Hilbert[J]} \text{ ).}
		\end{equation*}
		
	Moreover, for any $\varphi,\mu\in\StateSpace(\A)$, $\psi,\nu\in\StateSpace(\B)$, we have:
		\begin{equation}\label{product-dist-eq}
		\Kantorovich{\Dirac_{\times}}(\varphi\otimes\psi,\mu\otimes\nu) \leq \Kantorovich{\Dirac}(\varphi,\mu) + \Kantorovich{\Dirac[S]}(\psi,\nu) \text.
		\end{equation}
		
	\end{lemma}
	
	\begin{proof}
 Let $\psi \in \StateSpace(\B)$, which we extend by the Hahn-Banach theorem to a state of $\alg{L}(\Hilbert[J])$, still denoted by $\psi$. 
  Let $c = \sum_{j=1}^k a_j \otimes b_j$ with $a_1,\ldots,a_k \in \dom{\Lip_{\Dirac}}$ and $b_1,\ldots,b_k \in \dom{\Lip_{\Dirac[S]}}$. Writing $\Gamma\coloneqq \gamma\otimes \mathrm{id}_{\Hilbert{J}}$, which is a self-adjoint unitary which commutes with $\A$ and $\B$ but anticommutes with $\Dirac$, we compute:
	\begin{align*}
		[\Dirac,(\id_\A\otimes\psi)(c)]
		&=[\Dirac,\sum_{j=1}^k a_j\psi(b_j)] \\
		&=\sum_{j=1}^k [\Dirac,a_j]\psi(b_j) \\
		&= (\id_\A\otimes\psi)([\Dirac,\sum_{j=1}^k a_j ]\otimes b_j) \\
		&=(\id_\A\otimes\psi)([\Dirac,\sum_{j=1}^k a_j ]\otimes b_j) \\
		&=(\id_\A\otimes\psi)([\Dirac\otimes 1_{\Hilbert[J]},\sum_{j=1}^k a_j \otimes b_j]) \\
		&=(\id_\A\otimes\psi)([\Dirac\otimes 1_{\Hilbert[J]},c]) \\
		&=(\id_\A\otimes\psi)(\frac{\Gamma}{2}\left(\Gamma[\Dirac_\times),c] - [\Dirac_\times,c]\Gamma\right) \text.
	\end{align*}
Therefore, 
\begin{align*}
	\opnorm{[\Dirac,(\id_\A\otimes\psi)(c)]}{}{\Hilbert}
	&=\opnorm{(\id_\A\otimes\psi)(\frac{\Gamma}{2}\left(\Gamma[\Dirac_\times),c] - [\Dirac_\times,c]\Gamma\right)}{}{\Hilbert} \\
	&\leq \opnorm{\frac{\Gamma}{2}\left(\Gamma[\Dirac_\times),c] - [\Dirac_\times,c]\Gamma\right)}{}{\Hilbert\otimes\Hilbert[J]} \\
	&\leq \frac{1}{2}\opnorm{\Gamma[\Dirac_\times),c] - [\Dirac_\times,c]\Gamma}{}{\Hilbert\otimes\Hilbert[J]} \\
	&\leq \opnorm{[\Dirac_\times,c]}{}{\Hilbert\otimes\Hilbert[J]} \text,
\end{align*}	
and this extends to  for all $c \in \A\otimes\B$ such that $c\dom{\Dirac_\times}\subseteq\dom{\Dirac_\times}$ since our derivations are closed.
		
		Now, let $c \in \dom{\Lip_{\times}}$ with $\Lip_{\times}(c) \leq 1$ and $\varphi, \mu \in \StateSpace(\A)$, $\psi, \nu \in \StateSpace(\B)$. Then
		\begin{align*}
		|\varphi\otimes\psi(c) - \mu\otimes\nu(c)|
		&\leq|\varphi\otimes\psi(c) - \mu\otimes\psi(c)| + |\mu\otimes\psi(c) - \mu\otimes\nu(c)| \\
		&\leq|\varphi( \underbracket[1pt]{(\id_\A\otimes\psi)(c)}_{\Lip_{\Dirac} \leq 1} ) - \mu( (\id_\A\otimes\psi)(c) )| + |\psi( \underbracket[1pt]{(\mu\otimes\id_\B)(c)}_{\Lip_{\Dirac[S]}\leq 1} ) - \nu( (\mu\otimes\id_\B)(c) )| \\
		&\leq\Kantorovich{\Dirac}(\varphi,\mu) + \Kantorovich{\Dirac[S]}(\psi,\nu) \text.
		\end{align*}
		Hence Equation (\ref{product-dist-eq}) holds as claimed.
	\end{proof}

	It is not generally known under what condition a product of metric spectral triples is itself metric, even though some progress has recently made in \cite{Kaad23}. However, such a product is always metric when one of the two spectral triples is built over an Abelian C*-algebra. Additionally,  we continue to assume that  the spectral triple over the Abelian C*-algebra is even.
	
	\begin{lemma}
		Let $X$ be a compact Hausdorff space. If $(C(X),\Hilbert,\Dirac)$ is an even metric spectral triple with grading $\gamma$, and $(\A,\Hilbert[J],\Dirac[S])$ is a metric spectral triple, then $(C(X,\A),\Hilbert\otimes\Hilbert[J],\Dirac_x)$ is also a metric spectral triple.
	\end{lemma}
	
	\begin{proof}
		We identify $C(X,\A)$ with $C(X)\otimes\A$ in the canonical way: $f\in C(X)$ acts as $f\otimes 1_{\Hilbert[J]}$ on $\Hilbert\otimes\Hilbert[J]$, and $a\in\A$ acts as $1_{\Hilbert}\otimes a$.
		
		Set 
		\begin{equation*}
		\dom{\Lip_\times}\coloneqq\left\{ f \in C(X,\sa{\A}) : f\dom{\Dirac_\times}\subseteq\Dirac_\times, [\Dirac_\times,f]\text{ bounded}\right\}\text,
		\end{equation*}
		and $\Lip_\times(f) = \opnorm{[\Dirac_\times,f]}{}{\Hilbert\otimes\Hilbert[J]}$.
		
		Suppose $\Lip_\times(f) \leq 1$ for some $f \in C(X,\A)$. Then by Proposition (\ref{comparison-lemma}), the following holds:
		\begin{equation*}
		\opnorm{[1_{\Hilbert}\otimes \Dirac[S],f]}{}{\Hilbert\otimes\Hilbert[J]} \leq \Lip_\times(f) \leq 1 \text.
		\end{equation*}
		Thus, for each $x\in X$, we have $\opnorm{[\Dirac[S],f(x)]}{}{\Hilbert[J]} \leq 1$. 
		
		By Lemma (\ref{integral-comp-lemma}), we then observe that for any state $\varphi \in \StateSpace(\A)$ of $\A$, we have $(1_\A\otimes \varphi)f \in C(X)$ satisfies $\opnorm{[\Dirac,(1_\A\otimes\varphi)f]}{}{\Hilbert\otimes\Hilbert[J]} \leq \opnorm{[\Dirac_\times,f]}{}{\Hilbert\otimes\Hilbert[J]}\leq 1$. 
		Since $(C(X),\Hilbert,\Dirac)$ is metric,  by \cite[Proposition 2.2]{Rieffel99},  the restriction of $\Kantorovich{\Dirac}$ to the space of characters of $C(X)$, which is identified with $X$, gives a metric for the weak* topology on $X$. 
		So fixed any $x,x_0 \in X$, and identifying $x,x_0$ with their associated characters given by evaluation maps, we obtain
		\begin{align*}
		|\varphi(f(x)) - \varphi(f(x_0))| 
		&\leq \Kantorovich{\Dirac_\times}(\varphi_x,\varphi_{x_0})\opnorm{[\Dirac_\times,f]}{}{\Hilbert\otimes\Hilbert[J]} \\
		&\leq \Kantorovich{\Dirac}(x,x_0)\qdiam{C(X)}{\Dirac} \text.
		\end{align*}
		Hence
		\begin{equation*}
		\norm{f(x) - f(x_0)}{\A} \leq \qdiam{C(X)}{\Dirac} \text.
		\end{equation*}
		Moreover the  Lipschitz seminorm for the  metric induced by $\Kantorovich{\Dirac}$ on $X$  is less or equal to $\opnorm{[\Dirac,\cdot]}{}{\Hilbert\otimes\Hilbert[J]}$. Thus $\{ f \in C(X,\A) : \Lip_\times(f) \leq 1\}$ is an equicontinuous family of functions in  $C(X)$.

		Consequently, fixing $x_0 \in X$, the set $\{ f(x) : f \in C(X,\A),f(x_0) = 0,\Lip_\times(f) \leq 1\}$ is a subset of the compact set $\left\{ a\in \A : \Lip_{\Dirac[S]}(a) \leq 1, \norm{a}{\A} \leq \qdiam{C(X)}{\Dirac} \right\}$. Therefore, $\{ f \in C(X,\A) : f(x_0) = 0, \Lip_\times(f) \leq 1\}$ is an equicontinuous family of functions over the compact space $X$, and all valued in the common compact $\{a\in\A:\Lip(a)\leq 1,\norm{a}{\A}\leq \diam{C(X)}{\Dirac} \}$: by Arzel\`a-Ascoli's theorem, we conclude that $\{ f \in C(X,\A) : f(x_0) = 0,\Lip_\times(f) \leq 1\}$ is totally bounded. By lower semicontinuity of $\Lip_\times$, this set is in fact compact. By \cite{Rieffel05}, our proof is now complete.
	\end{proof}
	
	\medskip
	
	We now apply Theorem (\ref{main-thm}) to our product case, with an Abelian factor.
	
	\begin{theorem}\label{product-case-thm}
		Let $(\A,\Hilbert,\Dirac)$ be an even metric spectral triple with grading $\gamma$, and let $(\B,\Hilbert[J],\Dirac[S])$ be a metric spectral triple, with $\A$ or $\B$ Abelian.  For each $\varepsilon > 0$, we define
		\begin{equation*}
		\Dirac_{\times,\varepsilon} \coloneqq \Dirac\otimes 1_{\Hilbert[J]} + \frac{1}{\varepsilon}(\gamma\times\Dirac[S]) 
		\end{equation*}
		on $\dom{\Dirac}\otimes\dom{\Dirac[S]}$ inside $\Hilbert\otimes\Hilbert[J]$.
		Then $(\A\otimes\B,\Hilbert\otimes\Hilbert[J],\Dirac_{\times,\varepsilon})$ is a metric spectral triple for all $\varepsilon > 0$, and: 
		\begin{equation*}
		\lim_{\varepsilon\rightarrow 0} \spectralpropinquity{}\left(  (\A\otimes\B, \Hilbert\otimes\Hilbert[J],\Dirac_\varepsilon), (\A,\Hilbert\otimes\ker \Dirac[S],  \Dirac\otimes 1_{\ker \Dirac[S]} \right) = 0 \text.
		\end{equation*}
	\end{theorem}
	
	\begin{proof}
		If we denote by  $p$ to be the orthogonal projection from $\Hilbert\otimes\Hilbert[J]$ onto $\ker(1_{\Hilbert}\otimes\Dirac[S]) = \Hilbert\otimes\ker\Dirac[S]$, we note that $p = 1_{\Hilbert}\otimes q$ where $q$ is the orthogonal projection from $\Hilbert[J]$ onto $\ker\Dirac[S]$. Therefore, $p$ commutes with $a\otimes 1_{\Hilbert[J]}$ for all $a\in \A$, and with $\Dirac\otimes 1_{\Hilbert[J]}$.

		Fix $\mu \in \StateSpace(\B)$. We define $\mathds{E} : c \in \A\otimes\B\mapsto (1\otimes\mu)(c)$. By construction, $\mathds{E}(a\otimes 1) = a\otimes 1$ for all $a\in\A$. Moreover, for all $a \in \A_{\Dirac}$, where $\A_{\Dirac}$ is the Lipschitz algebra of $\Dirac$, and $b \in \B$:
		\begin{equation*}
		[1_{\Hilbert}\otimes\Dirac[S],\mathds{E}(a\otimes b)]
		=a\otimes[\Dirac[S],\mu(b)] = 0 \text. 
		\end{equation*}
		By linearity and since $[1_{\Hilbert}\otimes\Dirac[S],\cdot]$ is a closed derivation, we conclude that $[1_{\Hilbert}\otimes\Dirac[S],\mathds{E}(\cdot)] = 0$, as required.
		
		Now,
		\begin{align*}
		[\Dirac\otimes 1_{\Hilbert[J]},\mathds{E}(a\otimes b)]
		= [\Dirac,a]\mu(b) 
		= 1_\A \otimes \left( \mu[\Dirac \otimes 1_{\Hilbert[J]},a\otimes b]\right)
		\end{align*}
		so, again by linearity and since our derivations are closed, 
		\begin{equation*}
		\opnorm{[\Dirac\otimes 1_{\Hilbert[J]},\mathds{E}(a \otimes b)]}{}{\Hilbert\otimes\Hilbert[J]} \leq \opnorm{[\Dirac \otimes 1_{\Hilbert[J]},a \otimes b]}{}{\Hilbert\otimes\Hilbert[J]}\text.
		\end{equation*}
		Moreover, by definition of the projections $p$ and $q$:
		\begin{align*}
		p[\Dirac\otimes 1_{\Hilbert[J]},a\otimes 1]p = (1_{\Hilbert}\otimes q)([\Dirac,a]\otimes 1_{\Hilbert[J]})(1_{\Hilbert}\otimes q) = [\Dirac,a]\otimes q
		\end{align*}
		and thus $\opnorm{p[\Dirac\otimes 1_{\Hilbert[J]},a\otimes 1_{\Hilbert[J]}]p}{}{\Hilbert\otimes\Hilbert[J]} = \opnorm{[\Dirac,a]}{}{\Hilbert}$.

		Last, let $c \in \A\otimes\B$. Let $\varphi \in \StateSpace(\A)$, and $\psi \in \StateSpace(\B)$. Then, by Lemma (\ref{integral-comp-lemma}):
		\begin{align*}
		|\varphi\otimes\psi(c-\mathds{E}(c))|
		&=|\varphi\otimes\psi(c-(1_\A\otimes\mu)(c))| \\
		&=|\varphi\otimes\psi(c) - \varphi\otimes\mu(c)| \\
		&\leq \Kantorovich{\Dirac}(\varphi,\varphi) + \varepsilon \Kantorovich{\Dirac[S]}(\psi,\mu) \leq \varepsilon \qdiam{\B}{\Dirac[S]} \text.
		\end{align*}
		
		We thus meet all the hypothesis of Theorem (\ref{main-thm}), and our conclusion follows.
	\end{proof}
	
	We now can extend our result to the case where  the non-collapsing spectral triple  is odd, rather than even. The idea is simply to  choose two anticommuting self-adjoint unitaries $\gamma_1$ and $\gamma_2$ on some finite dimensional space $E$, and apply Theorem (\ref{product-case-thm}) to the even spectral triple $(\A,\Hilbert\otimes E,\Dirac\otimes\gamma_1)$, with grading $1_{\Hilbert} \otimes \gamma_2$. We thus get (flipping the second and third tensor factor):
	
	\begin{corollary}
		Let $\gamma_1$ and $\gamma_2$ be two anticommuting self-adjoint unitaries on a finite dimensional vector space $E$. Let $(\A,\Hilbert,\Dirac)$ and $(\B,\Hilbert[J],\Dirac[S])$ be two metric spectral triples, one of which is Abelian. Define
		\begin{equation*}
		\Dirac_{\times,\varepsilon} \coloneqq \Dirac\otimes 1_{\Hilbert[J]} \otimes\gamma_1 + 1_{\Hilbert[S]}\otimes\frac{1}{\varepsilon}\Dirac[S]\otimes\gamma_2
		\end{equation*}
		on $\dom{\Dirac}\otimes\dom{\Dirac[S]}\otimes E \subseteq \Hilbert \otimes\Hilbert[S]\otimes E$.  Let  $\A\otimes\B$ act on $\Hilbert\otimes\Hilbert[J]\otimes E$ by extending the following action on elementary tensors: $(a\otimes b)(\xi\otimes\eta\otimes e) = a\xi\otimes b\eta\otimes e$. Then $(\A\otimes\B,\Hilbert\otimes\Hilbert[S]\otimes E,\Dirac_{\times,\varepsilon})$ is a metric spectral triple for all $\varepsilon>0$, and:
		\begin{equation*}
		\lim_{\varepsilon\rightarrow 0}\spectralpropinquity{}((\A\otimes\B,\Hilbert\otimes\Hilbert[J]\otimes E,\Dirac_{\times,\varepsilon}),(\A,\Hilbert\otimes\ker\Dirac_\B\otimes E,\Dirac\otimes 1_{\ker\Dirac[S]} \otimes \gamma_1)) = 0 \text.
		\end{equation*}	
	\end{corollary}
	
	\begin{proof}
		The triple $(\A,\Hilbert\otimes E,\Dirac\otimes \gamma_1)$, where $a \in \A$ acts as $a\otimes 1_E$, is an even spectral triple with grading $\gamma_2$. Our result follows by applying Theorem (\ref{product-case-thm}) (and flipping the factors $\Hilbert[J]$ and $E$).
	\end{proof}

	\section{Collapsing Noncommutative Principal $G$-Bundles}\label{sec:SW-appl}
	
	In this Section we apply Theorem \eqref{main-thm} to the spectral triples constructed in \cite{SchwiegerWagner22}.
	Let $\alpha$ be a strongly continuous and free action (the precise definition of free action  will be given later) of a compact Lie group $G$ on a unital C*-algebra $\A$. We denote by $\B \coloneqq \{ a \in \A : \forall g \in G \quad \alpha^g(a) = a \}$ the fixed point C*-subalgebra of $\alpha$. We assume that we are given a metric spectral triple $(\B,\Hilbert_\B,\Dirac_\B)$, where we identify $\B$ with its faithful image as an algebra of operators over $\Hilbert_\B$. We also endow $G$ with a left-invariant Riemannian metric, by choosing some inner product $\inner{\cdot}{\cdot}{G}$ on the Lie algebra $\mathfrak{g}$ of $G$. By the construction of Schwieger and Wagner in \cite{SchwiegerWagner22}, one obtain a metric spectral triple on $\A$ restricting on $\B$ to the spectral triple  $(\B,\Hilbert_\B,\Dirac_\B)$. We apply our Theorem \eqref{main-thm} to an $\epsilon$-perturbation of it, to show that the  perturbed spectral triple on $\A$ converges in the spectral propinquity to the  spectral triple on $\B$ if we shrink the metric on the group, which corresponds to making $\varepsilon \to 0$.
	
	We start by reviewing the decomposition of a C*-algebra $\A$ induced by the action of a compact Lie group $G$.
	
	\subsection{Decomposition of $\A$}

	We will first review how the  C*-algebra $\A$ can be decomposed into isotypic subspaces for the action $\alpha$. This decomposition is  parametrized by  the irreducible representations of $G$. We denote by $\widehat{G}$ the set of unitary-equivalence classes of irreducible representations of $G$, by abuse of notation, we identify $\sigma\in\widehat{G}$ with one of its representative, so $\sigma$ is seen, in practice, as a particular choice of an irreducible representation of $G$, in such a way that any irreducible representation of $G$ is unitary equivalent to one in $\widehat{G}$. Since $G$ is compact, all its irreducible representations are finite dimensional, and, by abuse of notation again, we write $\dim\sigma$ for the dimension of the space $V_\sigma$ on which $\sigma$ acts (which is obviously an invariant for the class of all representations unitary equivalent to $\sigma$).
	
	Now, if we denote by  $\lambda$  the unique Haar probability measure of $G$, then for any $f\in L^1(G,\lambda)$, we set:
	\begin{equation}\label{conv-eq}
	\alpha^f : a \in \A \longmapsto \int_G f(g) \alpha^g(a) \, d\lambda(g) \text;
	\end{equation}
	we note that $\alpha^f$ is a bounded linear operator over $\A$, with $\opnorm{\alpha^f}{}{\A} = \norm{f}{L^1(G,\lambda)}$. Of particular interest is the usual conditional expectation form $\A$ onto the fixed-point-subalgebra $\B$,
	\begin{equation}\label{expectation-eq}
	\mathds{E} = \alpha^1 : a \in \A \mapsto \int_G \alpha^g(a) \, d\lambda(g) \text.
	\end{equation}

	If $\sigma \in \widehat{G}$ is an irreducible representation, then the character of $\sigma$ is, by definition, the continuous function $\chi_\sigma : g \in G \mapsto \mathrm{tr}(\sigma(g))$, where $\mathrm{tr}$ is the normalized trace on the algebra of $\dim(\sigma)\times\dim(\sigma)$ matrices. The \emph{spectral subspace}, or \emph{isotypic subspace}, of $\alpha$ associated with $\sigma$ is then defined as:
	\begin{equation}\label{eq:eigenspaces}
	\A(\sigma) \coloneqq \left\{ a \in \A : a = \alpha^{\chi_\sigma}(a) \right\} \text.
	\end{equation}
	The space $\A(\sigma)$ is a Hilbert right $\B$-module, with $\B$-valued inner product \begin{equation*}
	\forall a,b\in\A(\sigma) \quad \inner{a}{b}{\B} \coloneqq \mathds{E}(a^\ast b)\text.
	\end{equation*}
	Moreover, $\A$ is the closure of the sum $\oplus_{\sigma \in \widehat{G}} \A(\sigma)$.
	
	A key observation of \cite{SchwiegerWagner22} is that, under the additional assumption that $\alpha$ is a  \emph{free action} (see below), the space $\A(\sigma)$ is actually isomorphic, as a $\B$-Hilbert module, to a finitely generated projective $\B$-module. In order to explain this, we introduce another version of spectral subspaces for $\alpha$, called \emph{multiplicities spaces}, which are defined as  fixed point spaces as follows in \cite{SchwiegerWagner22}:
	\begin{equation*}
	\Gamma_{\A}(\sigma) \coloneqq \left\{ x \in \A\otimes V_\sigma : \forall g \in G \quad \alpha^g\otimes \sigma^g(x) = x  \right\} \text.
	\end{equation*}
	
	The relationship between the isotypic space and the multiplicity space is given by the existence of a Hilbert right $\B$-module isomorphism
	in \cite[Equation 1]{SchwiegerWagner22}:
	\begin{equation}\label{def:Phi-sigma}
	\Phi_\sigma : \Gamma_\A(\sigma)\otimes \overline{V_\sigma} \longrightarrow \A(\overline{\sigma}) \text,
	\end{equation}
	which extends the map  defined for $a\in \A, v \in V_\sigma, w \in \overline{V_\sigma} $ by
	
	\begin{equation*}a\otimes v \otimes w \in \Gamma_\A(\sigma)\otimes \overline{V_\sigma} \mapsto \inner{v}{w}{} a,\end{equation*} 
	
	where $(\overline{\sigma},\overline{V_\sigma})$ is the conjugate representation of $\sigma$ (in particular, $V_{\overline{\sigma}} \coloneqq \overline{V_\sigma}$ is the conjugate vector space of $V_\sigma$).

	\subsection{Free Actions}
	We now review some properties of free actions, see \cite{Ellwood00,Phillips87,SchwiegerWagner17a,SchwiegerWagner17b, SchwiegerWagner17c, SchwiegerWagner22}.
	We henceforth assume that $\alpha$ is \emph{free}, which can be characterized in various manners; for our present purpose, it seems best to use \cite[Definition 3.1]{SchwiegerWagner17c}: we therefore assume that, for all $\sigma \in \widehat{G}$, we have $1_\B \in \inner{\Gamma_\A(\sigma)}{\Gamma_\A(\sigma)}{\B}$, where $1_\B$ is the unit of $\B$. As explained in \cite[Lemma 3.3]{SchwiegerWagner17c}, this implies in turn that there exists $s_1,\ldots,s_k \in \Gamma_\A(\sigma)$ such that $\sum_{j=1}^k \inner{s_j}{s_j}{\B} = 1_\B$. In \cite[Lemma 3.3]{SchwiegerWagner17c}, a coisometry (which they call $s$ in that paper) from $\Hilbert_\sigma\coloneqq\C^k$ onto $\A\otimes V_{\sigma}$ was defined by sending $(z_1,\ldots,z_k)$ to $\sum_{j=1}^k z_j s_j$; we denote the adjoint  of this coisometry by $s(\sigma)$.  
	To ease notation and construction ever so slightly, we also define $S(\sigma)$ as the adjoint of the coisometry 
	\begin{equation*}
	(b_1,\ldots,b_k) \in \B\otimes\Hilbert_{\sigma} \longmapsto \sum_{j=1}^k s_j b_j \in \A\otimes V_{\sigma} \text.
	\end{equation*}
	Note that $S(\sigma)$ is a $\B$-linear map, and that it is in fact, valued in the multiplicity space $\Gamma_\A(\sigma)$, since the latter is a $\B$-module. Also, $S(\sigma)$ restricted to $\C 1_\B \otimes \Hilbert_{\sigma} = \C^k$ is $s(\sigma)$.
	
	The key point here is that $\Gamma_\A({\sigma})$ is therefore a finitely generated projective module over $\B$, i.e. it is isomorphic to $P(\sigma)(\B\otimes\Hilbert_{{\sigma}})$, where the projection $P(\sigma)$ is defined by $P(\sigma) \coloneqq S(\sigma)S(\sigma)^\ast$. Using the isomorphism $\Phi_\sigma$, we then get that Equation \eqref{def:Phi-sigma} implies
	\begin{equation}\label{eq:split}
	\A(\overline{\sigma}) = \Gamma_\A(\sigma)\otimes \overline{V_\sigma}= \Phi_\sigma(P(\sigma)(\B\otimes\Hilbert_{\sigma})\otimes V_{\overline{\sigma}})\text.
	\end{equation}
	In other words, for all $\sigma \in  \widehat{G}$, the space $\A(\overline{\sigma})$ is also isomorphic, as a $\B$-module, to a finitely generated projective $\B$-module. Moreover, we can give a useful description of $\A(\overline{\sigma})$, as the closure in $\A$ of the linear span of elements 
	\begin{equation}\label{def:asigma}
	a_\sigma(b\otimes v\otimes w) \coloneqq \Phi_\sigma(P(\sigma)(b\otimes v)\otimes w) \hbox{ for all }b \in \B, v\in \Hilbert_{\sigma},  w \in V_{\overline{\sigma}}.
	\end{equation}

	Furthermore, one can  prove, with careful investigation of the above constructions, that we have, for all  $g \in G, b \in \B, v\in \Hilbert_\sigma$, and $ w \in V_{\overline{\sigma}}$ 
	\begin{equation}\label{eq:SW-typo}
	\alpha^g(a_\sigma(b\otimes v\otimes w)) = a_\sigma(b\otimes v\otimes \sigma^g w) . \footnote{There is a small typo in \cite[Equation (16)]{SchwiegerWagner22}: the action term $\sigma^g$ needs to be replaced by the conjugate representation  in that equation. However 
	\cite[Equation (21)]{SchwiegerWagner22} is correct.}
	\end{equation}
	
	Since $\A$ is the closure of $\oplus_{\sigma\in\widehat{G}}\A(\sigma)$, we thus obtain a description of $\A$ entirely in terms of $\B$ and various finite dimensional Hilbert spaces. This, in turns, enables the induction of a spectral triple on $\B$ to a spectral triple of $\A$.
	
	\medskip
	
	Now, $(1_\B \otimes s(\sigma)^\ast):\Hilbert_\B\otimes \Hilbert_\sigma \rightarrow \Hilbert_\B\otimes \A\otimes V_\sigma$, and thus $p(\sigma) \coloneqq (1_\B \otimes s(\sigma))(1_\B\otimes s(\sigma))^\ast$ is a projection of $\Hilbert_\B\otimes\Hilbert_\sigma$, since $1_\B\otimes s(\sigma)$ is an isometry.
	With this in mind, we define $p \coloneqq \oplus_{\sigma\in\widehat{G}} p(\sigma)\otimes \mathrm{id}_{V_{\overline{\sigma}}}$ as acting on the Hilbert sum $\oplus_{\sigma\in\widehat{G}} \Hilbert_\B\otimes\Hilbert_\sigma\otimes V_{\overline{\sigma}}$. We then define the following Hilbert space, which up to tensoring with a Hermitian space carrying a representation of spinors, will be part of our spectral triple:
	\begin{equation}\label{def:Hilbert spaces}
	\Hilbert_G \coloneqq \left(\overline{\oplus}_{\sigma\in\widehat{G}} \Hilbert_\sigma\otimes V_{\overline{\sigma}}\right)  \text{ and }  \Hilbert_p\coloneqq  p(\Hilbert_\B\otimes\Hilbert_G) \text.
	\end{equation}
	
	\subsection{The Representation of $\A$ and $G$ on $\Hilbert_p$}
	
	We will now describe Schwieger  and Wagner's  covariant representation of $(\A,G,\alpha)$ on $\Hilbert_p$  \cite{SchwiegerWagner22}. 
	
	First, as explained in \cite{SchwiegerWagner22}, we can fix that $V_1 = \C = \Hilbert_1$ and $p(1)$ is the identity, where $1\in\widehat{G}$ is the trivial representation (note that $\A(1) = \B$). Then, we can extend the map $\sigma \in \widehat{G} \rightarrow S(\sigma)$ to a map from the class of all (unitary classes of) representations of $G$ in a functorial way, by setting, for any unitary representation $\sigma$ of $G$, with decomposition $\sigma = \oplus_{j=1}^d \sigma_j$ in irreducible representations $\sigma_1,\ldots,\sigma_d \in \widehat{G}$:
	\begin{equation*}
	S(\sigma) \coloneqq \oplus_{j=1}^d S(\sigma_j) \text.
	\end{equation*}
	With this in mind, we introduce, for all $\sigma\in\widehat{G}$ (a word of caution about notation: $	\delta_\sigma$ is called $	\gamma_\sigma$ in \cite{SchwiegerWagner22} ): 
	\begin{equation*}
	\delta_\sigma : b \in \B \mapsto S(\sigma)(b\otimes 1_{V_\sigma})S(\sigma)^\ast \in \B\otimes\alg{L}(\Hilbert_\sigma)\text,
	\end{equation*}
	and, for all $\sigma,\tau \in \widehat{G}$:
	\begin{equation*}
	\omega(\sigma,\tau) \coloneqq  S(\sigma\otimes\tau) S(\sigma)^\ast S(\tau)^\ast \in \B\otimes\alg{L}(\Hilbert_\sigma\otimes\Hilbert_\tau,\Hilbert_{\sigma\otimes\tau}) \text.
	\end{equation*}
	
	{\color{teal} \eqref{def:asigma}}
	
	With the above notation, we will actually build a *-representation $\pi_p$ of $\A$ on $\Hilbert_p$ as below. 
	
	Firstly, recall   that the linear span of the elements  
	\begin{equation*}
	\{ \psi_{\sigma}(\xi \otimes v \otimes w) \coloneqq s(\sigma)s(\sigma)^* (\xi\otimes v) \otimes w :  \sigma\in\widehat{G}, \xi \in \Hilbert_\B,v \in\Hilbert_\sigma,  w\in V_{\overline{\sigma}}   \}
	\end{equation*}
	is  dense in $\Hilbert_p= p\left(\Hilbert_\B \otimes \Hilbert_G \right)\text.$
	Moreover, if we define  $\Phi_\sigma$ as in  Equation \eqref{def:Phi-sigma} and \cite[Equation (1)]{SchwiegerWagner22}, the  linear span of the elements
	\begin{equation}\label{eq:element}
	\{a_\sigma(b \otimes \eta \otimes  v) \coloneqq \Phi_\sigma \left( s(\sigma)^*(b \otimes \eta)  \otimes  v \right)\ : \sigma\in\widehat{G},b\in\B,v\in\Hilbert_\sigma,\eta\in \overline{V_\sigma}\}
	\end{equation}
	is dense in $\A$.
	Now, chosen  an element $a_\sigma(b\otimes v \otimes w)$,  and an element 
	$\psi_\tau(\xi \otimes\omega \otimes\eta)$, 
	with $ w \in \overline{V_\sigma}$,
	
	then the representation $\pi_p$ of $\A$ on $\Hilbert_p =p\left(\Hilbert_\B \otimes \Hilbert_G \right)$ is defined  in \cite[Equation (21)]{SchwiegerWagner22} by:
	\begin{equation} \label{def: A actn on Hp}
	\pi_p(a_\sigma(b\otimes v\otimes w)) \psi_\tau(\xi\otimes \omega\otimes \eta) \coloneqq \psi_{\sigma\otimes\tau}(\omega(\sigma,\tau)\delta_\tau(b)_{13}(\xi\otimes v \otimes \omega \otimes w \otimes \eta)) \text,
	\end{equation} 
	where $x\in A\otimes C \mapsto x_{13} \in A\otimes B\otimes C$ is the linear extension of the map $a\otimes c\in A\otimes C \mapsto a \otimes 1 \otimes c$.

	Indeed, since 
	$\{a_\sigma(b \otimes v\otimes w):\sigma\in\widehat{G},b\in\B,v\in\Hilbert_\sigma,w\in V_\sigma\}$ is    dense in  $\A$ \cite[Section 5.1]{SchwiegerWagner22},  and 
	the  linear span of $\{ \psi_{\sigma}(\xi \otimes v \otimes w) : \sigma\in\widehat{G}, \xi \in \Hilbert_\B,v \in\Hilbert_\sigma,  w\in V_{\overline{\sigma}}   \}$  is is dense in $\Hilbert_p$, it is a technical matter to check that these formulas indeed define a *-representation $\pi_p$ of $\A$ on $\Hilbert_p =p\left(\Hilbert_\B \otimes \Hilbert_G \right)$.
	
	We also define a representation $u$ of $G$ on $\Hilbert_p = p\left(\Hilbert_\B \otimes \Hilbert_G \right)$ by: for $g \in G$, , $\xi \in \Hilbert_\B,v \in\Hilbert_\sigma,  w\in V_{\overline{\sigma}}$, let:
	\begin{equation}\label{def: G actn on Hp}
	u^g \psi_\sigma(\xi\otimes v \otimes w) \coloneqq \psi_\sigma(\xi\otimes v \otimes {\overline{\sigma}}^g w)\text.
	\end{equation}
	Thus defined, $u$ extends to a unitary representation of $G$ on $\Hilbert_p$. Owing to properties of the isometries $S(\sigma)$ \cite[Lemma 3.3]{SchwiegerWagner17c}\cite[Lemma 3.1]{SchwiegerWagner22}, it is shown in \cite[Lemma 4.1 ]{SchwiegerWagner22} that $(\pi_p,u)$ is a indeed the sought-after covariant representation of $(\A,G,\alpha)$ on $\Hilbert_p$. Moreover, by construction, the fixed point subspace of $u$ is exactly $\Hilbert_\B\otimes\Hilbert_1\otimes V_1 = \Hilbert_\B$.

	In the rest of this section, we will identify $\A$ with $\pi_p(\A)$, writing $\Hilbert_p$ as a left $\A$-module and dropping the symbol $\pi_p$.
	
	We will also drop the subscript $p$ when it is clear from the context that we are considering $\pi_p $ or $u$.
	
	\medskip

	\subsection{The Hilbert Space and Spectral Triple Operators}
	
	As above, let $\alpha$ be a strongly continuous and free action of a compact Lie group $G$ on a unital C*-algebra $\A$ and let  $(\B,\Hilbert_\B,\Dirac_\B)$  be a metric spectral triple on the fixed-point algebra $\B$ of $\alpha$;   denote by $\B_0$ the  Lipschitz algebra of this spectral triple, i.e.
	\begin{equation*}
	\B_0 \coloneqq \left\{ b \in \B : b\dom{\Dirac_\B}\subseteq\dom{\Dirac_\B}\text{, }[\Dirac_\B,b] \text{ is bounded } \right\} \text.
	\end{equation*}
	Below we will review the construction of Schwieger and Wagner from  \cite{SchwiegerWagner22} of the spectral triple $(\A,\Hilbert_\A,\Dirac_\A)$ on $\A$
	that restricts to the fixed-point spectral triple $(\B,\Hilbert_\B,\Dirac_\B)$ on $\B$.

	For our construction, we fix a Hermitian space $\Hilbert_{\mathrm{spin}}$ and $(\dim G) + 1$  anticommuting self-adjoint unitaries $\gamma_0,\ldots,\gamma_{\dim G}$ acting on $\Hilbert_{\mathrm{spin}}$ --- i.e. we choose some finite dimensional representation of the Clifford algebra of $\C^{\dim G + 1}$. We then set, as the prospective Hilbert space for our spectral triple:
	\begin{equation*}
	\Hilbert_\A \coloneqq \Hilbert_p \otimes \Hilbert_{\mathrm{spin}} \text.
	\end{equation*}
	The  actions $\pi_p$ and  $u$ of $\A$ and $G$ on $\Hilbert_p$ we defined  in Equations \eqref{def: A actn on Hp} and \eqref{def: G actn on Hp} are extended to actions on $\Hilbert_\A $ in the following (trivial) way:

	\begin{equation}\label{def:repres Hilbert A}
	\pi_{\A}\coloneqq (\pi_{p}\otimes 1_{\Hilbert_{\mathrm{spin}}}),\quad u_{\A}\coloneqq (u\otimes 1_{\Hilbert_{\mathrm{spin}}})
	\end{equation}
	
	In the rest of this section, we will identify $\A$ with $\pi_\A(\A)$, writing $\Hilbert_\A$ as a left $\A$-module and dropping the symbol $\pi_\A$.
	

	We now define, on the subspace 
	\begin{equation*}
	p\left(\oplus_{\sigma\in\widehat{G}} \dom{\Dirac_\B}\otimes \Hilbert_\sigma\otimes V_{\overline{\sigma}}\right)\otimes\Hilbert_{\mathrm{spin}} \subseteq \Hilbert_\A,
	\end{equation*}	
	the operator:
	\begin{equation*}
	\Dirac_{h} \coloneqq \left(\oplus_{\sigma\in\widehat{G}}\left( p(\sigma)(\Dirac_\B\otimes \mathrm{id}_{\Hilbert_\sigma}) p(\sigma)\right)\otimes 1_{V_{\overline{\sigma}}}\right) \otimes \gamma_0 \text,
	\end{equation*}
	and without further mention, we also write $\Dirac_{h}$ for the closure of the above operator, which is indeed essentially self-adjoint. Moreover, when restricted to $\Hilbert_\B \otimes \C \otimes \C \otimes \Hilbert_{\mathrm{spin}}$, the operator $\Dirac_{h}$ equals $\Dirac_\B \otimes 1_{p\Hilbert_G}\otimes \gamma_0$.

	As seen naively from its definition, and established carefully in \cite{SchwiegerWagner22}, the operator $\Dirac_{h}$ commutes with the action $u$, namely for all $g \in G$, we have $u^g \dom{\Dirac_h}\subseteq\dom{\Dirac_h}$ and
	\begin{equation*}
	u^g \Dirac_h = \Dirac_h u^g \text.
	\end{equation*}

	So far we followed  a natural  pathway for extending the spectral triple over $\B$ to $\A$, but till here  our construction  has  no information on the ``vertical'' direction along the orbits of the action $\alpha$, and this presents itself, among other things, by the fact $\Dirac_h$ has no compact resolvent. We now address this matter by defining the vertical component of our prospective spectral triple over $\A$.
	
	To this end, we follow Rieffel's construction \cite{Rieffel98a}; see also \cite{Gabriel13, Gabriel16}. For all $\xi$ in the \emph{algebraic sum} $\oplus_{\sigma\in\widehat{G}}\Hilbert_\B\otimes\Hilbert_\sigma\otimes V_{\overline{\sigma}}$, the following limit is well-defined for any left invariant vector field $X \in \mathfrak{g}$:
	\begin{equation*}
	\partial_X \xi \coloneqq \lim_{t\rightarrow 0} \frac{1}{t}\left(u^{\exp(t X)}\xi-\xi\right)\text.
	\end{equation*}

	Fix an orthonormal basis $e_1,\ldots,e_d$, with $d \coloneqq \dim G$,  of the Lie algebra $\mathfrak{g}$ of $G$, for $\inner{\cdot}{\cdot}{G}$,  and  write $\partial_j \coloneqq \partial_{e_j}$ for each $j\in\{1,\ldots,d \}$.

	We then set $\Dirac_v$ to be the closure of the essentially self-adjoint operator
	\begin{equation*}
	\Dirac_v\coloneqq \ \sum_{j=1}^d \partial_j \otimes \gamma_j \quad \text{ on }\quad  \Hilbert_\A=\Hilbert_p\otimes\Hilbert_{\mathrm{spin}}  \text.
	\end{equation*}
	The kernel of $\Dirac_v$ is by construction $\Hilbert_\B\otimes \C \otimes \C \otimes\Hilbert_{\mathrm{spin}} \cong \Hilbert_\B \otimes\Hilbert_{\mathrm{spin}}$.

	\begin{remark} We also remark  that as $\Dirac_h$ commutes  by construction with the action $u$ of $G$ on $\Hilbert_\A$ which, in turn, is used to define $\Dirac_v$; so the operators $\Dirac_v$ and $\Dirac_h$ anti-commute.
	\end{remark}
	
	For the construction to move forward, we assume that: 
	\begin{equation*}
	\forall b \in \B_0: \quad \sup_{\tau \in \widehat{G}} \opnorm{[\Dirac_\B\otimes 1_{\Hilbert_\tau},\delta_\tau(b)]}{}{\Hilbert_\B\otimes\Hilbert_\tau} < \infty
	\end{equation*}
	and
	\begin{equation*}
	\forall \sigma \in \widehat{G}: \quad \sup_{\tau\in\widehat{G}} \opnorm{[\Dirac_\B\otimes 1,\omega(\sigma,\tau)]}{}{\Hilbert_\B\otimes\Hilbert_\tau} < +\infty \text.
	\end{equation*}

	Under these assumptions, the spectral triple constructed in \cite{SchwiegerWagner22} is then given by
	\begin{equation*}
	\left(\A, \Hilbert_\A, \Dirac_\A\right) \text{ where } \Dirac_\A \coloneqq \Dirac_h + \Dirac_v \text.
	\end{equation*}
	
	The fact that the above triple is indeed a spectral triple is seen by noting that the dense subspace 
	\begin{equation*}
	\A_0 \coloneqq \mathrm{Span} \left\{ a_\sigma(b\otimes v\otimes w) : \sigma \in \widehat{G}, b\in \B_0, v \in \Hilbert_\sigma, w\in V_{\overline{\sigma}} \right\}
	\end{equation*}
	has bounded commutator with $\Dirac_\A$ (see Theorem  \cite[Theorem 5.9]{SchwiegerWagner22}), and $\Dirac_\A$ thus defined has a compact resolvent, as needed.

	\subsection{The Noncommutative Principal $G$-Bundles Convergence Result}
	
	In this section, we apply Theorem  \eqref{main-thm} to those noncommutative principal $G$-bundles which are indeed equipped with a metric spectral triple.

	Let $\ell : G\rightarrow [0,\infty)$ be the distance from the unit $e$ of $G$, as computed using the Riemannian metric given by the translates of $\inner{\cdot}{\cdot}{G}$. We denote the diameter of $G$ for this metric by $\diam{G}{\ell}$, which is a finite number since $G$ is compact.
	Let $\lambda$ be the Haar measure on $G$.

	We begin with a useful lemma, due to Rieffel \cite[Proof of Theorem 3.1]{Rieffel98a}, which we include for convenience and to adapt it to our current notation.
	\begin{lemma}\label{MVT-lemma}
		For any $f \in L^1(G,\lambda)$, with $f\geq 0$, and for any $D \in \{\Dirac_h,\Dirac_v,\Dirac_\A\}$, we have for all $a\in \A_0$:
		\begin{equation*}
		\opnorm{[D,\alpha^f(a)]}{}{\Hilbert_\A}\leq \norm{f}{L^1(G)}\opnorm{[D,a]}{}{\Hilbert_\A}\text{,}
		\end{equation*}
		where $\alpha^f$ is defined in Equation \eqref{conv-eq}.
		
		 If, moreover, $\int_G f \, d\lambda = 1$, we also have:
		\begin{equation*}
		\norm{a-\alpha^f(a)}{\A} \leq \dim{G}\int_G f(g)\ell(g)\,d\lambda(g) \cdot \opnorm{[\Dirac_v,a]}{}{\Hilbert_\A}  \text. 
		\end{equation*}	
	\end{lemma}
	
	\begin{proof}
		Fix $a\in \A_0$ and let $D\in\{\Dirac_h,\Dirac_v,\Dirac_\A\}$. Since $D$ is self-adjoint, the seminorm $\opnorm{[D,\cdot]}{}{\Hilbert_\A}$ is lower semicontinuous, and therefore:
		\begin{align*}
		\opnorm{[D,\alpha^f(a)]}{}{\Hilbert_\A} 
		&\leq \int_G f(g)\opnorm{[D,\alpha^g(a)]}{}{\Hilbert_\A} \, d\lambda(g) \\
		&\leq \int_G f(g)\opnorm{u^g [D,a] u^{(g^{-1})}}{}{\Hilbert_\A} \, d\lambda(g) \\
		&= \int_G f(g) \opnorm{[D,a]}{}{\Hilbert_\A} \, d\lambda(g) \leq  \norm{f}{L^1(G)} \opnorm{[D,a]}{}{\Hilbert_\A} \text,
		\end{align*}
		as desired.

		We now assume that $\int_G f\,d\lambda = 1$. Define, for $a\in \A_0$, the map 
		\begin{equation*}
		da : X\in \mathfrak{g} \mapsto \partial_X a \text.
		\end{equation*}
		The map $da$ is linear, and thus bounded (since $\mathfrak{g}$ is finite dimensional). 
		
		Let $g \in G$. First note that if $c : [0,1]\rightarrow G$ is a smooth path from the unit $e$ of $G$ to $g$, then
		\begin{align*}
		\norm{a-\alpha^g(a)}{\A}
		&=\norm{\int_0^1 \frac{d}{dt}(\alpha^{c(t)}(a))\, dt }{\A}\\
		&\leq \int_G \norm{\alpha^{c(t)}(\partial_{c'(t)}a)}{\A} \, dt \\
		&\leq \opnorm{da}{\A}{\mathfrak{g}} \int_0^1 \norm{c'(t)}{\mathfrak{g}}\, dt \\
		&\leq \opnorm{da}{\A}{\mathfrak{g}} \ell(g) \text.
		\end{align*}
		
		Now, by the triangle inequality, since $(e_1,\ldots,e_d)$ is an orthonormal basis of $\mathfrak{g}$ for $\inner{\cdot}{\cdot}{G}$, we conclude that 
		\begin{equation*}
		\opnorm{da}{\mathfrak{g}}{\A} \leq \dim(G) \max_{j\in\{1,\ldots,\dim(G)\}} \norm{\partial_j(a)}{\A}\text.
		\end{equation*}
		Since
		\begin{equation*}
		\partial_j a \otimes 1_{\Hilbert_{\mathrm{spin}}} = \frac{1}{2}\left([\Dirac_v,a](1\otimes\gamma_j) + (1\otimes\gamma_j)[\Dirac_v,a]\right) \text,
		\end{equation*}
		we have that 
		\begin{equation*}
		\opnorm{da}{\mathfrak{g}}{\A} \leq \dim(G)\opnorm{[\Dirac_v,a]}{}{\Hilbert_\A}\text.
		\end{equation*}
		Therefore, for all $g \in G$,
		\begin{equation*}
		\norm{a-\alpha^g(a)}{\A} \leq \dim(G)\ell(g) \opnorm{[\Dirac_v,a]}{}{\Hilbert_\A} \text,
		\end{equation*}
		and so
		\begin{align*}
		\norm{a-\alpha^f(a)}{\A}
		&=\norm{\int_G f(g) a \, d\lambda(g) - \int_G f(g)\alpha^g(a) \, d\lambda(g)}{\A} \\ 
		&\leq \int_G f(g) \norm{a-\alpha^g(a)}{\A} \, d\lambda(g) \\
		&\leq \int_G \dim(G) f(g)\ell(g) \cdot \opnorm{[\Dirac_v,a]}{}{\Hilbert_\A} \, d\lambda(g) \\
		& = \dim(G) \int_G f(g)\ell(g)\,d\lambda(g) \cdot \opnorm{[\Dirac_v,a]}{}{\Hilbert_\A} \text,
		\end{align*}
		as claimed. 
	\end{proof}
	
	\medskip
	
	We now provide a sufficient condition to ensure that the spectral triple $(\A,\Hilbert_\A,\Dirac_\A)$ is metric.
	
	\begin{lemma}\label{pb-qcms-lemma}
		If, for all $\sigma\in\widehat{G}$, the set
		\begin{equation*}
		\left\{ a \in \A(\sigma) : \opnorm{[\Dirac_h,a]}{}{\Hilbert_\A} \leq 1, \norm{a}{\A} \leq 1 \right\}
		\end{equation*}
		is compact, then the spectral triple $(\A,\Hilbert_\A,\Dirac_\A)$ is metric.
	\end{lemma}
	
	\begin{proof} As usual, 
		to prove that a  spectral triple is metric, we apply \cite{Rieffel98a}; see also \cite[Theorem 2.1]{Rieffel05}.
		We will verify below that the conditions required in that theorem apply. To start, fix a state $\mu \in \StateSpace(\B)$ and define 
		$\varphi \coloneqq \mu\circ\mathds{E}$, where $\mathds{E}$ is the conditional expectation defined in Equation \eqref{expectation-eq}. 	
		 By construction $\varphi \in\StateSpace(\A)$. 
		Now let    $a \in \dom{\opnorm{[\Dirac_\A,\cdot]}{}{\Hilbert_\A}}$ with $\opnorm{[\Dirac_\A,a]}{}{{\Hilbert}_\A} \leq 1$ and $\varphi(a) = 0$, which implies   $\mu(\mathds{E}(a)) = \varphi(a) = 0$.
		By construction $\varphi(\A(\sigma)) = 0$ for all $\sigma\in\widehat{G}\setminus\{1\}$, too.  So
		\begin{align*}
		\opnorm{[\Dirac_\B,\mathds{E}(a)]}{}{\Hilbert_\B}
		&=\opnorm{[\Dirac_\B,\mathds{E}(a) ]\otimes 1_{\Hilbert{G}} \otimes \gamma_0 }{}{\Hilbert_\A} \\
		&\leq\opnorm{[\Dirac_\A,\mathds{E}(a)]}{}{\Hilbert_\A} \\
		&\leq 1.
		\end{align*}
	  By {\cite[Propostion 1.6]{Rieffel98a}} we conclude that $\norm{\mathds{E}(a)}{\A} \leq \qdiam{\B}{\Dirac_\B \otimes \gamma_0}$ . Now, $\norm{a-\mathds{E}(a)}{\A} \leq k \opnorm{[\Dirac_v,a]}{}{\Hilbert_\A} \leq k$ with $k\coloneqq \dim(G) \diam{G}{\ell}$ by Lemma (\ref{MVT-lemma}), so 
		\begin{align}\label{some-eq-1}
		\norm{a}{\A}
		&\leq \norm{a-\mathds{E}(a)}{\A} + \norm{\mathds{E}(a)}{\A} \\
		&\leq k + \qdiam{\B}{\Dirac_\B \otimes \gamma_0} \nonumber \text. \nonumber
		\end{align}
		
		Let $\varepsilon > 0$. By \cite{Latremoliere05}, there exists $f \in L^1(G,\lambda)$ with $f\geq 0$, $\int_G f(g) d\,\lambda(g) = 1$, $\int_G f(g)\ell(g)\,d\lambda(g) \leq \frac{\varepsilon}{2\dim{G}}$, and a finite subset $F\subseteq\widehat{G}$, such that $f = \sum_{\sigma\in F} x_\sigma \chi_\sigma$ is the linear combination of the characters of the representations in $F$, with coefficients $(x_\sigma)_{\sigma\in F}$. In particular, the range of $\alpha^f$ lies in $\oplus_{\sigma\in F} \A(\sigma)$. By Lemma (\ref{MVT-lemma}), we conclude that 
		\begin{equation*}
		\norm{a - \alpha^f(a)}{\A} \leq \frac{\varepsilon}{2} \text,
		\end{equation*}
		and, moreover, 
		\begin{equation*}
		\opnorm{[\Dirac_h,\alpha^f(a)]}{}{\Hilbert_\A} \leq \opnorm{[\Dirac_h,a]}{}{\Hilbert_\A} \leq 1 \text.
		\end{equation*}
		Now define  $c \coloneqq \alpha^f(a)$ and $K \coloneqq \max\{ |x_\sigma| : \sigma\in F\}$, and,  for each $\sigma \in F$, let $c_\sigma \coloneqq \alpha^{x_\sigma\chi_\sigma}(c) = x_\sigma \alpha^{\chi_\sigma}(x)$ (this latter notation is as in Equation \eqref{eq:eigenspaces}); of course, $c = \sum_{\sigma \in F} c_\sigma$.  We have, again by Equation \eqref{some-eq-1}, and the definition of $c_\sigma$:
		\begin{equation*}
		\norm{c_\sigma}{\A} \leq \norm{c}{\A} \leq k + \qdiam{\B}{\Dirac_\B \otimes \gamma_0}\quad  
		\end{equation*}
		and 
		\begin{align*}
		\opnorm{[\Dirac_h,c_\sigma]}{}{\Hilbert_\A} 
		&\leq \norm{x_\sigma\chi_\sigma}{L^1(G)} \opnorm{[\Dirac_h,c]}{}{\Hilbert_{\A}}  \text{ by Lemma (\ref{MVT-lemma}),}\\
		&\leq |x_\sigma| \opnorm{[\Dirac_h,c]}{}{\Hilbert_\A} \\
		&\leq |x_\sigma| \opnorm{[\Dirac,a]}{}{\Hilbert_\A} \text{ by Lemma (\ref{MVT-lemma}),}\\
		&\leq |x_\sigma| \leq K \text.
		\end{align*}
		In summary, we have shown that:
		\begin{multline*}
		\left\{ c \in \A : \norm{c}{\A} \leq k+\qdiam{\B}{{\Dirac_\A \otimes \gamma_0}},\ \opnorm{[\Dirac_h,c]}{}{\Hilbert_\A} \leq 1 \right\} \\ \subseteq
		\mathrm{sum}\left(\prod_{\sigma\in F} \left\{ c\in \A(\sigma) : \norm{c}{\A} \leq k + \qdiam{\B}{{\Dirac_\B \otimes \gamma_0}}, \opnorm{[\Dirac_h,c]}{}{\Hilbert_\A} \leq K \right\}\right) \text,
		\end{multline*}
		where $\mathrm{sum} : (c_\sigma)_{\sigma\in\widehat{G}} \mapsto \sum_{\sigma\in F} c_\sigma$. Of course, $\mathrm{sum}$ is continuous; therefore, by assumption, $\left\{ c \in \A : \norm{c}{\A} \leq k+\qdiam{\B}{\Dirac_\B \otimes \gamma_0},\opnorm{[\Dirac_h,c]}{}{\Hilbert_\A} \leq 1 \right\}$ is totally bounded, as the subset of the image of a compact set by a continuous map. So there exists a $\frac{\varepsilon}{2}$-dense subset $S$ of that set.
		
		Hence, there exists $d \in S$ such that $\norm{\alpha^f(a) - d}{\A} \leq \frac{\varepsilon}{2}$. Therefore, 
		\begin{equation*}
		\norm{a-d}{\A} \leq \norm{a-\alpha^f(a)}{\A} + \norm{\alpha^f(a)-d}{\A} < \frac{\varepsilon}{2} + \frac{\varepsilon}{2} = \varepsilon \text.
		\end{equation*}
		
		Summarizing the above result we now see that we have established that the set
		\begin{equation*}
		\left\{ a \in \A : \varphi(a) = 0, \opnorm{[\Dirac_\A,a]}{}{\Hilbert_\A} \leq 1 \right\}
		\end{equation*}
		is totally bounded. Since it is closed, and since $\A$ is complete, it is compact. Therefore, by \cite[Theorem 2.1]{Rieffel05}, our proof is complete.
	\end{proof}
	
	Recall that an action is \emph{cleft} when the isometry $s(\sigma)$ can be chosen to be a unitary for all $\sigma$, and thus $p(\sigma) = 1_{\B\otimes \Hilbert_\sigma}$ for each $\sigma\in\widehat{G}$. Cleft actions are always free \cite{SchwiegerWagner22}, and include many interesting examples of well-known actions. In fact, several of the examples we consider are cleft, which, for this paper, will already open up various interesting situations.
	
	\begin{corollary}\label{metric-sp-cor}
		If, for each $\sigma\in\widehat{G}$, there exists a linearly independent finite set $U(\sigma)$ of unitaries of $\A(\sigma)$ such that $\A(\sigma) \coloneqq \left\{ \sum_{v \in U(\sigma)} b_v v : \ b_v  \in \B \right\}$, and $[\Dirac_h,v]=0$ for each $v\in U(\sigma)$, then $\alpha$ is cleft, and $(\A,\Hilbert_\A,\Dirac_\A)$ is a metric spectral triple. 
	\end{corollary}
	
	\begin{proof} To prove our result we will use 
	Lemma (\ref{pb-qcms-lemma}), and so we will verify below that the conditions in that lemma are satisfied.
	We start with fixing $\sigma\in\widehat{G}$ and  $a \in \A(\sigma)$ with
		\begin{equation*}
		\norm{a}{\A} \leq 1 \text{ and }\opnorm{[\Dirac_h,a]}{}{\Hilbert_\A} \leq 1 \text.
		\end{equation*}
	Then there exists $b_1,\ldots,b_d \in \B$ such that $a = \sum_{j=1}^d b_j v_j$. By assumption, since $[\Dirac_h,v_j] = 0$, we have:
		\begin{equation*}
		[\Dirac_h, a] = \sum_{j=1}^d [\Dirac_h,b_j] v_j \text.
		\end{equation*}
		Also, if we define the conditional expectation  $\mathds{E}$ as in Equation \eqref{expectation-eq}, for each $j\in\{1,\ldots,d\}$,  we have:
		
		\begin{equation*}
		b_j = \mathds{E}( a v_j^\ast),\quad  \text{ which implies }\quad \norm{b_j}{\A} \leq \norm{a}{\A} \leq 1 \text.
		\end{equation*}
		
		Thus, we compute:
		\begin{align*}
		\opnorm{[\Dirac_\B,b_j]}{}{\Hilbert_\B} 
		&= \opnorm{[\Dirac_h,b_j]}{}{\Hilbert_\B \otimes \C\otimes  \Hilbert_{\mathrm{spin}}} \\
		&\leq \opnorm{[\Dirac_h,b_j]}{}{\Hilbert_\A } \\
		&= \opnorm{[\Dirac_h,\mathds{E}(a v^\ast_{j})]}{}{\Hilbert_\A} \\
		&\leq \int_G \opnorm{u^g [\Dirac_h, \sum_{k=1}^d b_k v_k v_j^\ast] u^{(g^{-1})}}{}{\Hilbert_\A} \,  d\lambda(g) \\
		&=\int_G \opnorm{[\Dirac_h, \sum_{k=1}^d b_k v_k v_j^\ast]}{}{\Hilbert_\A} \,  d\lambda(g) \\
		&= \opnorm{[\Dirac_h, \sum_{k=1}^d b_k v_k ] v_j^\ast}{}{\Hilbert_\A}  \\
		&=\opnorm{\sum_{k=1}^d [\Dirac_h,  b_k ]v_k v_j^\ast}{}{\Hilbert_\A}  \\
		&\leq\opnorm{\sum_{k=1}^d [\Dirac_h,  b_k ]v_k}{}{\Hilbert_\A} \\
		&=\opnorm{[\Dirac_h,a]}{}{\Hilbert_\A} \leq 1 \text.
		\end{align*}
		
		Therefore we have proven: 
		\begin{multline*}
		\left\{ a \in \A(\sigma) : \opnorm{[\Dirac_h,a]}{}{\Hilbert_\A}\leq 1, \norm{a}{\A} \leq 1 \right\} \subseteq \\
		\mathrm{sum}\left(\prod_{j=1}^d \left\{ b \in \B : \opnorm{[\Dirac_\B,b]}{}{\Hilbert_\B} \leq 1, \norm{b}{\A} \leq 1 \right\}\right) \text,
		\end{multline*}
		where $\mathrm{sum} : (b_j)_{j=1}^d \in \B^d \mapsto \sum_{j=1}^d b_j v_j $. Since $(\B,\Hilbert_\B,\Dirac_\B)$ is metric, the set $\{ b \in \B : \opnorm{[\Dirac_\B,b]}{}{\Hilbert_\B} \leq 1, \norm{b}{\B} \leq 1 \}$ is compact, and since $\mathrm{sum}$ is a continuous map, we conclude that the right hand side set is compact. So $\left\{ a \in \A(\sigma) : \opnorm{[\Dirac_h,a]}{}{\Hilbert_\A}\leq 1, \norm{a}{\A} \leq 1 \right\}$ is totally bounded. As it is a closed set, since $\Dirac_h$ is self-adjoint and thus $[\Dirac_h,\cdot]$ is a closed derivation, this set is compact. 
		Our result now follows from Lemma (\ref{pb-qcms-lemma}).
	\end{proof}

	We now note that  the construction of $\Dirac_v$, and hence of $\Dirac_\A$, depends on our choice of a metric $\inner{\cdot}{\cdot}{G}$  over $G$. Therefore, if  in our constructions we replace $\inner{\cdot}{\cdot}{G}$ by $\varepsilon\inner{\cdot}{\cdot}{G}$, where $\R\ni \varepsilon >0$, then we  can get a vertical operator $\Dirac_v^\varepsilon$ corresponding to this rescaling, as well as  a new Diract operator $\Dirac^\varepsilon \coloneqq\Dirac_h + \Dirac_v^\varepsilon$. A direct computation shows that $\Dirac_v^\varepsilon = \frac{1}{\varepsilon}\Dirac_v$. This rescaling 
	produces in turn a new spectral triple
	$(\A,\Hilbert_\A,\Dirac_\varepsilon)$. ``Collapsing'' the fibers then means taking the metric along the fiber to $0$, i.e. $\varepsilon$ to $0$. The effect of collapsing  on the  spectral triple $(\A,\Hilbert_\A,\Dirac_\varepsilon)$
	  is made precise in the theorem below, which is the main result of this section and  is a consequence of Theorem  \eqref{main-thm}.
	\begin{theorem}\label{thm: main conv result G-bundles}
		Under the assumption of this section, if we set $\Dirac_\varepsilon \coloneqq \Dirac_h + \Dirac_v^\varepsilon= 
		\Dirac_h + \frac{1}{\varepsilon}\Dirac_v$, then:
		\begin{equation*}
		\lim_{\varepsilon\rightarrow 0} \spectralpropinquity{}((\A,\Hilbert_\A,\Dirac_\varepsilon),(\B,\Hilbert_\B\otimes\Hilbert_{\mathrm{spin}},\Dirac_\B\otimes\gamma_0)) = 0 \text. 
		\end{equation*}
		
	\end{theorem}
	
	\begin{proof} We will verify that the hypotheses of Theorem \eqref{main-thm} are satisfied. First of all, Hypotheses (1) and (2) of Theorem \eqref{main-thm} are met, thanks to our choice of $\gamma_0,\ldots,\gamma_d$; in fact  we note that for all $a$ in the Lipschitz algebra of $(\A,\Hilbert_\A,\Dirac_\A)$, by Proposition \eqref{comparison-lemma} we have:
		\begin{equation*}
		\max\left\{ \opnorm{[\Dirac_h,a]}{}{\Hilbert_\A}, \frac{1}{\sqrt{\dim G}}\opnorm{[\Dirac_v,a]}{}{\Hilbert_\A} \right\} \leq \opnorm{[\Dirac_\A,a]}{}{\Hilbert_\A} \text.
		\end{equation*} 
		Next, by construction, the kernel of $\Dirac_v$ is $\Hilbert_\B\otimes (\C\otimes \C) \otimes \Hilbert_{\mathrm{spin}} \cong \Hilbert_\B\otimes\Hilbert_{\mathrm{spin}}$  (cf. Equations \eqref{def: A actn on Hp} and \eqref{def:repres Hilbert A} for the restriction of the action to $\B$). The projection $q : \Hilbert_\A \rightarrow \ker\Dirac_v$ is thus just the projection onto $\Hilbert_\B\otimes(\C\otimes \C) \otimes  \Hilbert_{\mathrm{spin}}$. By construction, $pq=qp=q$, and $q$ commutes with $\Dirac_h$; so Hypotheses (3) and (4) are satisfied. Moreover:
		\begin{equation*}
		q\Dirac_h q =  (\Dirac_\B\otimes 1_{p\Hilbert_G} \otimes \gamma_0) \text,
		\end{equation*}
		Now, since $qb=bq:$
		\begin{align}\label{eq:ineq restr}
		\opnorm{q[\Dirac_h,b]q}{}{\Hilbert_\A}&=\opnorm{[q\Dirac_hq,b]}{}{\Hilbert_\A} = \opnorm{[(\Dirac_\B\otimes 1_{p\Hilbert_G} \otimes \gamma_0),b]}{}{\Hilbert_\A}\\ &= \opnorm{[(\Dirac_\B \otimes \gamma_0),b]}{}{\Hilbert_\B \otimes  \Hilbert_{\mathrm{spin}}} 
		\text.\nonumber
		\end{align}
		
		Hypothesis (5) is satisfied since $(\B,\Hilbert_\B\otimes\Hilbert_{\mathrm{spin}},\Dirac_\B\otimes\gamma_0)$ is metric. To see this, for all $b \in \sa{\B}$ which boundedly commute with $\Dirac_\B$, note that
		\begin{equation*}
		\opnorm{[\Dirac_\B\otimes\gamma,b]}{}{\Hilbert_\B\otimes\Hilbert_{\mathrm{spin}}} = \opnorm{[\Dirac_\B,b]}{}{\Hilbert_\B} \text.
		\end{equation*}
		Moreover, $(\Dirac_\B\pm i)^{-1}\otimes\gamma_0 = (\Dirac_\B\otimes\gamma_0 \pm i)^{-1}$, and since $\gamma_0$ acts on the finite dimensional space $\Hilbert_{\mathrm{spin}}$, we conclude that $(\B,\Hilbert\otimes\Hilbert_{\mathrm{spin}},\Dirac_\B\otimes\gamma_0)$ is indeed a metric spectral triple. 
		
		To check Hypothesis (6), let $\lambda$ be the Haar probability measure over $G$. We now use again the conditional expectation $\mathds{E}: \A \to \B$ as in Equation \eqref{expectation-eq}, defined by, for all $a\in\A$:
		\begin{equation*}
		\mathds{E}(a) \coloneqq \int_G \alpha^g(a) \, d\lambda(g) \text.	
		\end{equation*}
	 Moreover, since $\Dirac_h$ commutes with $u$, and since the derivation $[\Dirac_h,\cdot]$ is closed, we conclude:
		\begin{align*}
		\opnorm{[\Dirac_h,\mathds{E}(a)]}{}{\Hilbert_\A}
		&=\opnorm{\int_G [\Dirac_h,u^g a u^{(g^{-1})}] \, d\lambda(g)}{}{\Hilbert_\A} \\
		&\leq \int_G \opnorm{u^g [\Dirac_h,a] u^{(g^{-1})}}{}{\Hilbert_\A} \, d\lambda(g) \\
		&= \opnorm{[\Dirac_h,a]}{}{\Hilbert_\A} \text.
		\end{align*}
		Together with Equation \eqref{eq:ineq restr} this proves the hypotheses concerning the horizontal operator  in Hypotheses (6).
		We now turn to the vertical component. First, $[\Dirac_v,b] = 0$ by construction for all $b \in \B$. Moreover, by \cite[proof of Theorem 3.1]{Rieffel98a}, see Lemma \eqref{MVT-lemma} for details, we also note that there exists $k > 0$ such that:
		\begin{equation*}
		\norm{a-\mathds{E}(a)}{\A} \leq k \opnorm{[\Dirac_v,a]}{}{\Hilbert_\A} \text.
		\end{equation*}

		We therefore have all the needed assumptions to apply Theorem (\ref{main-thm}), and get our conclusion.
	\end{proof}
	
	\subsection{Examples}
	
	We now provide a few examples of applications of  Theorem (\ref{thm: main conv result G-bundles}). 
	
	We begin with the case of equicontinuous actions of $\Z^d$ \cite{Bellissard10, Hawkins13, Paterson14,Iochum16, Klisse23, Austad24} for actions and tensor products.
	Let $\Z^d$ act via $\alpha $ on the unital C*-algebra $\B$
	and form the crossed product C*-algebra $\A \coloneqq \B\rtimes_\alpha\Z^d$. Let $\widehat{\alpha}$ be the dual action of $\T^d$ on $\A$. Of course $\widehat{\T^d} = \Z^d$. The C*-algebra $\A$ contains  canonical unitaries $v_1$,\ldots,$v_d$ generating  the canonical copy of $C^\ast  (\Z^d) = C(\T^d)$ in $\A$.  Moreover,  the fixed point C*-subalgebra $\A(0)$ of $\widehat{\alpha}$ is $\B$, and more generally, for each  $(k_1,\ldots,k_d) \in \Z^d$, the isotypic component $\A(k)$ is $\B v_1^{k_1} \cdots v_d^{k_d}$. 
	
	Let now $(\B,\Hilbert_\B,\Dirac_\B)$ be a metric spectral triple such that:
	\begin{equation*}
	\B_0 \coloneqq \left\{ b \in \B : \sup_{k \in \Z^d} \opnorm{[\Dirac,\alpha^k(a)]}{}{\Hilbert_\B} < \infty \right\} 
	\end{equation*}
	is dense in $\B$ --- such an action $\alpha$ is called equicontinuous .
	In this case, as seen in \cite[Section 6]{SchwiegerWagner22}, 
	the spectral triple on $\A$ constructed from $(\B,\Hilbert_\B,\Dirac_\B)$ can be described as follows \cite{Hawkins13}. Let $\gamma_0,\ldots,\gamma_d$ be a choice of $d+1$ anticommuting self-adjoint unitaries acting on $\C^{d+1}$. On its natural domain inside $\Hilbert_\B\otimes\ell^2(\Z^d)\otimes\C^{d+1}$, the Dirac operator $\Dirac_\A$ above become:
	\begin{equation*}
	\Dirac_\A \coloneqq \underbracket[1pt]{\Dirac_\B\otimes 1_{\ell^2(\Z^d)} \otimes\gamma_0}_{=\Dirac_h} + \underbracket[1pt]{\sum_{j=1}^d Z_j \otimes\gamma_j}_{=\Dirac_v} \text,
	\end{equation*}
	where $Z_j$ is the closure of the unique linear operator such that $Z_j(\xi\otimes\eta) : (z_1,\ldots,z_d)\in\Z^d \mapsto z_j\eta(z_1,\ldots,z_d)\xi$ for all $\xi\in\Hilbert_\B$, $\eta\in\ell^2(\Z^d)$, and we identify $\Hilbert_\B\otimes\ell^2(\Z^d)$ with $\ell^2(\Z^d,\Hilbert_\B)$. 
	
	By Corollary (\ref{metric-sp-cor}), we thus conclude that the spectral triple $(\A,\Hilbert_\A,\Dirac_\A)$ is metric if $(\B,\Hilbert_\B,\Dirac_\B)$ is, and moreover, by Theorem (\ref{thm: main conv result G-bundles}):
	\begin{corollary}
		Under the above assumptions,
		\begin{equation*}
		\lim_{\varepsilon\rightarrow 0} \spectralpropinquity{}((\B,\Hilbert\otimes\C^d,\Dirac_\B\otimes\gamma_0), (\A,\Hilbert_\A,\Dirac_\varepsilon)) = 0\text.
		\end{equation*}
		In particular,
		\begin{equation*}
		\spectrum{\Dirac_\B\otimes\gamma_0} = \left\{ \lim_{n\rightarrow\infty} \lambda_n : (\lambda_n)_{n\in\N} \text{ convergent sequence such that }\forall n \in \N \quad \lambda_n \in \spectrum{\Dirac_h + \frac{1}{\varepsilon_n} \Dirac_v} \right\} \text,
		\end{equation*}
		for any choice of sequence $(\varepsilon_n)_{n\in\N}$ in $(0,\infty)$ converging to $0$.
	\end{corollary}
	
	We now turn to the example of the quantum $4$-torus of \cite[Section 7]{SchwiegerWagner22}. We start with a quantum $4$-torus $\A^4_\theta$ generated by the four canonical unitaries $u_1,u_2,u_3,u_4$. We restrict the gauge action of $\T^4$ on $\A^4_\theta$ to the torus $\T^2 \cong \{(1,1)\}\times\T^2 \subseteq\T^4$, so if $(z_1,z_2) \in \T^2$, then
	\begin{equation*}
	\alpha^{(z_1,z_2)} (u_1) = u_1, \alpha^{(z_1,z_2)} u_2 = u_2, \alpha^{(z_1,z_2)} u _3 = z_1 u_3 \text{ and }\alpha^{(z_1,z_2)}u_4 = z_2 u_4 \text.
	\end{equation*}
	The fixed point C*-algebra of $\alpha$ is of course the quantum $2$-torus $\A^2_{\theta'}$ generated by $u_1$ and $u_2$ --- the matrix $\theta'$ is well-defined by this simple description modulo an integer-valued matrix. Moreover, the isotypic subspaces for $\alpha$ are classified by pairs of integers, and for all $k,l \in \Z$, we have $\A^4_{\theta}(k,l) = (\A^2_{\theta'}) u_3^k u_4^l$. As above, we can follow Rieffel's construction to obtain a spectral triple $(\A^4_\theta,L^2(\A^4_\theta)\otimes\C^4,\Dirac)$ where $L^2(\A^4_\theta)$ is the GNS space for the canonical tracial state of $\A^4$ (i.e. the conditional expectation for the dual action of $\T^4$), and
	\begin{equation*}
	\Dirac \coloneqq \text{ the closure of }\sum_{j=1}^4 \partial_j \otimes\gamma_j
	\end{equation*}
	where $\partial_j$ is associated to  the $j$-th component $z_j$ of the action $\R\ni t\in\R\alpha^{z_j(t)}$:
	\begin{equation*}
z_j  : t\in\R\mapsto (1, \ldots, \underbracket[1pt]{\exp(2i\pi t)}_{\text{$j$-th position}}, \ldots, 1)\text,
	\end{equation*}
	and the matrices  $\gamma_1$,\ldots,$\gamma_4$ are again anticommuting self-adjoint unitaries on $\C^4$.
	
	For all $\varepsilon > 0$, we define $\Dirac_\varepsilon$ as the closure of 
	\begin{equation*}
	\partial_1\otimes\gamma_1+\partial_2\otimes\gamma_2 + \frac{1}{\varepsilon}\partial_3\otimes\gamma_3 + \frac{1}{\varepsilon}\partial_4\otimes\gamma_4\text.
	\end{equation*}
	
	Again, we have a canonical spectral $(\A^2_{\theta'},L^2(\A^2_{\theta'}),\Dirac')$ on the fixed-point subalgebra and and extension $(\A^4_\theta,L^2(\A^4_\theta),\Dirac_\varepsilon)$ to the whole algebra.
	
	Thus, we may again apply Corollary (\ref{metric-sp-cor}), and then, Theorem (\ref{thm: main conv result G-bundles}), to obtain the following limit result.
	\begin{corollary}
		With the assumption as above,
		\begin{equation*}
		\lim_{\varepsilon\rightarrow 0^+} \spectralpropinquity{}((\A^4_\theta,L^2(\A^4_\theta),\Dirac_\varepsilon),(\A^2_{\theta'},L^2(\A^2_{\theta'}),\Dirac')) = 0 \text.
		\end{equation*}	
	\end{corollary}
	
	%
	%
	%
	%
	%
	%


	\section{Collapsing Commutative Spin Principal $U(1)$-Bundles: the Smooth Projectable Case \cite{Ammann99, Ammann-Bar-98}}
	\label{sec:smooth projectable case}
	
	In this section we present an example of an application of our Theorem \eqref{main-thm} to classical spaces, that is  the case of  smooth projectable principal  Riemannian closed spin manifold $U(1)$-bundles with smooth quotient space, see \cite{Ammann99, Ammann-Bar-98,  Roos2018ManMath,  Roos2018Thesis}. We will write $U(1)$ for the circle group $\T$ here, to keep the notations in our references. Our hypotheses here are as in the work of Ammann and Ammann and B\"ar in \cite{Ammann99, Ammann-Bar-98}. (See below for the precise definitions.) More general set-ups   are considered in the literature, also in the context of Gromov-Hausdorff limits of manifolds for  closed manifolds with bounded curvature and volume, see e.g.  \cite{Fukaya87}, \cite{Lott2002Berk,Lott2002Duke,Lott2002Euro}; see also  in the case of orbifold quotients the papers \cite{Roos2018ManMath,  Roos2018Thesis}. Moreover some of the cited  results (in the form of factorization) have been extended to    suitable noncommutative settings for example in  \cite{ForsythRennie19, KaadvS18, KaadvS18a, Brain16, Dab-Sit, Dab-Sit-Zucca, Dab-Zucca, ZuccaThesis}.   To simplify matters,  throughout this section we will assume that the group $U(1)$ acts smoothly, freely and isometrically on the spin closed manifold $M$, so that  the associated principal $U(1)$-bundle is a Riemannian submersion which has a manifold quotient space $N$. 
	We assume the all of the manifolds we consider are spin and that (when relevant) their spin structures are projectable, see below for the precise definitions. This principal $U(1)$-bundle  framework includes in particular the classic example of Hitchin of the Hopf fibration \cite{Hitchin}, as well as actions of $U(1)$ on tori. We will refer to the case of principal $U(1)$-bundles of the above type, as the \emph{smooth projectable case}.
	
	Our goal is to  use the structure detailed in   \cite{Ammann99, Ammann-Bar-98} and Theorem \eqref{main-thm} to prove convergence, under rescaling of the metric by $\varepsilon >0$ in the vertical direction, with respect to the spectral propinquity of a bounded variation of the  Dirac spectral triple on $M$ to the  Dirac spectral triple on $N$. Even in  the classical cases of  \cite{Ammann99, Ammann-Bar-98} this gives, besides convergence of the eigenvalues, a stronger convergence of the  continuous functional calculus. 
	
	For simplicity's sake we will assume in the sequel that the dimension $n$ of $N$ is even. Mutatis mutandi,  our constructions will also  apply when the dimension $n$ of $N$ is odd.

	\subsection{Collapsing Commutative Smooth Projectable Spin  $U(1)-$Bundles \cite{Ammann99, Ammann-Bar-98}}
	\label{sec:generalities-S-1-bundles}

	We now recall the context of \cite{Ammann99, Ammann-Bar-98,  Roos2018ManMath,  Roos2018Thesis}.
	We refer to these references for more details.

	We suppose that $U(1)$ acts smoothly, freely and isometrically on the closed connected  Riemannian spin manifold 
	$(M,\tilde g)$ of dimension $(n+1)$. Assume  that $n$ even. The base space $N$ will carry the unique Riemannian metric $g$ such that the projection
	\[ \pi : (M,\tilde g) {\longrightarrow} (N,g)
	\]
	is a Riemannian submersion. In particular we view $M$ as the total space of an principal 
	$U(1)$-bundle over the base space $N \coloneqq M/U(1)$.

	The $U(1)$-action induces a Killing vector field $K$ on $M$. 
	To keep the discussion simple we will assume that the length $\ell\coloneqq \Vert K\Vert >0$ is
	constant on $M$, that is,  
	the fibers of $\pi$ are assumed to be totally
	geodesic of equal length $2\pi \ell$.
	We also note that one  
	can relax this equal length assumption, see \cite[Remark 4.2]{Ammann-Bar-98}. Here too,   the case of fibers of non-constant length could be handled by a slight generalized version of Theorem \eqref{main-thm}; we leave to the interested reader to formulate it. However, in  the interest of simplicity, we assume here that all of the fibers have equal length. 
	
	The metric $\tilde{g}$ can be characterized in the following way. Let $K / \ell $ denote the normalized  Killing vector field
	associated to the $U(1)$ action and let 
	\begin{equation}\label{def:frame} f_1=\partial_1, \ldots, f_n=\partial_n   
	\end{equation}
	
	be the canonical (local) orthonormal frame
	on $N$. Then 
	
	\begin{equation*} \mathfrak{F}\coloneqq \left\{   e_{0}=K/ \ell , e_1=\widetilde{\partial}_1, \ldots, e_n=\widetilde{\partial}_n  \right\}  
	\end{equation*}

	where  $\tilde{X}$
	denotes the horizontal lift of a vector field $X$
	with respect to the connection $\omega$, is a local orthonormal frame for $\tilde{g}$. (This convention of using $\tilde \ $ for lifts will also be used or spinor fields, Christoffel symbols, etc. throughout this section.)
	
	Denote by 
	\begin{equation*}
	\widehat{\mathfrak{F}}\coloneqq  \{ e^j \}_{j=0,\ldots,n} \end{equation*}
	the dual frame to 
	$\mathfrak{F}$ for 1-forms.
	
	This principal $U(1)$-bundle has a unique connection $1$-form
	$i \omega: TM \to i \R$
	such that $\ker \omega |_m$ is orthogonal to the fibers for all $m \in M$; here we choose  $\omega = e^0$.
	The connection has a curvature $2$-form $d\omega$.
	For example in the case of the Hopf fibration the curvature is $-2i e^1\wedge e^2$ \cite{Hitchin}, \cite{Orduz}. 
	
	As the metric $\tilde{g}$ on $M$ is completely characterized by the 
	connection 1-form $i\omega$, the fiber length  $2 \pi \ell $ 
	and the metric $g$ on $N$, we can express the Dirac operator $\Dirac_M$ on $M$
	in terms of $\omega$, $\ell$, and $g$.
	This allowed Ammann and B\"ar  \cite{Ammann99, Ammann-Bar-98} to analyze the behavior of the spectrum 
	for collapsing $U(1)$-fibers. 
	In the 
	projectable case there is convergence of eigenvalues, and we will use the structure detailed in \cite{Ammann99, Ammann-Bar-98} to also prove convergence of the associated perturbed Dirac spectral triple under rescaling of the metric in the vertical direction.

	The $U(1)$-action on $M$ induces a $U(1)$-action on the $SO$-frame bundle ${P_{\mathrm{SO}}}(M)$.
	A spin structure $\tilde \varphi: P_\mathrm{Spin}(M) \to {P_{\mathrm{SO}}}(M)$ will
	be called {\it projectable} 
	if this $U(1)$-action on ${P_{\mathrm{SO}}}(M)$ lifts to $P_\mathrm{Spin}(M)$. 
	Otherwise it will be called \emph{nonprojectable}.
	
	Any projectable spin structure  on $M$ induces a spin structure on $N$. On the other hand, any spin structure on $N$ canonically induces a 
	projectable spin structure on $M$ via pull-back.
	
	\begin{equation*}\tilde\varphi:=\pi^*\varphi\times_{\Theta_n}\Theta_{n+1}:
	\pi^*\mathrm{Spin}(N)\times_{\mathrm{Spin}(n)}\mathrm{Spin}(n+1)\to 
	{P_{\mathrm{SO}(n)}}(M)\times_{SO(n)} SO(n+1)
	\end{equation*}
	
	yields a spin structure on $M$.

	By rescaling the metric $\tilde g$ on $M$ by the factor $\varepsilon >0$ along the fibers while keeping it
	the same on $\ker \omega$ we obtain a 1-parameter family of metrics
	$\tilde g_\varepsilon$ on $M$ for which $\pi_\varepsilon : M \to N$ (where $\pi_\varepsilon$ is given pointwise  by the same formula as $\pi$) is still a Riemannian submersion, with fibers of length $2 \pi \ell_\varepsilon$, where $\ell_\varepsilon\coloneqq \ell\,\varepsilon$,  is the length of the Killing field.

	\medskip

	To main idea used in the proof of the main result of Ammann and Ammann-B\"ar (reported as Theorem \ref{collapstheo} below)   is to 
	decompose  the Dirac operator $\Dirac_M$
	on $M$  as a sum of a vertical Dirac operator, a horizontal 
	Dirac operator, and a zero order term, very much as we have seen in prior sections. This decomposition is respected when we shrink the metric on the fibers by $\varepsilon$.
	In order to define    the \emph{horizontal and vertical Dirac operators} we first need to introduce  some additional definitions.
	
	If we denote by $\Sigma_{n+1}$ (resp. $\Sigma_{n}$) a unitary representation of $\mathrm{Spin}(n+1)$ (resp. $\mathrm{Spin}(n)$) of dimension $2^{[\frac{(n+1)}{2}]}$ (resp. $2^{[\frac{n}{2}]}$), we define the  \emph{spinor bundle} of $M$ (resp. $N$) by $\Sigma M \coloneqq P_\mathrm{Spin}(M)\times_{\mathrm{ Spin}(n+1)}\Sigma_{n+1}$ (resp. $\Sigma N \coloneqq P_\mathrm{Spin}(N)\times_{\mathrm{ Spin}(n)}\Sigma_{n}$).
	The action of $U(1)$ on $P_\mathrm{Spin}(M)$ induces an  
	action of $U(1)$ on the spinor bundle $\Sigma M$  
	which we denote by $\kappa$. 
	A spinor with base point $m$ will be 
	mapped by $\kappa(e^{it})$ to a spinor with base point $m\cdot e^{it}$.
	We define the {\em Lie derivative} of a smooth spinor $\Psi$ 
	in the direction of the Killing field
	$K$ by
	\begin{equation}\label{eq:def spinor Lie derivative}
	\mathcal{L}_K(\Psi)(m)= \frac{d} {dt}|_{t=0}
	\kappa(e^{-it})(\Psi(m\cdot e^{it})).
	\end{equation}
	
	Since $\mathcal{L}_K$ is the differential of a representation of the Lie group $U(1)$ 
	on $L^2(\Sigma M)$, we get the decomposition
	\begin{equation}\label{eq:decomp}
	L^2(\Sigma M)=\bigoplus_{k \in \Z} V_k
	\end{equation}
	into the eigenspaces $V_k$ of $\mathcal{L}_K$ for the eigenvalue $ik$, $k\in \Z$.
	The $U(1)$-action commutes with the Dirac operator $\Dirac_M$ on $M$, 
	hence this decomposition is preserved by $\Dirac_M$.  
	
	We  will also use the convention that any $r$-form 
	$\alpha$ acts on a spinor $\Psi$ by
	\begin{equation*}\gamma(\alpha)\Psi:=\sum_{i_1<\dots<i_r}\alpha(e_{i_1},\dots,e_{i_r})\,
	\gamma(e_{i_1})\cdots \gamma(e_{i_r}) \Psi\end{equation*}
	where the $e_i$ form an orthonormal basis of the tangent space.  
	
	The spinor covariant derivative differs from the Lie derivative in the direction $e_0$ of the Killing field  by: 
	
	\begin{equation}\label{vertinabla}
	\nabla_{e_0} = \mathcal{L}_{e_0} +
	{\ell \over 4}\,  \gamma(d\omega)=
	\mathcal{L}_{e_0} +
	{\ell \over 4}\, \sum_{j<k; i,j=1,\ldots,n} \gamma(d\omega (e_j, e_k)).
	\end{equation} 
	
	In light of the above difference  between $	\nabla_{e_0}$ and $ \partial_{e_0} $, we define the {\it vertical Dirac operator}  by
	
	\begin{equation}\label{eq:defofvertDirac}
	\Dirac_v:=\gamma(K/\ell)\,\mathcal{L}_K.
	\end{equation}
	
	For later reference, we also define  the zero order terms \begin{equation*}\label{eq:zeroorderterm}Z:=-\frac{1}{4}\,\gamma(K/\ell )\,\gamma(d\omega), \quad V:=-\frac{1}{4}\,\ell\,\gamma(K/\ell )\,\gamma(d\omega).\end{equation*}

	Next we associate to the $U(1)$-bundle $M\to N$ the complex line bundle 
	$L:=M\times_{U(1)}\C$ with the natural connection given by $i \omega$.
	Recall that if  $L$ is a line bundle, then 
	by convention $ L^k \coloneqq L^{\otimes k}$
	and $ L^{-k} \coloneqq (L^*)^{\otimes k}$.

	In \cite{Ammann-Bar-98} it is shown that when $n$ is even there is a natural  homothety of Hilbert spaces (which is an  isometry since our  fibers have constant length) 
	
	\begin{equation*} Q_k:L^2(\Sigma N \otimes L^{-k}) \to  V_k, 
	\end{equation*}

	which commutes with Clifford multiplication and 
	such that the horizontal covariant derivative is given by (recall that tilde's are use to denote lifts.)
	
	\begin{equation*}
	\nabla_{\tilde X}{Q_k(\Psi)} =Q_k(\nabla_{X}\Psi) 
	+ {\ell \over 4}\, \gamma(K/\ell )\gamma(\tilde V_X){Q_k(\Psi)}
	\end{equation*}
	
	where $V_X$ is the vector field on $N$ satisfying
	$d\omega(\tilde X,\cdot)=\langle \tilde V_X, \cdot \rangle$.
	
	Then  the \emph{horizontal operator} $\Dirac_h:L^2(\Sigma M) \to L^2(\Sigma M)$ 
	is defined as the unique closed linear operator, such that on each $V_k$ 
	it is given by the formula below, where
	$\nabla^N$ denotes the covariant spinor derivative on $N$ associated to the Levi-Civita connection on $N$, and
	$ k\nabla^\omega$ is the covariant derivative on the bundle $L^{-k}$ associated to the connection $i\omega$. (Note that now  we switched order of the tensor product factors so that  the vertical direction corresponds to the $0$ value of the  index.)
	
	\begin{align}\label{def:k hor Dirac} \Dirac_h:L^2(\Sigma M) \to L^2(\Sigma M):\quad \Dirac_h|_{V_k}\coloneqq Q_k \circ D'_k \circ {Q_k}^{-1},\\
	\hbox{ where }\quad  
	D'_{h,k} \coloneqq \sum_{i=1}^n (1_{L^{-k}} \otimes \gamma _i)\, (1 \otimes  \nabla^N_{f_i}+ k\nabla^\omega_{f_i} \otimes 1).
	\end{align} 
	with $	D'_{h,k}$  the twisted (of charge $k$) Dirac operator on 
	$ L^{-k} \otimes L^2(\Sigma N)$. 
	Moreover, we have 
	$\gamma(K/\ell)Q_k(\Psi)=Q_k(c \gamma({\rm dvol}_n) \Psi)$ with $c \in \{1,i,-1,-i\}$ 
	depending on $n$ and the representation of the Clifford algebra $Cl_{n+1}$.
	
	By putting everything together, it follows that
	
	\begin{equation}\label{eq:sum of Diracs}
	\Dirac_M = \sum_{i=0}^{n}\nabla_{e_i} \gamma(e_i)=\Dirac_v +\Dirac_h +V, \hbox{ with }\quad V=-\frac{1}{4}\,\ell\,\gamma(K/\ell )\,\gamma(d\omega).
	\end{equation}
	
	We will now list below the commutation relations between  the operators in our construction:
	
	Since $\gamma({\rm dvol}_n)$ anticommutes with any twisted Dirac operator on
	$N$, we know that $\gamma(K/\ell)$ anticommutes with $D_h$ and hence  with 
	$\gamma(K/\ell )$ \cite[Page 241]{Ammann-Bar-98}; therefore   it  also anticommutes with the vertical operator
	$D_v=\gamma(K/\ell)\mathcal{L}_K$; therefore the  squares of the vertical and horizontal Dirac operators can be simultaneously diagonalized.
	
	We now rescale the metric in the vertical direction by $\varepsilon >0$. Everything can be defined very much as in the case $\varepsilon =1$ detailed above, with the exception of the symbols and formulas being decorated by $\varepsilon$ or $\frac{1}{\varepsilon}$ . More in detail, by \cite{Bo-Gau-Moroi},  after the rescaling, the classical Dirac operator $\Dirac_{M_\varepsilon}$ associated to $(M_\varepsilon, \tilde{g}_\varepsilon)$ and defined on $L^2(\Sigma M_\varepsilon)$ can now be reinterpreted as the operator 
	
	\begin{eqnarray}
	\Dirac_{ M_\varepsilon}=	{1\over \ell_\varepsilon  } \Dirac_v + \Dirac_h + V_\varepsilon \quad \hbox{ defined on }\quad  L^2(\Sigma M).
	\end{eqnarray}
	
	Indeed, the rescaling of the metric corresponds to the rescaling of the spinors in the vertical component by $\varepsilon$, while the Dirac operator does not change for this type of rescaling. The term $\frac{1}{\varepsilon}$ takes care of all of the changes, \cite{Bo-Gau-Moroi}.
	
	This alternative interpretation of the classic operator $\Dirac_{M_\varepsilon}$ on $(M_\varepsilon, \tilde{g}_\varepsilon )$ as a rescaled Dirac operator defined on $\Sigma M$ will be used in the rest of this section as needed.

	In cases such as the ones described in \cite{Roos2018ManMath}, in which the change of the metric is more general than just vertical rescaling by $\varepsilon$, 
	one has to take into account more explicitly  the  isomorphism between $\Sigma M_\varepsilon$ and $\Sigma M$, as well as the way the Dirac operator transforms under this isomorphism, which lead to formulas that are  more complicated than what described here, see e.g. \cite{Bo-Gau-Moroi} and \cite{Roos2018Thesis,Roos2018ManMath,Roos20}.
	
	Associated with the rescaling, we have:
	\begin{enumerate}
		\item The rescaled Killing vector has norm $\ell_\varepsilon = \ell\,\varepsilon $ and the length of the fibers is $2\pi\ell_\varepsilon$.
		\item We have \cite[Page 38]{Ammann99}
		\begin{eqnarray*}
			\Dirac_{M_\varepsilon}  = {1\over \ell_\varepsilon  } \Dirac_v + \Dirac_h +   V_\varepsilon \hbox{ on } L^2(\Sigma M) .\\
		\end{eqnarray*}
		with  
		
		\begin{equation*}V_\varepsilon\coloneqq -(1/4)\,\ell_\varepsilon \gamma(K/\ell_\varepsilon )\,\gamma(d\omega_\varepsilon) \hbox{ on } L^2(\Sigma M).
		\end{equation*}
		
		\item 	The Ammann and B\"ar collapsing condition is,  for $\varepsilon \to 0$  \cite[Equation (1)]{Ammann99}: 
		\begin{equation}\label{eq:Ammann coll cond} \ell_\varepsilon  \to 0, \quad \Vert \ell_\varepsilon d\omega_\varepsilon \Vert \to 0\quad \hbox{ for } \varepsilon \to 0.
		\end{equation}
	\end{enumerate}
	
	\begin{remark} In \cite{Roos2018ManMath} the Ammann and B\"ar collapsing  condition of Equation \eqref{eq:Ammann coll cond} is weakened to $ (\ell_\varepsilon d\omega_\varepsilon)  $ converging to a bounded operator.
	\end{remark}
	
	The main result of \cite{Ammann99, Ammann-Bar-98,  Roos2018ManMath,  Roos2018Thesis}  is the following theorem:

	\begin{theorem}(\cite{Ammann99, Ammann-Bar-98,  Roos2018ManMath,  Roos2018Thesis} )\label{collapstheo}
		Let $(M, \tilde{g})$ be a closed Riemannian spin manifold, and 
		let $U(1)$ act isometrically on $M$.
		We assume that the orbits have constant length $2\pi \ell$, which is equivalent to them being  totally geodesic. 
		Let $N = M/U(1)$ carry the induced Riemannian metric, which we will call $g$. Let $E \to N$ be a Hermitian vector bundle with a metric connection $\nabla^E$.  Let $\tilde g_\varepsilon$
		be the metric on $M$  obtained by shrinking $\tilde{g}$ in the vertical direction, with constant length of the fibers equal to $\ell_\varepsilon \coloneqq 2\pi \ell\varepsilon $.

		We suppose that the spin structure on $M$ is projectable and that 
		$N$ carries the induced
		spin structure.
		Let $\mu_1,\mu_2,\ldots$ be 
		the eigenvalues of the twisted Dirac operator $\Dirac_N^E$ on $L^2(\Sigma N) \otimes E $.
		
		Then we can number the eigenvalues 
		${(\lambda_{j,k}(\ell_\varepsilon  ))}_{j \in \N, k\in \Z}$ of the twisted Dirac
		operator ${\Dirac_M}_\varepsilon $ on $M$ for ${\tilde g}_\varepsilon $ on $L^2(\Sigma M_\varepsilon)\otimes \pi^*E $
		such that they depend continuously on $\ell_\varepsilon $ and such that for $\ell_\varepsilon  \to 0$:
		\begin{enumerate}[(1)]
			\item For any  $j\in\N$ and $k\in\Z$ 
			\begin{equation*}\ell_\varepsilon  \cdot \lambda_{j,k} (\ell_\varepsilon ) \to k. \end{equation*}
			In particular,  
			$\lambda_{j,k}(\ell_\varepsilon  ) \to \pm \infty$ if $k \not= 0$.
			\item If $n=\dim N$ is even, then 
			\begin{equation*}\lambda_{j,0}(\ell_\varepsilon ) \to \mu_j.\end{equation*}
			\item  If $n=\dim N$ is odd, then
			\begin{eqnarray*}
				\lambda_{2j-1,0}(\ell_\varepsilon  ) & \to & \phantom{-} \mu_j\cr
				\lambda_{2j,0}(\ell_\varepsilon  )   & \to &           - \mu_j
			\end{eqnarray*}
			In both cases, the convergence of the 
			eigenvalues $\lambda_{j,0}(\ell_\varepsilon  )$ is uniform in $j$.
		\end{enumerate} 
	\end{theorem}

	We  will now give a brief sketch of the proof of Theorem \eqref{collapstheo} in the  case when $n=\dim(N)$ is even.
	Let $\Psi$ be a common eigenspinor for $\mathcal{L}_K$ and $\Dirac_h$ for the 
	eigenvalues $ik$ and $\mu$ resp.
	
	On $U:=span\{ \Psi, \gamma(K/\ell_\varepsilon  ) \Psi\}$
	the operator $ (1/\ell_\varepsilon) \Dirac_v + \Dirac_h$ is represented by the matrix
	
	\begin{equation*}{1\over \ell\varepsilon}\begin{pmatrix}0& -ik\\ ik & 0\end{pmatrix} + \begin{pmatrix}\mu & 0 \\ 0 & -\mu\end{pmatrix}=
	\begin{pmatrix}\mu & -ik/\ell_\varepsilon  \\ ik/\ell_\varepsilon  & -\mu \end{pmatrix},
	\end{equation*}
	
	where $\mu$ are the eigenvalues of the Dirac operator on $N$.
	Thus for $k=0$ the restriction of $(1/\ell_\varepsilon) \Dirac_v + \Dirac_h$ has eigenvalues
	$\pm \mu$. For $k\neq 0$ the eigenvalues of $res|_U ({(1/\ell_\varepsilon) \Dirac_v + \Dirac_h})$ are 
	the square roots of $(k/\ell_\varepsilon )^2+\mu^2$.
	Therefore the eigenvalues $(\lambda_{j,k}^0(\ell_\varepsilon ))_{j\in\N,k \in \Z}$ of 
	$res|_U ({(1/\ell) \Dirac_v + \Dirac_h}) $ can be numbered such that they are continuous in $\ell_\varepsilon $ and 
	satisfy properties (1) and (2) of Theorem \ref{collapstheo}. 
	The additional term $\ell_\varepsilon  Z_\varepsilon $ does not change this behavior because
	tends to zero in norm for $\varepsilon \to 0$.
	

	\subsection{Convergence with Respect to the  Spectral Propinquity}
	\label{sub:case S 1 action Ammann Roos}

	We now prove that we have convergence with respect to the spectral propinquity convergence as $\varepsilon \to 0$. In particular the goal of this section  is to show Theorem \eqref{thm:basic- S1 case sp prop conv}, which will be proved applying  Theorem \eqref{main-thm}. For simplicity's sake we will consider the case $E =\C$. To reconcile the notation we are using here with the notation used in Theorem \eqref{main-thm}, define
	
	\begin{equation}\label{eq:reconcile}
	\A \coloneqq C(M),\quad \Hilbert\coloneqq L^2(\Sigma M), \quad  \hbox{ and }
	\quad \Dirac \coloneqq \Dirac_h + \Dirac_v,\quad \Dirac_\varepsilon \coloneqq \Dirac_h + \frac{1}{\varepsilon} \Dirac_v,
	\end{equation}
	where $\Dirac_v $ and $\Dirac_h $ are defined respectively in Equations \eqref{eq:defofvertDirac} and \eqref{def:k hor Dirac}.

	Of course we also have:
	
	\begin{equation}\label{eq:comparison D}
	\Dirac_{M_\varepsilon}  = \Dirac_\varepsilon + V_\varepsilon\quad \hbox{ on }L^2(\Sigma M) .
	\end{equation}

	\begin{theorem}\label{thm:basic- S1 case sp prop conv}
		Let $(M, \tilde{g})$ be a closed Riemannian spin manifold endowed with the structure of an principal $U(1)-$bundle over the quotient  manifold $N$, which can be assumed to be a Riemannian submersion over $(N,g)$ with fibers of constant length $2 \pi \ell$:
		\begin{equation}\label{def:Riemm sub structure}
		\pi: ( M, \tilde{g})\to ( N, g):
		\end{equation}
		Assume all of the  hypotheses of Theorem  \eqref {collapstheo};  in particular  we assume that we are in the smooth projectable case.  Let
		$(C(M),\Sigma M,\Dirac_M)$ be the standard metric spectral triple associated to the Dirac on $M$. 
		Fix $\varepsilon>0$,  and define $\ell_\varepsilon \coloneqq \ell \, \varepsilon$ and, with notation as above,  the operator
		
		\begin{equation}\Dirac_\varepsilon \coloneqq {1\over \ell_\varepsilon  } \Dirac_v + \Dirac_h\quad \hbox{ on } L^2(\Sigma M) .\end{equation}
		
		Then for all $\varepsilon >0$,  the   operator $\Dirac_\varepsilon $  is self-adjoint on $\Sigma M$ and
		the  spectral triple 
		
		\begin{equation}\label{def:epsilon sp-tr}
		(C(M),L^2(\Sigma M),\Dirac_\varepsilon)
		\end{equation}
		
		is  metric. Moreover
		
		\begin{equation}\label{eq:from main thm ammann}
		\lim_{\varepsilon\rightarrow 0} \spectralpropinquity{}\left((C(M),L^2(\Sigma M) ,\Dirac_{\varepsilon} ),(C(N),L^2(\Sigma N),\Dirac_N)\right) = 0 
		\end{equation}  
	\end{theorem}

	\begin{proof} As we already said, we will prove   Theorem \eqref{thm:basic- S1 case sp prop conv}  by applying  Theorem \eqref{main-thm}; see Equation \eqref{eq:reconcile} for the correspondence between our case and the situation in Theorem \eqref{main-thm}.

		Indeed we  will now check that the 
		hypotheses of Theorem \eqref{main-thm} are satisfied by checking  them item-by-item  as below. The precise statements to check are indeed:
			\begin{enumerate}
				\item[(0)] For all $\varepsilon >0$,  $\Dirac_{\varepsilon}  $ is self-adjoint and $0$ is isolated in $spec(\Dirac_\varepsilon)$.  (This is stronger than what required.)
				\item[(1)] The following norm inequalities hold, for all $a \in C(M)$ in the Lipschitz subalgebra of $\Dirac$:
				\begin{equation*}
				\max \left\{ \opnorm{[\frac{1}{\varepsilon}\Dirac_v, a]}{}{L^2(\Sigma M)}, \opnorm{[\Dirac_h, a]}{}{L^2(\Sigma M)}  \right\} \leq \opnorm{[\Dirac_\varepsilon, a]}{}{L^2(\Sigma M)}
				\end{equation*}
				\item[(2)] For all $b$ in the Lischitz subalgebra of $C(N)$, we have:
				\[
				[\Dirac_v, b]=0 .
				\]
				\item[(3)] If we let $p$ be the projection onto $\ker(\Dirac_v)$, then $[p,b] =0$ and $[\Dirac_h, p]=0$ for all $b\in C(N)$.
				\item[(4)] $(\B, \ker{\Dirac_v},p\Dirac_hp)$ is a metric spectral triple.
				\item[(5)]  There exists a positive linear map $\mathbb{E}: C(M) \to C(N)$ and a constant $k>0$ such that for all $a \in \A$ belonging to the Lipschitz algebra such that:
				
				\begin{equation*}
				\norm{a- \mathbb{E}(a)}{C(M)}\leq k\, \opnorm{[\Dirac_h, a]}{}{L^2(\Sigma M_\varepsilon)}
				\end{equation*}
				
				and
				
				\begin{equation*}
				\norm{p[\Dirac_h, \mathbb{E}(a)]p}{C(M)}= \opnorm{[\Dirac_h, \mathbb{E}(a)]}{}{L^2(\Sigma M)} \leq \opnorm{[\Dirac_v, a]}{}{L^2(\Sigma M)}.
				\end{equation*}
			\end{enumerate}

	 The    proof of the above points is given below.
	 
			\begin{enumerate}
				\item[(0)] The operator  $\Dirac_{\varepsilon}  $ is self-adjoint for all $\varepsilon >0$ since it is the sum of two self-adjoint operators, with one of them being bounded (see e.g.\cite{Mortad11} or \cite{LeMe}). Moreover,  $0$ is isolated in $\spectrum{\Dirac_{\varepsilon}}$ since all of its nonzero eigenvalues are given by 
				the square roots of $(k/\ell_\varepsilon )^2+\mu^2$ (where $\mu$ are the eigenvalues of the Dirac operator on $N$), as seen in the proof of Theorem \eqref{collapstheo}. Alternatively, the Dirac operator (which has compact resolvent) plus a bounded operator still has compact resolvent.

				\item[(1)] We now need to show the two inequalities in Theorem \eqref{main-thm}.  These will follow from Lemma \eqref{Fourier-estimate-lemma}. Indeed recall that
				on each of the eigenspaces $V_k$ (of Equation \eqref{eq:decomp}),  the Dirac operator on $M$ is given (up to the isometry $Q_k$) by the twisted Dirac $\Dirac_k$ operator of charge $k$ on  $V_k= L^{-k}\otimes \Sigma N$, given by:
				
				\begin{equation*}
				D'_{h,k} \coloneqq \sum_{i=1}^n (1_{L^{-k}} \otimes \gamma _i)\, (1 \otimes  \nabla^N_{f_i}+ k\nabla^\omega_{f_i} \otimes 1).
				\end{equation*} 
				An application of Proposition \eqref{comparison-lemma} ends the proof.
				\item[(2)] We need to show that  we have, for all $b\in C(N):$
				$[\Dirac_v, b]=0 .$
				This follows by explicitly computing the following expression (note that $b$ commutes with $Q_k$ for all $k$):
				\begin{equation*}
				[D'_{h,k},b] = [\sum_{i=1}^n (1_{L^{-k}} \otimes \gamma _i)\, (1 \otimes  \nabla^N_{f_i}+ k\nabla^\omega_{f_i} \otimes 1),b]=0.
				\end{equation*} 
				
				\item[(3)] If we let $p$ be the projection onto $\ker(\Dirac_v)$, then we need to show that: $[p,b] =0$  and $[\Dirac_h, p]=0$ for all $b\in C(N)$.	But, as in \cite[Equation (4.9)]{Orduz}):	 
				\begin{equation}\label{eq:splitting}
				\ker{\Dirac_v} = \{\psi|\mathcal{L}_{K/\ell}(\psi) =0 \}  = \Gamma(M, \Sigma M) \cong \pi^*( \Gamma(N, L^2(\Sigma N))) ,
				\end{equation} 
				which implies the wanted results.
				\item[(4)] This is the standard Dirac triple on $N$.
				\item[(5)]  Verified in the same way as in the proof of Theorem \eqref{thm: main conv result G-bundles}.
			\end{enumerate}
		\end{proof}
		
		So  Theorem \eqref{thm:basic- S1 case sp prop conv} is proven.


	\providecommand{\bysame}{\leavevmode\hbox to3em{\hrulefill}\thinspace}
	\providecommand{\MR}{\relax\ifhmode\unskip\space\fi MR }
	\providecommand{\MRhref}[2]{%
		\href{http://www.ams.org/mathscinet-getitem?mr=#1}{#2}
	}
	\providecommand{\href}[2]{#2}

	\vfill
\end{document}